\documentclass[12pt]{amsart}

\usepackage{amsfonts, amstext, amsmath, amsthm, amscd, amssymb, enumitem, url}
\usepackage{tikz-cd}
\usepackage{xcolor}
\usepackage{svg}
\usepackage{pdfpages}
\usepackage{pdflscape}

\usepackage[parfill]{parskip}
\usepackage{microtype}

\usepackage[colorlinks,citecolor=cyan,linkcolor=magenta]{hyperref}
\usepackage[nameinlink]{cleveref}
\usepackage{multirow}
\usepackage{subfig}

\newtheorem{thm}{Theorem}[section]

\newtheorem{lemma}[thm]{Lemma}
\newtheorem{prop}[thm]{Proposition}

\theoremstyle{definition}
\newtheorem{defn}[thm]{Definition}
\newtheorem{rmk}[thm]{Remark}
\newtheorem{constr}[thm]{Construction}
\newtheorem{quest}[thm]{Question}

\numberwithin{equation}{section}

\renewcommand{\epsilon}{\varepsilon}

\newcommand{\spn}{\mathrm{span}}

\newcommand{\rad}{\mathrm{rad}}
\newcommand{\cone}{\mathrm{cone}}
\newcommand{\real}{\mathrm{real}}
\newcommand{\infs}{\mathrm{inf}}
\newcommand{\vol}{\mathrm{vol}}

\allowdisplaybreaks

\begin{document}

\title[Standard embeddings and expansion factors]{Standardly embedded train tracks and pseudo-Anosov maps with minimum expansion factor}

\author{Eriko Hironaka}
\address{Department of Mathematics; Florida State University; Tallahassee, FL}
\email{hironaka@math.fsu.edu}
\thanks{The first author was partially supported by a grant from the Simons Foundation \#426722.}

\author{Chi Cheuk Tsang}
\address{University of California, Berkeley \\
    970 Evans Hall \#3840 \\
    Berkeley, CA 94720-3840}
\email{chicheuk@math.berkeley.edu}
\thanks{The second author was partially supported by a grant from the Simons Foundation \#376200.}

\maketitle

\begin{abstract}
We show that given a fully-punctured pseudo-Anosov map $f:S \to S$ whose punctures lie in at least two orbits under the action of $f$, the expansion factor $\lambda(f)$ satisfies the inequality
$$\lambda(f)^{|\chi(S)|} \ge \mu^4 \approx 6.85408,$$
where $\mu = \frac{1 + \sqrt{5}}{2} \approx 1.61803$ is the golden ratio. The proof involves a study of standardly embedded train tracks, and the Thurston symplectic form defined on their weight space.
\end{abstract}

\section{Introduction} \label{sec:intro}
Let $S = S_{g,s}$ be an oriented surface of finite type with genus $g$ and $s$ punctures, such that $\chi(S) = 2-2g-s$ is negative, and let $f:S \to S$ be a homeomorphism of $S$. By the Nielsen-Thurston classification of mapping classes, up to isotopy, $f$ either preserves an essential multicurve (and is \textit{periodic} or \textit{reducible}), or it is \textit{pseudo-Anosov}. In the latter case, there exists a transverse pair of measured, singular \textit{stable} and \textit{unstable} foliations $\ell^s,\ell^u$ such that $f$ stretches the measure of $\ell^u$ by $\lambda(f)$ and contracts the measure of $\ell^s$ by $\frac{1}{\lambda(f)}$. The number $\lambda(f)>1$ here is known as the \textit{expansion factor} of $f$.

Thurston's train track theory links the dynamics of pseudo-Anosov maps with that of Perron-Frobenius matrices, and implies that $\lambda(f)$ is a bi-Perron algebraic unit with degree bounded in terms of $\chi(S)$ \cite{FLP79}. In particular, for fixed $(g,s)$, $\lambda(f)$ attains a minimum $\lambda_{g,s} > 1$. Furthermore, each pseudo-Anosov map determines a closed geodesic on the moduli space $\mathcal M(S)$ with respect to the Teichm\"uller metric, whose length equals $\log(\lambda(f))$. Thus $\log(\lambda_{g,s})$ is the length of the shortest geodesics on $\mathcal M(S_{g,s})$ \cite{Abi80}. However, so far $\lambda_{g,s}$ is known only for small $g$ and $s$, see \cite{HS07}, \cite{CH08}, \cite{LT11b}.

In this paper, we focus on the normalized expansion factor 
$$L(S,f) = \lambda(f)^{|\chi(S)|},$$
and consider pseudo-Anosov maps that are \textit{fully-punctured}, meaning the singularities of $\ell^s$ and $\ell^u$ lie on the punctures of $S$. Let $\mu$ be the golden ratio $\frac{1 + \sqrt{5}}{2} \approx 1.61803$. Our main result is the following.

\begin{thm} \label{thm:mainthm}
Let $f:S \to S$ be a fully-punctured pseudo-Anosov map with at least two puncture orbits, then $L(S,f)$ satisfies the sharp inequality
$$L(S,f) \geq \mu^4 \approx 6.85408.$$
\end{thm}

Here by a puncture orbit, we mean an orbit of the action of $f$ on the punctures of $S$. If we replace the punctures of $S$ with boundary components, then the number of puncture orbits equals the number of boundary components of the mapping torus of $f$.
We remark that the assumption of $f$ having at least two puncture orbits is necessary. See \Cref{subsec:oneboundary} for explicit counterexamples. 

For the rest of this introduction, we will first give some background on what is known about the pattern of minimum expansion factors and its relations to the topology and geometry of the mapping torus. Then we will state a more precise version of \Cref{thm:mainthm}, discuss some special cases, and finally explain the tools that are used in the proof.

\subsection{Background on small expansion factors and mapping tori}

The expansion factor is a measure of the dynamical complexity of a pseudo-Anosov map. More concretely, it is equal to the exponential of the entropy, which measures how much the map mixes points on the surface. Thus, the minimum expansion factor is related to the complexity of the underlying surface and of the mapping torus, with the latter being a fibered hyperbolic 3-manifold \cite{Thu88}. In the following, we describe what is known about this relation.

\subsubsection{Minimum expansion factor as a function on the $(g,s)$ plane}

Applying properties of Perron-Frobenius matrices, Penner \cite{Pen91} observed that $\log\lambda_{g,s}$ is bounded below by the reciprocal of a linear function in $g$ and $s$
$$\log \lambda_{g,s} \geq \frac{\log 2}{12g-12+4s}.$$
Furthermore, for $s=0$, Penner constructed a sequence of examples giving an upper bound, yielding
$$\log \lambda_{g,0} \leq \frac{C}{g}$$
for some $C > 0$, which together with the lower bound implies
$$\log \lambda_{g,0} \asymp \frac{1}{|\chi(S_{g,0})|}.$$
This asymptotic behavior has been shown to persist along other lines in the $(g,s)$-plane, namely, $s = mg$ where $m > 0$ \cite{Val12}; $g = 0$ \cite{HK06}; $g = 1$ \cite{Tsa09}; and $s = s_0$ for fixed $s_0 > 0$ \cite{Yaz20}. For lines $g = g_0$ with fixed $g_0 \ge 2$, however, \cite{Tsa09} shows that
$$\log \lambda_{g,s} \asymp \frac{\log|\chi(S_{g,s})|}{|\chi(S_{g,s})|}.$$

We also mention a result of Agol, Leininger, and Margalit \cite{ALM16} that concerns the minimum expansion factor among pseudo-Anosov maps on an oriented closed surface with fixed genus $g$ whose action on the first homology preserves a $k$-dimensional subspace. We denote this minimum expansion factor as $\lambda_{g,0,k}$. Their result is that
$$\log \lambda_{g,0,k} \asymp \frac{k+1}{g}.$$
See \cite{BBKT22} for a partial generalization of this result to punctured surfaces. 

Let $\lambda_K$ be the minimum expansion factor for pseudo-Anosov maps $f:S \to S$ satisfying $|\chi(S)| = K$.

\begin{quest}[McMullen \cite{McM00}] \label{quest:limitKques}
Does the sequence $(\lambda_K)^K$ converge, and if so does it converge to $\mu^4$?
\end{quest}

As we will see in this paper, it is possible to use the theory of digraphs and standardly embedded train tracks to get information about minimum expansion factors in the fully-punctured case. To translate results from the fully-punctured case to the general case, the following question is of interest.

\begin{quest} \label{quest:levelK}
For fixed $K$, where, on the level sets in the $(g,s)$-plane where $2g+s-2=K$, is the minimum $\lambda_K$ realized? Is there a bound $s_0$, such that $\lambda_K$ is always achieved by $\lambda_{g,s}$ for $s \leq s_0$?
\end{quest}

A positive answer to \Cref{quest:levelK} would suggest that the minimum expansion factors $\lambda^\circ_K$ for fully-punctured pseudo-Anosov maps $f:S \to S$ satisfying $|\chi(S)| = K$ should have the same asymptotic behavior as $\lambda_K$.

We note that \Cref{thm:sharpthm} below answers \Cref{quest:levelK} in the affirmative if we restrict to fully-punctured pseudo-Anosov mapping classes with even Euler characteristic and at least two puncture orbits.

\subsubsection{Mapping tori}

Pseudo-Anosov mapping classes $(S,f)$ can be partitioned into flow-equivalence classes, that is, collections of pseudo-Anosov mapping classes whose mapping tori and suspension flows are equivalent. Thurston's fibered face theory gives a way to parameterize elements of a flow-equivalence class by primitive integral points on a cone over a top-dimensional face $F$, called a \textit{fibered face}, of the Thurston norm ball in $H^1(M;\mathbb{R})$. Here, the Thurston norm of a primitive integral element $a \in \cone (F)$ associated to $(S_a,f_a)$ equals $|\chi(S_a)|$, and hence its projection to $F$ is given by $\overline a = \frac{1}{|\chi(S_a)|}a$, see \cite{Thu86}.

It follows that the rational elements on $F$ are in one-to-one correspondence with elements of the flow equivalence class. By a result of Fried, the normalized expansion factor $L(S_a,f_a)$, considered as a function of integral classes $a$, extends to a continuous concave function defined for all $a \in F$ that goes to infinity towards the boundary of $F$ \cite{Fri82}. Thus, in particular, for any compact subset of $F$, the corresponding pseudo-Anosov maps have bounded normalized expansion factor. 

For a hyperbolic 3-manifold $M$ with fibered face $F$, let $C_{M,F}$ be the infimum of $L(S,f)$, where $(S,f)$ is associated to an element of $\cone(F)$. One consequence of Fried and Thurston's fibered face theory is the following.

\begin{thm} [Fried-Thurston \cite{Fri82} \cite{Fri85} \cite{Thu86}] \label{thm:FTthm}
If $b_1(M) \geq 2$, then for any $\epsilon > 0$ there are infinitely many fibrations $M \to S^1$ whose monodromy $(S,f)$ satisfies
$$L(S,f) < C_{M,F} + \epsilon.$$
\end{thm}

Given $C > 1$, the pseudo-Anosov mapping classes $(S,f)$ that satisfy $L(S,f) < C$ must have mapping tori $M$ and associated fibered face $F$ with $C_{M,F} \leq C$. The following result is often referred to as the universal finiteness property and states that after removing the singular orbits of the suspension pseudo-Anosov flow, the number of such pairs $(M,F)$ is finite.

\begin{thm}[Farb-Leininger-Margalit \cite{FLM11}] \label{thm:FLMthm}
Given any $C$, there is a finite set of fibered, hyperbolic 3-manifolds $\Omega_C$ such that if $(S,f)$ is a fully-punctured pseudo-Anosov map, and $L(S,f) < C$, then the mapping torus of $(S,f)$ lies in $\Omega_C$.
\end{thm}

(See also \cite{Ago11} and \cite{AT22}.)

The above theorem also implies finiteness for families of defining polynomials that can arise for small expansion factors. Given a polynomial $p(t)$, let $|p|$ be the complex norm of the largest root of $p(t)$. 

\begin{thm}[McMullen \cite{McM00}]
Let $M$ be a fibered hyperbolic 3-manifold, with $n = b_1(M) \geq 2$ and let $F \subset H^1(M;\mathbb{R})$ be a fibered face. Then there is a polynomial $\Theta \in \mathbb{Z}[t_1,\dots,t_n]$ with the property that for any primitive integral $a = (a_1,\dots,a_n)$ in the fibered cone over $F$, $|\Theta(t^{a_1},\dots,t^{a_n})|$ is the expansion factor of the monodromy of $(S,f)$ corresponding to $a$.
\end{thm}

The polynomial $\Theta$ is known as the {\it Teichm\"uller polynomial} of the fibered face.

Define $C_M$ to be the infimum of $L(S,f)$, where $(S,f)$ is the monodromy of a fibration $M \to S^1$. The topological invariant $C_M$ of a fibered hyperbolic 3-manifold $M$ is related to the geometry of $M$ as shown by the following theorem.

\begin{thm}[Kojima-McShane \cite{KM18}] \label{thm:KMthm} 
For any hyperbolic, fibered $3$-manifold $M$, $C_M$ satisfies the following inequality 
$$C_M \geq \exp \left(\frac{\vol(M)}{3\pi} \right).$$
\end{thm}

\begin{rmk}
\Cref{thm:KMthm} implies, for example, that if $(S,f)$ satisfies $L(S,f) = \mu^4$, the volume of the mapping torus $M$ must satisfy
$$\vol(M) \leq 12\pi \log (\mu) \approx 18.14123.$$
For comparison, the magic manifold, which realizes many of the smallest known expansion factors, has volume $\approx 5.33349$ according to SnapPy \cite{SnapPy}. Further work comparing known pseudo-Anosov maps with small expansion factor and the volume of their mapping tori can be found in \cite{AD10} and \cite{KKT13}.
\end{rmk}

\subsection{Refinement of the main theorem and some special cases}

For positive integers $a,b$ with $a < b$, define
$$LT_{a,b}(t) = t^{2b} - t^{b+a} - t^b - t^{b-a} + 1,$$
This family of polynomials was first noticed by Lanneau and Thiffeault \cite{LT11a} to play a role in the study of minimum expansion factors.

The more precise version of our main theorem is as follows.
 
\begin{thm} \label{thm:sharpthm}
Let $f:S \to S$ be a fully-punctured pseudo-Anosov map with at least two puncture orbits where $|\chi(S)|=K \geq 3$. Then $\lambda(f)$ satisfies the inequality
$$\lambda(f) \geq |LT_{1,\frac{K}{2}}|,$$
for $K$ even, and
$$\lambda(f)^K \geq 8,$$
for $K$ odd.

Moreover, for each even $K$, equality is achieved by a fully-punctured pseudo-Anosov map with $|\chi(S)|=K$ and $s \leq 4$.
\end{thm}

Sharpness in \Cref{thm:sharpthm} follows from computations in \cite{Hir10}, \cite{AD10}, \cite{KKT13}.

\subsubsection{Orientable pseudo-Anosov mapping classes}

A pseudo-Anosov mapping class $(S,f)$ is called \textit{orientable} if its stable and unstable foliations are orientable. These have the property that the expansion factor of $f$ is equal to the largest eigenvalue of its action on the integral homology of $S$.

In \cite{LT11a}, Lanneau and Thiffeault asked whether the minimum expansion factor $\lambda_g^{+}$ for even genus $g$ orientable pseudo-Anosov mapping classes is given by
$$\lambda_g^{+} = |LT_{1,g}|,$$
and showed that $|LT_{1,g}|$ is a lower bound for $\lambda_g^+$ for $g = 2,4,8,10$. For the case when $g \geq 2$ is even and not divisible by 3, $|LT_{1,g}|$ is an upper bound for $\lambda_g^+$ \cite{Hir10}.

\begin{thm} 
Let $(S,f)$ be an orientable pseudo-Anosov mapping class on a genus $g \geq 2$ closed surface with exactly two singularities, each of which is fixed by $f$. Then $\lambda(f)$ satisfies the inequality
$$\lambda(f) \geq |LT_{1,g}|$$
and this inequality is sharp when $g$ is even and not divisible by 3.
\end{thm}

\subsubsection{Braid Monodromies}
Consider $(S,f)$ where $S = S_{0,n+1}$, and one puncture $p_\infty$ is fixed by $f$. A mapping class of this type is called a \textit{braid monodromy on $n$ strands}. Given a standard braid $\beta$ on $n$ strands, there is an associated braid monodromy, and vice versa (see e.g. \cite{Bir74}). For convenience, we use the standard Artin braid group generators to specify a braid and its associated monodromy.

\begin{thm} \label{thm:braids}
Let $\beta$ be a fully-punctured pseudo-Anosov braid monodromy on $n$ strands. Then $\lambda(\beta)$ satisfies the inequality
$$\lambda(\beta)^{n-1} \geq \mu^4.$$
More precisely, for $n \geq 4$, $\lambda(\beta)$ satisfies the inequality
$$\lambda(\beta) \geq |LT_{1,\frac{n-1}{2}}|$$
for $n$ odd, and 
$$\lambda(\beta)^{n-1} \geq 8$$
for $n$ even.
\end{thm}

The 3-strand braid $\beta = \sigma_1\sigma_2^{-1}$ satisfies $\lambda(\beta) = \mu^2$ (see \cite{Hir10}), while the 5-strand braid $\beta = \sigma_1 \sigma_2 \sigma_3 \sigma_4 \sigma_1 \sigma_2$ satisfies $\lambda(\beta) = |LT_{1,2}|$ (see \cite{HS07}). Since both of these braid monodromies are fully-punctured, this shows that \Cref{thm:braids} is sharp. For larger number of strands, however, there are no known examples achieving the lower bound. Instead, the smallest known expansion factors of braid monodromies follow the pattern found in \cite{HK06} (see also \cite{KKT13}). This invites the following question.

\begin{quest} \label{quest:braidsharpness}
Is it true that aside from $\sigma_1\sigma_2^{-1}$ and $\sigma_1 \sigma_2 \sigma_3 \sigma_4 \sigma_1 \sigma_2$, the expansion factor of a fully-punctured pseudo-Anosov braid monodromy on $n$ strands is strictly larger than $|LT_{1,\frac{n-1}{2}}|$?
\end{quest}

A positive answer to \Cref{quest:levelK} would suggest a positive answer to \Cref{quest:braidsharpness} as well. This is because the former would imply that the smallest values of $\lambda_{g,s}$ are concentrated in a vertical band $s \leq s_0$ on the $(g,s)$-plane, whereas the braid monodromies concern the horizontal line $g=0$ which moves away from this vertical band.

We remark that the answer to \Cref{quest:braidsharpness} is `no' if the word `fully-punctured' is dropped: It has been pointed out to us by Eiko Kin that the 5-strand braid $\beta = (\sigma_1 \sigma_2 \sigma_3 \sigma_4)^2 \sigma_4 \sigma_3$ satisfies $\lambda(\beta) = |LT_{1,2}|$ as well (but it is not fully-punctured).

\subsection{Techniques used in the proof}
Our proof of \Cref{thm:mainthm} involves a study of properties of Perron-Frobenius directed graphs and of standardly embedded train tracks.
 
\subsubsection{Perron-Frobenius digraphs}

To any non-negative square integer matrix $M \geq 0$, one can associate a directed graph, or \textit{digraph}, $\Gamma$ where each vertex corresponds to a row (or column) of $M$, and the number of directed edges counted with multiplicity from one vertex to another equals the corresponding $(\mathrm{row},\mathrm{column})$ entry of $M$. The matrix $M$ is \textit{Perron-Frobenius} if it has the property that $M^n > 0$ for large enough $n$. This translates to the condition that $\Gamma$ is \textit{strongly connected}, i.e. any ordered pair of vertices can be connected by a directed edge path, and \textit{aperiodic}, i.e. the greatest common divisor of the lengths of directed cycles is $1$.

The characteristic polynomial $\theta_\Gamma$ of $\Gamma$ is defined to be the characteristic polynomial of $M$. We say that a polynomial is \textit{reciprocal} if the set of roots (counted with multiplicities) is closed under inverses. This is equivalent to the polynomial being palindromic, that is, the list of its coefficients forms a palindrome, up to a sign.

\begin{thm}[McMullen \cite{McM15}] \label{thm:McMthm}
Let $\Gamma$ be a Perron-Frobenius digraph of rank $K \geq 3$ with reciprocal characteristic polynomial $\theta_\Gamma$. Then 
$$|\theta_\Gamma| \ge |LT_{1,\frac{K}{2}}|$$
for $K$ even, and
$$|\theta_\Gamma|^K \ge 8.$$
for $K$ odd.
\end{thm}

\subsubsection{Expansion factors and train tracks}

Recall that to compute expansion factors of pseudo-Anosov maps $f:S \to S$ one uses train tracks to translate the geometric dynamics of $f$ to combinatorial information. In their most general form train tracks are embedded graphs on surfaces with smoothings at vertices indicating directions which paths are allowed to take. A foliation on a surface is said to be \textit{carried} by a train track $\tau$ if any leaf can be isotoped to a smooth path on $\tau$. Thurston showed that if $f:S \to S$ is pseudo-Anosov, then there is a train track $\tau$ embedded on $S$ which carries the unstable foliation of $f$, and for which the map $f$ determines a graph map on $\tau$ which sends vertices to vertices and edges to edge paths. Thus, $f$ naturally induces a linear map 
$$f_* : \mathbb{R}^\mathcal{E} \to \mathbb{R}^\mathcal{E}$$
where $\mathcal E$ is the set of edges of $\tau$ and $\mathbb{R}^\mathcal{E}$ can be thought of as choices of weights defined on the edges. Under an appropriate choice of $\tau$, $f_*$ is Perron-Frobenius and the spectral radius of $f_*$ is the expansion factor of $f$, see for example \cite{FLP79}.

One might be tempted to apply \Cref{thm:McMthm} to the digraph associated to this $f_*$. However, there are some problems with this. Firstly, notice that one also needs to show that $f_*$ has reciprocal characteristic polynomial. A natural idea for doing so is to make use of the \textit{Thurston symplectic form} $\omega$. However, $\omega$ is actually degenerate on $\mathbb{R}^\mathcal{E}$, so reciprocity is not immediate here. Even if one can show reciprocity, the dimension of $\mathbb R^{\mathcal E}$ is rather large, so the bound obtained from \Cref{thm:McMthm} would not be very good.

One can instead restrict $f_*$ to the subspace of \textit{allowable weights} $\mathcal W \subset {\mathbb R}^{\mathcal E}$, consisting of weights that satisfy what are called the branching conditions. This is the minimal condition required so that any simple closed loop which is carried by $\tau$ will define a positive vector in $\mathcal W$. This would reduce the dimension and remove a large part of the degeneracy of $\omega$, but it is now unclear whether the action of $f_*$ on $\mathcal{W}$ is Perron-Frobenius or not. Indeed, there is not even a natural basis of $\mathcal{W}$ that we can write the action of $f_*$ as a matrix in, making it difficult to apply digraph techniques directly.

These problems are avoided by considering only pseudo-Anosov maps that are carried by a special kind of train track based on the Bestvina-Handel algorithm \cite{BH95} \cite{Bri00}, which we define next.

\subsubsection{Standardly embedded train tracks}

\begin{defn}
Let $f:S \to S$ be a fully-punctured pseudo-Anosov map. A train track $\tau$ that carries $f$ is \textit{standardly embedded} if its edges can be partitioned into two types: $\mathcal E_{\real}$, the set of \textit{real edges} and $\mathcal E_{\infs}$, the set of \textit{infinitesimal edges}, such that:
\begin{enumerate}
\item each connected component of the complement of $\tau$ is homeomorphic to a once-punctured disk,
\item the action of $f$ permutes the set of infinitesimal edges, and
\item the linear map on $\mathbb{R}^{\mathcal{E}_{\real}}$ induced by $f$,
$$f_*^{\real}: \mathbb{R}^{\mathcal{E}_{\real}} \to \mathbb{R}^{\mathcal{E}_{\real}}$$ 
is Perron-Frobenius with respect to the standard basis.
\end{enumerate}
\end{defn}

For computations and properties of standardly embedded train tracks see, for example, \cite{KLS02}, \cite{HK06}, \cite{LV17}, \cite{FRW22}.

It follows that given a pseudo-Anosov map $f$ with a standardly embedded train track $\tau$ that carries it, the expansion factor $\lambda(f)$ is the spectral radius of the restriction $f_*^{\real}$ of $f_*$ to $\mathbb{R}^{\mathcal{E}_{\real}}$. Furthermore, the degree of the characteristic polynomial of $f_*^{\real}$ equals $|\chi(S)|$.

Our main technical result is the following.

\begin{thm} \label{thm:standardlyembeddedtt}
If $f:S \to S$ is a fully-punctured pseudo-Anosov map with at least two puncture orbits, then there exists a standardly embedded train track $\tau$ such that
\begin{enumerate}[label=(\roman*)]
\item $\tau$ carries $f$, and
\item the characteristic polynomial of $f_*^{\real}$ is reciprocal.
\end{enumerate}
\end{thm}

The proof of (ii) uses an explicit description of the radical of the Thurston symplectic form on the weight space of $\tau$. We refer to \Cref{prop:radical} for the technical statement used in our proof. Later in \Cref{prop:generalradical} we give an extension that applies to general train tracks carrying a pseudo-Anosov map.

Our main theorem is then obtained by applying \Cref{thm:McMthm} to the digraph associated to $f_*^{\real}$ for such a standardly embedded train track.

\subsection{Organization of paper} We set up some basic definitions in \Cref{sec:background}. In \Cref{sec:standardlyemb} we define standardly embedded train tracks and prove \Cref{thm:standardlyembeddedtt}. We note that unlike in the introduction, we will define standardly embedded train tracks as a class of objects separate from surfaces and pseudo-Anosov maps. This will make the discussion easier in \Cref{sec:thurstonform} and \Cref{sec:radical}, where we prove that the characteristic polynomial is reciprocal by studying the Thurston symplectic form.

We prove our main theorem in \Cref{sec:mainthm}. In \Cref{sec:sharpness} we explore the sharpness of the main theorem, in particular demonstrating examples that show the last statement of \Cref{thm:sharpthm}. Finally, in \Cref{sec:questions}, we discuss some questions and future directions.

\subsection*{Acknowledgements} We would like to thank Ian Agol and Curtis McMullen for their support and encouragement throughout this project. We would like to thank Luya Wang, Anna Parlak, Saul Schleimer, and Henry Segerman for helpful conversations. We would like to thank Eiko Kin and Livio Liechti for their comments on an earlier version of this paper. This work was done while the first author was a participant of the Complex Dynamics program at the MSRI and later while she served at the NSF.

\section{Background} \label{sec:background}

\subsection{Pseudo-Anosov maps}

We recall the definition of a pseudo-Anosov map. More details can be found in \cite{FLP79}.

\begin{defn} \label{defn:pamap}
A \textit{finite-type surface} is an oriented closed surface with finitely many points, which we call the \textit{punctures}, removed. 

A homeomorphism $f$ on a finite-type surface $S$ is said to be \textit{pseudo-Anosov} if there exists a pair of singular measured foliations $(\ell^s,\mu^s)$ and $(\ell^u,\mu^u)$ such that:
\begin{enumerate}
    \item Away from a finite collection of \textit{singular points}, which includes the punctures, $\ell^s$ and $\ell^u$ are locally conjugate to the foliations of $\mathbb{R}^2$ by vertical and horizontal lines respectively.
    \item Near a singular point, $\ell^s$ and $\ell^u$ are locally conjugate to either
    \begin{itemize}
        \item the pull back of the foliations of $\mathbb{R}^2$ by vertical and horizontal lines by the map $z \mapsto z^{\frac{n}{2}}$ respectively, for some $n \geq 3$, or
        \item the pull back of the foliations of $\mathbb{R}^2 \backslash \{(0,0)\}$ by vertical and horizontal lines by the map $z \mapsto z^{\frac{n}{2}}$ respectively, for some $n \geq 1$.
    \end{itemize}
    In this case, we say that the singular point is \textit{$n$-pronged}.
    \item $f_* (\ell^s,\mu^s) = (\ell^s,\lambda^{-1} \mu^s)$ and $f_* (\ell^u,\mu^u) = (\ell^u,\lambda \mu^u)$ for some $\lambda=\lambda(f)>1$.
\end{enumerate} 
We call $(\ell^s,\mu^s)$ and $(\ell^u,\mu^u)$ the \textit{stable} and \textit{unstable} measured foliations respectively. We call $\lambda(f)$ the \textit{expansion factor} of $f$ and call $\lambda(f)^{|\chi(S)|}$ the \textit{normalized expansion factor} of $f$.
\end{defn}

\begin{defn} \label{defn:fullypunctured}
$f$ is said to be \textit{fully-punctured} if the set of singular points equals the set of punctures. In this case, we will denote the set of punctures by $\mathcal{X}$. Note that $f$ acts on $\mathcal{X}$ by some permutation, hence it makes sense to talk about the orbits of punctures under the action of $f$, or \textit{puncture orbits} for short.
\end{defn}

We recall two standard facts about pseudo-Anosov maps which will be used in \Cref{sec:standardlyemb}.

\begin{defn} \label{defn:halfleaf}
A \textit{half-leaf} of $\ell^s$ is a properly embedded copy of $[0,\infty)$ in a leaf of $\ell^s$. A \textit{half-leaf} of $\ell^u$ is similarly defined.
\end{defn}

\begin{prop} \label{prop:halfleafdense}
We have the following properties of $f$, $\ell^s$, and $\ell^u$.
\begin{itemize}
    \item The set of periodic points of $f$ is dense.
    \item Every half-leaf of $\ell^s$ and of $\ell^u$ is dense.
\end{itemize}
\end{prop}
\begin{proof}
This is Proposition 9.20 and Proposition 9.6 of \cite{FLP79} respectively.
\end{proof}

\subsection{Train tracks} \label{subsec:traintrack}

We define train tracks and set up some terminology.

\begin{defn} \label{defn:traintrack}
A \textit{(ribbon) train track} $\tau$ is a finite graph endowed with two additional pieces of data:
\begin{enumerate}
    \item A cyclic ordering on the set $\mathcal{E}_v$ of half-edges incident to each vertex $v$
    \item A partition $\mathcal{E}_v = \mathcal{E}^1_v \sqcup \mathcal{E}^2_v$ into two nonempty cyclically consecutive subsets, which we refer to as the \textit{smoothing} at $v$.
\end{enumerate}

We denote the set of vertices of $\tau$ by $\mathcal{V}(\tau)$ and denote the set of edges of $\tau$ by $\mathcal{E}(\tau)$. We will sometimes refer to the vertices of $\tau$ as the \textit{switches}. 

We will think of the two pieces of additional data on each $\mathcal{E}_v$ as represented by having a small oriented disc at $v$ with the half-edges being arranged around $v$ as determined by the cyclic order, and with a tangent line at $v$ such that the half edges in $\mathcal{E}^1_v$ are tangent to one side of the tangent line and those in $\mathcal{E}^2_v$ tangent to the other side. See \Cref{fig:traintrack} left. 

\begin{figure}
    \centering
    \fontsize{14pt}{14pt}\selectfont
    \resizebox{!}{3cm}{
\begingroup%
  \makeatletter%
  \providecommand\color[2][]{%
    \errmessage{(Inkscape) Color is used for the text in Inkscape, but the package 'color.sty' is not loaded}%
    \renewcommand\color[2][]{}%
  }%
  \providecommand\transparent[1]{%
    \errmessage{(Inkscape) Transparency is used (non-zero) for the text in Inkscape, but the package 'transparent.sty' is not loaded}%
    \renewcommand\transparent[1]{}%
  }%
  \providecommand\rotatebox[2]{#2}%
  \newcommand*\fsize{\dimexpr\f@size pt\relax}%
  \newcommand*\lineheight[1]{\fontsize{\fsize}{#1\fsize}\selectfont}%
  \ifx\svgwidth\undefined%
    \setlength{\unitlength}{392.77548887bp}%
    \ifx\svgscale\undefined%
      \relax%
    \else%
      \setlength{\unitlength}{\unitlength * \real{\svgscale}}%
    \fi%
  \else%
    \setlength{\unitlength}{\svgwidth}%
  \fi%
  \global\let\svgwidth\undefined%
  \global\let\svgscale\undefined%
  \makeatother%
  \begin{picture}(1,0.2365991)%
    \lineheight{1}%
    \setlength\tabcolsep{0pt}%
    \put(0,0){\includegraphics[width=\unitlength,page=1]{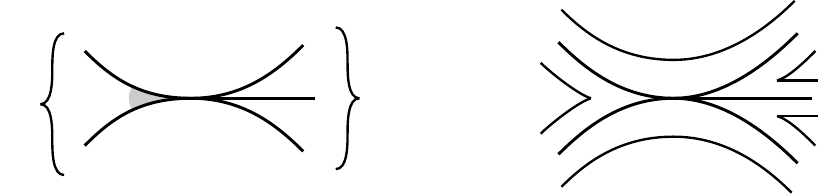}}%
    \put(0.22541343,0.12684041){\color[rgb]{0,0,0}\makebox(0,0)[lt]{\lineheight{1.25}\smash{\begin{tabular}[t]{l}\small $v$\end{tabular}}}}%
    \put(0.31779002,0.16476233){\color[rgb]{0,0,0}\makebox(0,0)[lt]{\lineheight{1.25}\smash{\begin{tabular}[t]{l}\small $e'$\end{tabular}}}}%
    \put(0.31779002,0.0573499){\color[rgb]{0,0,0}\makebox(0,0)[lt]{\lineheight{1.25}\smash{\begin{tabular}[t]{l}\small $e$\end{tabular}}}}%
    \put(0.45354567,0.09889384){\color[rgb]{0,0,0}\makebox(0,0)[lt]{\lineheight{1.25}\smash{\begin{tabular}[t]{l}$\mathcal{E}^2_v$\end{tabular}}}}%
    \put(-0.00346841,0.09889384){\color[rgb]{0,0,0}\makebox(0,0)[lt]{\lineheight{1.25}\smash{\begin{tabular}[t]{l}$\mathcal{E}^1_v$\end{tabular}}}}%
    \put(0,0){\includegraphics[width=\unitlength,page=2]{traintrack.pdf}}%
  \end{picture}%
\endgroup%
}
    \caption{A train track and its tie neighborhood near a switch.}
    \label{fig:traintrack}
\end{figure}

In fact, one can form an oriented surface $N$ by taking the union of these small oriented discs and a rectangular strip connecting the discs, respecting their orientations, along each edge. $N$ can be foliated by properly embedded intervals transverse to the edges of $\tau$. We call $N$ the \textit{tie neighborhood} of $\tau$, and call the intervals foliating $N$ the \textit{ties}. See \Cref{fig:traintrack} right.

We refer to a pair of adjacent elements in some $\mathcal{E}^\beta_v$ as a \textit{cusp} at $v$. Equivalently, consider the complementary regions of $\tau$ in its tie neighborhood. The cusps are in one-to-one correspondence with the non-smooth points on the boundary that meet $\tau$.

In general, let $e$ and $e'$ be two elements in $\mathcal{E}^\beta_v$ for some switch $v$ and some $\beta \in \mathbb{Z}/2$. $e$ is said to \textit{lie on the left of} $e'$ if for some (hence any) $f \in \mathcal{E}^{\beta+1}_v$, $(e,e',f)$ is oriented counterclockwise in $\mathcal{E}_v$. We show one example of this in \Cref{fig:traintrack} left. 
\end{defn}

We make two remarks regarding this definition.

Firstly, train tracks are frequently defined to be subsets of surfaces in the literature. However, it will be to our advantage to define them independently of a fixed surface. To connect with the usual definition, we have to consider embeddings of the tie neighborhoods into surfaces. The following definition illustrates an example of this.

\begin{defn} \label{defn:carry}
Let $l$ be a lamination on finite type surface $S$. Let $\tau$ be a train track with tie neighborhood $N$, and let $\iota:\tau \hookrightarrow S$ be an embedding. $\iota(\tau)$ is said to \textit{carry} $l$ if $\iota$ can be extended into an orientation preserving embedding $\iota:N \hookrightarrow S$ such that the leaves of $l$ are contained in $\iota(N)$ and are transverse to the ties. $\iota(\tau)$ is said to \textit{fully carry} $l$ if in addition every tie intersects some leaf of $l$.
\end{defn}

Secondly, the underlying graph of a train track is sometimes required to be trivalent. Here we allow our train tracks to have arbitrary valence since this is the natural setting for discussing standardly embedded train tracks.

We then define maps of train tracks.

\begin{defn} \label{defn:traintrackmap}
An edge path or cycle $c$ of $\tau$ is said to be \textit{carried by $\tau$} if at each vertex $v$ on $c$, the incoming and outgoing edges do not lie in the same $\mathcal{E}^\beta_v$. 

A map $f:\tau \to \tau'$ is said to be a \textit{train track map} if $f$ sends switches of $\tau$ to switches of $\tau'$ and edges of $\tau$ to edge paths carried by $\tau'$, such that the induced map $D_v f: \mathcal{E}_v \to \mathcal{E}_{f(v)}$ of half-edges at each switch preserves the cyclic ordering and sends elements in distinct $\mathcal{E}^\beta_v$ to distinct $\mathcal{E}^{\beta'}_{f(v)}$. 
\end{defn}

Suppose $f:\tau \to \tau'$ is a train track map. Let $N$ and $N'$ be the tie neighborhoods of $\tau$ and $\tau'$ respectively. Then $f$ induces an embedding $N \hookrightarrow N'$ which sends ties in $N$ into ties in $N'$. See \Cref{fig:traintrackmap}. For convenience of notation, we will denote this induced embedding by $f$ as well.

\begin{figure}
    \centering
    \resizebox{!}{4cm}{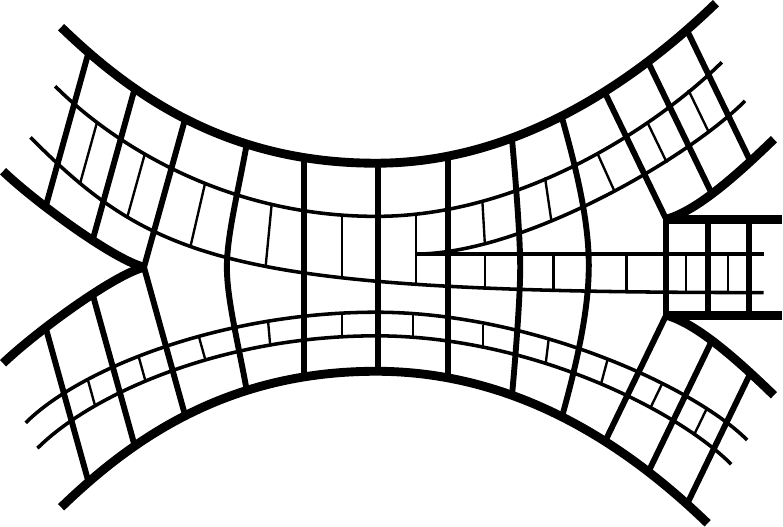}
    \caption{A train track map $\tau \to \tau'$ induces an embedding of tie neighborhoods $N \hookrightarrow N'$.}
    \label{fig:traintrackmap}
\end{figure}

\begin{defn} \label{defn:transitionmatrix}
Let $f: \tau \to  \tau'$ be a train track map. The \textit{transition matrix} of $f$ is the matrix $f_* \in M_{\mathcal{E}(\tau') \times \mathcal{E}(\tau)}(\mathbb{Z})$  whose entries are defined by 
$$(f_*)_{e',e}=\text{\# times $f(e)$ passes through $e'$}.$$ 
\end{defn}

\subsection{Reciprocal Perron-Frobenius matrices} \label{subsec:PFmatrix}

We recall some terminology regarding matrices.

\begin{defn} \label{defn:PF}
A matrix $B$ is said to be \textit{positive} if all of its entries are positive.

Let $A \in M_{n \times n}(\mathbb{Z}^{\geq 0})$ be a matrix with nonnegative integer entries. $A$ is said to be \textit{Perron-Frobenius} if there exists $k \geq 1$ such that $A^k$ is positive. 
\end{defn}

Perron-Frobenius matrices arise naturally in the study of pseudo-Anosov dynamics, and are frequently applied with the following classical theorem.

\begin{thm}[Perron-Frobenius theorem] \label{thm:PF}
Let $A$ be a Perron-Frobenius matrix. Let $\lambda$ be the largest real eigenvalue of $A$. Then $\lambda$ is equal to the spectral radius of $A$ and there exists a positive $\lambda$-eigenvector $v$ of $A$. Moreover, up to scalar multiplication by a positive number, $v$ is the unique positive eigenvector of $A$.
\end{thm}

Another property of matrices which we will use in this paper is reciprocity. 

\begin{defn} \label{defn:reciprocal}
A linear map $T:V \to V$ is said to be \textit{reciprocal} if its eigenvalues (taken with multiplicity) are invariant under $\mu \mapsto \mu^{-1}$. Equivalently, $T$ is reciprocal if and only if its characteristic polynomial $p(t)$ is reciprocal, that is, it satisfies $p(t)=\pm t^n p(t^{-1})$ where $n=\dim V$.
\end{defn}

We remark that even though we will freely regard matrices as linear transformations on Euclidean spaces using the standard bases, the reader should observe that reciprocity is a property of \emph{linear transformations}, whereas being Perron-Frobenius is strictly a property of \emph{matrices}. 

For matrices that are both Perron-Frobenius and reciprocal, McMullen proved the following sharp inequality regarding their spectral radius.

\begin{thm}[{\cite{McM15}}] \label{thm:McMullen}
Let $A \in M_{n \times n}(\mathbb{Z}^{\geq 0})$, $n \geq 2$, be a reciprocal Perron-Frobenius matrix with spectral radius $\rho(A)$. Then 
$$\rho(A)^n \geq \mu^4.$$

More precisely, for $n \geq 3$, we have
$$\rho(A) \geq \left| LT_{1,\frac{n}{2}} \right|$$
if $n$ is even, and
$$\rho(A)^n \geq 8$$
if $n$ is odd.
\end{thm}
\begin{proof}
In \cite{McM15}, the first statement is stated as Theorem 7.1. The second statement is stated as Theorem 1.1 for $n$ even, and stated in P.33 under the proof of Theorem 7.1 for $n$ odd.
\end{proof}

To conclude this subsection, we state some elementary facts that will help us establish reciprocity.

\begin{prop} \label{prop:reciprocalfacts}
We have the following examples and properties of reciprocal matrices.
\begin{enumerate}
    \item A symplectic matrix is reciprocal.
    \item Suppose $P \in M_{n \times n}(\mathbb{Z})$ is a signed permutation matrix and $V \subset \mathbb{R}^n$ is a subspace invariant under $P$, then $P|_V:V \to V$ is reciprocal.
    \item Suppose matrix $A$ admits a block decomposition 
    $\begin{bmatrix} B & * \\ 0 & C \end{bmatrix}$. Then two of $A$, $B$, or $C$ being reciprocal implies that the remaining one is reciprocal as well.
\end{enumerate}
\end{prop}
\begin{proof}
In (1), a \textit{symplectic matrix} is a matrix $A \in M_{n \times n}(\mathbb{R})$ satisfying $\omega(Av,Av')=\omega(v,v')$ for all $v,v' \in \mathbb{R}^n$, for some symplectic form $\omega$. It is an elementary fact that these are reciprocal, see for example the proof of \cite[Theorem 2.1]{GM02}.

In (2), by a \textit{signed permutation matrix}, we mean that the only nonzero entries of $P$ are $P_{i,\sigma(i)}=\pm 1$ for some permutation $\sigma \in S_n$. Such matrices are orthogonal (with respect to the standard inner product on $\mathbb{R}^n$), hence the restriction $P|_V$ is also orthogonal, and it is another elementary fact that orthogonal maps are reciprocal (again see \cite[Theorem 2.1]{GM02}).

The statement in (3) follows from the observation that the set of eigenvalues of $A$ is the union of that of $B$ and $C$ (taken with multiplicity).
\end{proof}

\section{Standardly embedded train tracks} \label{sec:standardlyemb}

For the rest of this paper, we fix the following setting: Let $f : S \to S$ be a fully punctured pseudo-Anosov map. Let $(\ell^s, \mu^s)$ and $(\ell^u, \mu^u)$ be the stable and unstable measured foliations of $f$ respectively. Let $\lambda$ be the expansion factor of $f$.

In this section we will define standardly embedded train tracks and explain how they are related to certain Markov partitions that encode the dynamics of $f$.

\subsection{Definition of standardly embedded train tracks} \label{subsec:standardlyemb}

We first define the notion of standardly embedded train tracks. The definition we use is adapted from \cite{CH08}.

Let $\tau$ be a train track with tie neighborhood $N$. Each complementary region of $\tau$ in $N$ is homeomorphic to an annulus, with one boundary component along $\partial N$ and the other boundary component along $\tau$. We call the boundary component along $\tau$ a \textit{boundary component of $\tau$}. Under this terminology, the boundary components of $\tau$ are in one-to-one correspondence with the boundary components of $N$. We denote the collection of boundary components of $\tau$ by $\partial \tau$.

We say that a boundary component $c$ of $\tau$ is $n$-pronged if it contains $n$ cusps. We say that $c$ is odd/even-pronged if it is $n$-pronged for odd/even $n$ respectively.

\begin{defn} \label{defn:standardlyembedded}
Let $\partial \tau = \partial_I \tau \sqcup \partial_O \tau$ be a partition of the boundary components of $\tau$ into a nonempty set of \textit{inner boundary components} and a nonempty set of \textit{outer boundary components} respectively.

A train track $\tau$ is said to be \textit{standardly embedded with respect to $(\partial_I \tau, \partial_O \tau)$} if its set of edges $\mathcal{E}$ can be partitioned into a set of \textit{infinitesimal edges} $\mathcal{E}_{\infs}$ and a set of \textit{real edges} $\mathcal{E}_{\real}$, such that:
\begin{itemize}
    \item The smoothing at each vertex $v$ is defined by separating the infinitesimal edges and the real edges.
    \item The union of infinitesimal edges is a disjoint union of cycles, which we call the \textit{infinitesimal polygons}.
    \item The infinitesimal polygons are exactly the inner boundary components of $\tau$.
\end{itemize}

We show one example of a standardly embedded train track in \Cref{fig:standardlyemb}, where the infintesimal polygons are drawn in gray.

\begin{figure}
    \centering
    \resizebox{!}{5cm}{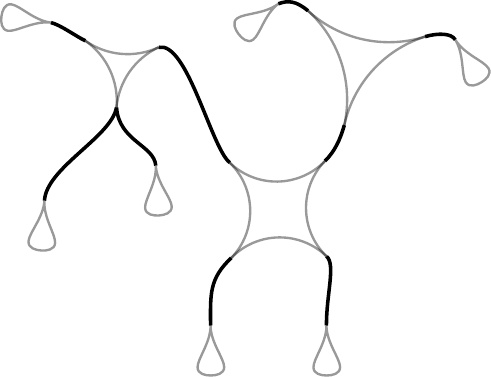}
    \caption{An example of a standardly embedded train track. The infintesimal polygons are drawn in gray.}
    \label{fig:standardlyemb}
\end{figure}
\end{defn}

\begin{prop} \label{prop:realedgenumber}
Let $\tau$ be a standardly embedded train track. The number of real edges of $\tau$ is equal to $-\chi(\tau)$.
\end{prop}
\begin{proof}
It follows from the definition that $|\mathcal{E}_{\infs}| = |\mathcal{V}|$, so $\tau$ has $|\mathcal{V}| + |\mathcal{E}_{\real}|$ edges in total. Hence $\chi(\tau) = |\mathcal{V}| - (|\mathcal{V}| + |\mathcal{E}_{\real}|) = -|\mathcal{E}_{\real}|$. 
\end{proof}

\subsection{Train track partitions} \label{subsec:traintrackpartition}

We now define the notion of train track partitions. These will be used to construct standardly embedded train tracks in the next subsection.

\begin{defn} \label{defn:partition}
A \textit{rectangle} is defined to be a subset $R$ of $S$ such that 
$(R,\ell^s|_R,\ell^u|_R)$ is homeomorphic to 
$$(([0,1] \times [0,1]) \backslash \{x_1,...,x_n\}, \mathrm{Vert}, \mathrm{Hor})$$
where $\{x_1,...,x_n\}$ is a (possibly empty) collection of points on the boundary of $[0,1] \times [0,1]$, $\mathrm{Vert}$ is the foliation by vertical lines and $\mathrm{Hor}$ is the foliation by horizontal lines. See \Cref{fig:rectangle}, where the empty dots denote omitted points. In this paper we will draw stable leaves in red and unstable leaves in blue.

\begin{figure}
    \centering
    \resizebox{!}{3cm}{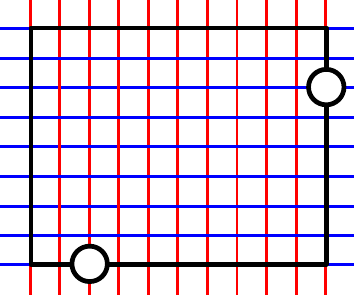}
    \caption{A rectangle in $S$.}
    \label{fig:rectangle}
\end{figure}

We call the two sides of $\partial R$ that lie along leaves of $\ell^s$ the \textit{stable sides} of $R$ and write $\partial_s R$ for the union of the two stable sides. Similarly, we call the two sides of $\partial R$ that lie along leaves of $\ell^u$ the \textit{unstable sides} of $R$ and write $\partial_u R$ for the union of the two unstable sides. In \Cref{fig:rectangle}, the stable sides are the vertical sides while the unstable sides are the horizontal sides.

A \textit{partition} is defined to be a finite collection of rectangles $\{R_i\}$ that have disjoint interiors and cover $S$. A partition $\{R_i\}$ is \textit{Markov} if it satisfies
\begin{itemize}
    \item $f(\bigcup_i \partial_s R_i) \subset \bigcup_i \partial_s R_i$, and
    \item $f^{-1}(\bigcup_i \partial_u R_i) \subset \bigcup_i \partial_u R_i$.
\end{itemize}
\end{defn}

\begin{defn} \label{defn:prong}
Let $x$ be a puncture of $S$. A \textit{stable prong at $x$} is a connected subset of a stable leaf that limits to $x$. A \textit{stable star at $x$} is a maximal disjoint union of prongs at $x$. In particular, a stable star at $x$ has $n$ connected components when $x$ is an $n$-pronged puncture. A \textit{side} of a stable star is the union of two adjacent prongs.

An \textit{unstable prong} and an \textit{unstable star at $x$} are similarly defined.
\end{defn}

See \Cref{fig:prongs} for an example. Here $x$ is a $5$-pronged puncture, so both the stable and unstable stars at $x$ in the figure have $5$ prongs. We have also indicated a side of the unstable star in dark blue.

\begin{figure}
    \centering
    \resizebox{!}{5cm}{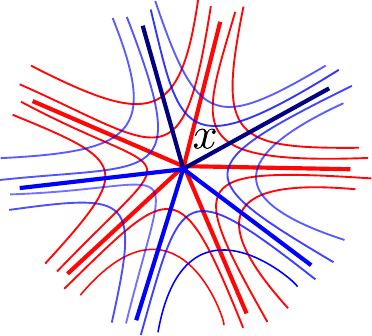}
    \caption{The local picture of $\ell^s$ and $\ell^u$ at a $5$-pronged puncture $x$. We have indicated a side of the unstable star at $x$ in dark blue.}
    \label{fig:prongs}
\end{figure}

\begin{defn} \label{defn:traintrackpartition}
Let $\mathcal{X}=\mathcal{X}_I \sqcup \mathcal{X}_O$ be a partition of the set of punctures of $S$ into a nonempty set of \textit{inner punctures} and a nonempty set of \textit{outer punctures} respectively.

A \textit{train track partition with respect to $(\mathcal{X}_I,\mathcal{X}_O)$} consists of
\begin{itemize}
    \item a partition $\{R_i\}$, along with
    \item a stable star $\sigma^s_x$ at every inner puncture $x \in \mathcal{X}_I$, and
    \item an unstable star $\sigma^u_x$ at every outer puncture $x \in \mathcal{X}_O$
\end{itemize}
such that:
\begin{itemize}
    \item Each stable star $\sigma^s_x$ is disjoint from each unstable star $\sigma^u_x$.
    \item Each stable side of each rectangle $R_i$ lies along the closure of some side of some stable star $\sigma^s_x$, and each point on each $\sigma^s_x$ meets the stable side of some $R_i$.
    \item Each unstable side of each rectangle $R_i$ lies along the closure of some side of some unstable star $\sigma^u_x$, and each point on each $\sigma^u_x$ meets the unstable side of some $R_i$.
\end{itemize}

A train track partition $(\{R_i\},\{\sigma^s_x\},\{\sigma^u_x\})$ is \textit{Markov} if the partition $\{R_i\}$ is Markov.
\end{defn}

Notice that the collection of stable and unstable stars actually determines the partition: the rectangles are the complementary regions of their union.

Also notice that if the train track partition is Markov, then $f$ must preserve the set of inner and outer punctures respectively. In particular, $f$ has at least two puncture orbits. We will show below that this is in fact a sufficient condition. The following construction will be used in the proof.

\begin{constr} \label{constr:ttpartition}
Let $\mathcal{X}=\mathcal{X}_I \sqcup \mathcal{X}_O$ be some partition of $\mathcal{X}$ into two nonempty subsets.

Let $\widehat{\sigma^u_x}$ be the unstable star at each $x \in \mathcal{X}_O$ for which each of its prongs has $\mu^s$-length 1. Let $\sigma^s_x$ be the stable star at each $x \in \mathcal{X}_I$ that is maximal with respect to the property that it is disjoint from $\bigcup_{x \in \mathcal{X}_O} \widehat{\sigma^u_x}$. That is, each $\sigma^s_x$ is defined by extending the stable prongs at $x$ until it bumps into some $\widehat{\sigma^u_x}$. Then define $\sigma^u_x$ to be the unstable star at each $x \in \mathcal{X}_O$ that is maximal with respect to the property that it contains $\widehat{\sigma^u_x}$ and is disjoint from $\bigcup_{x \in \mathcal{X}_I} \sigma^s_x$. That is, each $\sigma^u_x$ is defined by extending the prongs of $\widehat{\sigma^u_x}$ until it bumps into some $\sigma^s_x$.

We claim that each complementary region of $\bigcup_{x \in \mathcal{X}_I} \sigma^s_x \cup \bigcup_{x \in \mathcal{X}_O} \sigma^u_x$ is a rectangle. By construction, the punctures never lie in the interior of a complementary region, hence each complementary region is foliated by properly embedded intervals that are the restriction of the stable leaves. Meanwhile, a stable star $\sigma^s_x$ and an unstable star $\sigma^u_x$ can only meet in the interior of the $\sigma^s_x$ and at the endpoint of the $\sigma^u_x$ or vice versa. That is, they meet in a $\perp$ form. Hence each complementary region has convex corners along its boundary, which implies the claim. 

By taking the partition to be the collection of complementary regions, we obtain a train track partition, which we denote by $\mathcal{M}_{(\mathcal{X}_I,\mathcal{X}_O)}$.
\end{constr}

We remark that this construction of $\mathcal{M}_{(\mathcal{X}_I,\mathcal{X}_O)}$ is natural. More precisely, if $f_1:S_1 \to S_1$ and $f_2:S_2 \to S_2$ are fully-punctured pseudo-Anosov maps, and $g:S_1 \to S_2$ is a homeomorphism that sends the stable and unstable measured foliations $(\ell^{s/u}_1,\mu^{s/u}_1)$ of $f_1$ to $(\ell^{s/u}_2,\mu^{s/u}_2)$ of $f_2$, then for any partition $\mathcal{X}_1=\mathcal{X}_{1,I} \sqcup \mathcal{X}_{1,O}$ of the set of punctures of $S_1$, we have $$g_*\mathcal{M}_{(\mathcal{X}_{1,I},\mathcal{X}_{1,O})}=\mathcal{M}_{(g_*\mathcal{X}_{1,I},g_*\mathcal{X}_{1,O})}.$$
More generally, if $f:S \to S$ is a fully-punctured pseudo-Anosov map, $\pi:\widetilde{S} \to S$ is a finite cover, and $\widetilde{f}:\widetilde{S} \to \widetilde{S}$ is a lift of $f$, with its stable and unstable measured foliations lifted from that of $f$, then for any partition $\mathcal{X}=\mathcal{X}_I \sqcup \mathcal{X}_O$, we have
$$\pi^*(\mathcal{M}_{(\mathcal{X}_I,\mathcal{X}_O)})=\mathcal{M}_{(\pi^*\mathcal{X}_I,\pi^*\mathcal{X}_O)}.$$
We will be applying this fact implicitly in the sequel. 

\begin{prop} \label{prop:markovttpartition}
If $\mathcal{X}_I$ and $\mathcal{X}_O$ are $f$-invariant, then $\mathcal{M}_{(\mathcal{X}_I,\mathcal{X}_O)}$ is Markov.
\end{prop}
\begin{proof}
In \Cref{constr:ttpartition}, notice that $f^{-1}(\bigcup_{x \in \mathcal{X}_O} \widehat{\sigma^u_x}) \subset \bigcup_{x \in \mathcal{X}_O} \widehat{\sigma^u_x}$ since $f$ expands $\mu^s$-lengths. This implies that $f^{-1}(\bigcup_{x \in \mathcal{X}_I} \sigma^s_x) \supset \bigcup_{x \in \mathcal{X}_I} \sigma^s_x$, since each $f^{-1}(\sigma^s_x)$ is obtained by extending the stable prongs at $x$ until it bumps into some $f^{-1}(\widehat{\sigma^u_x})$. Similarly, this in turn implies that $f^{-1}(\bigcup_{x \in \mathcal{X}_O} \sigma^u_x) \subset \bigcup_{x \in \mathcal{X}_O} \sigma^u_x$. Now the proposition follows from the observation that $\bigcup_i \partial_s R_i = \bigcup_{x \in \mathcal{X}_I} \sigma^s_x$ and $\bigcup_i \partial_u R_i = \bigcup_{x \in \mathcal{X}_O} \sigma^u_x$. 
\end{proof}

\subsection{From train track partitions to standardly embedded train tracks} \label{subsec:ttpartitiontott}

Train track partitions give rise to standardly embedded train tracks via the following construction.

\begin{constr} \label{constr:ttpartitiontott}
Let $\mathcal{M}=(\{R_i\},\{\sigma^s_x\},\{\sigma^u_x\})$ be a train track partition with respect to $(\mathcal{X}_I,\mathcal{X}_O)$.

Define a graph $\tau_\mathcal{M}$ by taking a set of vertices in one-to-one correspondence with the sides of the stable stars $\sigma^s_x$. The edges of $\tau_\mathcal{M}$ will come in two types: infinitesimal and real. The infinitesimal edges are in one-to-one correspondence with the prongs of $\sigma^s_x$, with their endpoints at the two vertices corresponding to the two sides the prong lies in. The real edges are in one-to-one correspondence with the rectangles $R_i$, with their endpoints at the two vertices corresponding to the two sides the stable sides of $R_i$ lie along.

The smoothing at each vertex is defined by separating the infinitesimal edges and the real edges. Each vertex is incident to exactly two infinitesimal half-edges so it suffices to order the real half-edges: This order is determined by the position of rectangles at the corresponding side. It is straightforward to check that this makes $\tau_\mathcal{M}$ into a standardly embedded train track.

We illustrate a local picture of this construction in \Cref{fig:partitiontott}.

\begin{figure}
    \centering
    \resizebox{!}{8cm}{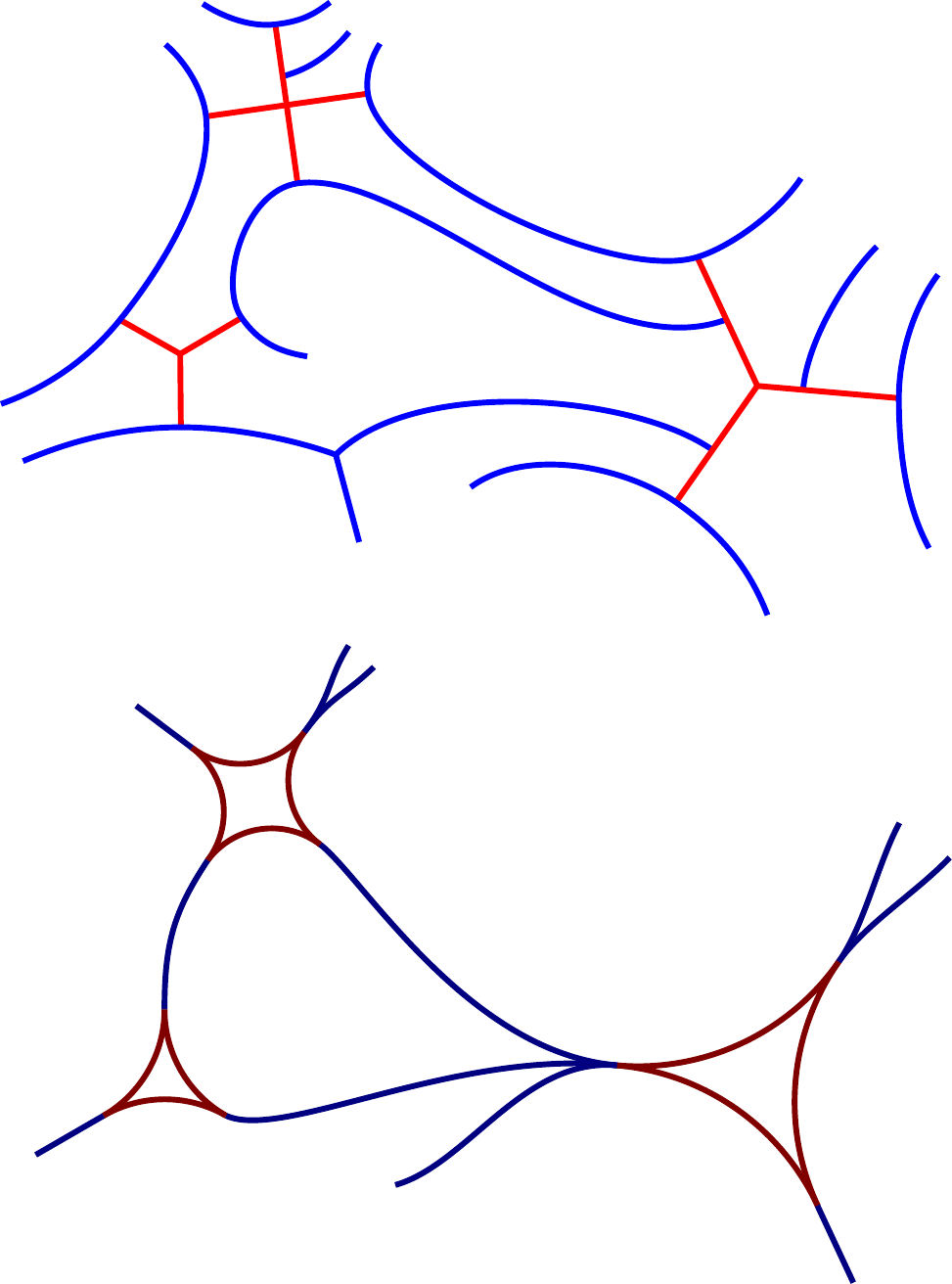}
    \caption{A local example of \Cref{constr:ttpartitiontott}.}
    \label{fig:partitiontott}
\end{figure}

When $\mathcal{M}=\mathcal{M}_{(\mathcal{X}_I,\mathcal{X}_O)}$, we write $\tau_{(\mathcal{X}_I,\mathcal{X}_O)} = \tau_{\mathcal{M}_{(\mathcal{X}_I,\mathcal{X}_O)}}$.
\end{constr}

\begin{constr} \label{constr:markovttpartitiontott}
When $\mathcal{M}$ is a Markov train track partition, we can in addition define a train track map $f_\mathcal{M}:\tau_\mathcal{M} \to \tau_\mathcal{M}$ as follows:

$f_\mathcal{M}$ maps the vertex corresponding to a side $s$ of a stable star to the vertex corresponding to the side containing $f(s)$, maps the infinitesimal edge corresponding to a stable prong $p$ to the infinitesimal edge corresponding to the prong containing $f(p)$, and maps the real edge corresponding to a rectangle $R_i$ to the edge path corresponding to the sequence of rectangles and prongs passed through by $f(R_i)$.

When $\mathcal{M}=\mathcal{M}_{(\mathcal{X}_I,\mathcal{X}_O)}$, we write $f_{(\mathcal{X}_I,\mathcal{X}_O)} = f_{\mathcal{M}_{(\mathcal{X}_I,\mathcal{X}_O)}}$.
\end{constr}

The train track $\tau_\mathcal{M}$ and the train track map $f_\mathcal{M}$ capture the dynamics of $f$ in the following sense.

\begin{prop} \label{prop:tienbd=surface}
For every train track partition $\mathcal{M}$, there is an embedding $\iota$ of the tie neighborhood $N$ of $\tau_\mathcal{M}$ into $S$ that is a homotopy equivalence, such that $\iota(\tau_\mathcal{M})$ fully carries the unstable lamination of $f$ (obtained by blowing air into the leaves of the unstable foliation $\ell^u$ that contain the punctures). 

Moreover, if $\mathcal{M}$ is Markov, then $f \iota$ and $\iota f_\mathcal{M}$ are homotopic embeddings of $N$ in $S$.
\end{prop}
\begin{proof}
Cut $S$ along the unstable stars $\sigma^u_x$, which does not change the homeomorphism type of $S$. The resulting space $\widehat{S}$ can be obtained by gluing the rectangles $R_i$ along their stable sides to the stable stars $\sigma^s_x$. Since the pattern of gluing corresponds exactly to that used to glue real edges to infinitesimal polygons of $\tau_\mathcal{M}$, we can define an embedding $\iota:\tau_\mathcal{M} \hookrightarrow S$ by sending the infinitesimal polygons into a neighborhood of the corresponding stable stars, then sending the real edges into their corresponding rectangles. 

After blowing air into the leaves of the unstable foliation $\ell^u$ that contain the punctures, we can assume that the leaves of the unstable lamination are contained in $\widehat{S}$. In this case, we can extend $\iota$ into an embedding of $N$ such that the ties are transverse to the leaves.

With this $\iota$, $f \iota$ will send each the union of ties in $N$ intersecting an infinitesimal polygon into a neighborhood of the image of the corresponding stable star under $f$, which by definition is the stable star corresponding to the image of the infinitesimal polygon under $f_\mathcal{M}$. Similarly, $f \iota$ will send the union of ties that intersect a real edge into the sequence of rectangles passed through by the image of the corresponding rectangle under $f$, which by definition is the sequence of rectangles corresponding to the image of the real edge under $f_\mathcal{M}$. From this one can construct a homotopy $f \iota \simeq \iota f_\mathcal{M}$.
\end{proof}

In the sequel, we will always use $\iota$ in \Cref{prop:tienbd=surface} to embedd $\tau_\mathcal{M}$ in $S$ if needed.

\begin{rmk} \label{rmk:ttpuncture}
The proof of \Cref{prop:tienbd=surface} shows that the boundary components of $\tau_\mathcal{M}$ are in one-to-one correspondence with the punctures of $S$. Namely, a boundary component $c$ corresponds to the puncture $p$ which $\iota(c)$ is homotopic into. Under this correspondence, $c$ is an element of $\partial_{I/O} \tau_\mathcal{M}$ if and only if $p$ is an element of $\mathcal{X}_{I/O}$ respectively. Also, $c$ is $n$-pronged if and only if $p$ is $n$-pronged.
\end{rmk}

The fact that $f_\mathcal{M}$ captures the dynamics of $f$ implies that we can compute the expansion factor of $f$ using the transition matrix of $f_\mathcal{M}$. More specifically, notice that by definition, $f_\mathcal{M}$ maps each infinitesimal edge to a single infinitesimal edge, hence if we list the infinitesimal edges in front of the real edges, the transition matrix of $f_\mathcal{M}$ will be of the form 
$$f_{\mathcal{M}*} =
\begin{bmatrix}
P & * \\
0 & f_{\mathcal{M}*}^\real
\end{bmatrix}$$
where $P$ is a permutation matrix. We call $f_{\mathcal{M}*}^\real$ the \textit{real transition matrix} of $f_\mathcal{M}$.

\begin{prop} \label{prop:realedgesPF}
For any Markov train track partition $\mathcal{M}$, $f_{\mathcal{M}*}^\real$ is Perron-Frobenius.
\end{prop}
\begin{proof}
The rows and columns of $f_{\mathcal{M}*}^\real$ are indexed by the real edges of $\tau_\mathcal{M}$, which in turn correspond to the rectangles $R_i$ in the train track partition $\mathcal{M}$. Under this correspondence, the $(R_j,R_i)$-entry of $(f_{\mathcal{M}*}^\real)^k$ is the number of times $f^k(R_i)$ crosses $R_j$. We claim that for each $i$ there exists $k_i$ such that the $(R_j,R_i)$-entry of $(f_{\mathcal{M}*}^\real)^{k_i}$ is positive for each $j$. This would imply that $(f_{\mathcal{M}*}^\real)^{\prod k_i}$ is positive, hence $f_{\mathcal{M}*}^\real$ is Perron-Frobenius.

To show the claim, we fix an $i$. By \Cref{prop:halfleafdense}, there exists a periodic point $z$ in the interior of $R_i$, say of period $p$. Consider a short interval $I$ lying along the unstable leaf passing through $z$, with one endpoint on $z$ and contained within $R_i$. Up to doubling $p$, we have $I \subset f^p(I)$ and $\mu^s(f^p(I))=\lambda^p \mu^s(I)$. 

Now $\bigcup_{s=0}^\infty f^{sp}(I)$ is a half-leaf, hence by \Cref{prop:halfleafdense}, is dense in $S$. So there exists $s>0$ such that $f^{sp}(I)$ meets the interior of each rectangle, which implies that $f^{sp}(R_i)$ crosses every rectangle and proves the claim. 
\end{proof}

\begin{prop} \label{prop:realedgesdilatation}
For any Markov train track partition $\mathcal{M}$, the spectral radius of $f_{\mathcal{M}*}^\real$ is the expansion factor of $f$.
\end{prop}
\begin{proof}
We will directly define an eigenvector $u$ of $f_{\mathcal{M}*}^\real$. The entries of $u$ can be indexed by the rectangles $R_i$ in $\mathcal{M}$ as above. We define the $R_i$-entry of $u$ to be the $\mu^s$-measure of an unstable side of $R_i$. By definition, the $R_i$-entry of $f_{\mathcal{M}*}^\real u$ is the $\mu^s$ measure of the image of an unstable side of $f(R_i)$, which is $\lambda$ times the $R_i$-entry of $u$. Hence $u$ is a $\lambda$-eigenvector of $f_{\mathcal{M}*}^\real$. 

We have shown that $f_{\mathcal{M}*}^\real$ is Perron-Frobenius in \Cref{prop:realedgesPF}. Here $u$ is positive, hence by \Cref{thm:PF}, $\lambda$ is the spectral radius of $f_{\mathcal{M}*}^\real$. 
\end{proof}

Let $\mathcal{X}=\mathcal{X}_I \sqcup \mathcal{X}_O$ be a partition of the set of punctures of $S$ into nonempty $f$-invariant sets. \Cref{prop:realedgenumber} and \Cref{prop:tienbd=surface} imply that $f_{(\mathcal{X}_I,\mathcal{X}_O)*}^\real$ is a $|\chi(S)|$-by-$|\chi(S)|$ matrix, while \Cref{prop:realedgesPF} implies that it is Perron-Frobenius. Our goal in \Cref{sec:thurstonform} and \Cref{sec:radical} is to show that $f_{(\mathcal{X}_I,\mathcal{X}_O)*}^\real$ is reciprocal, for then we can apply \Cref{thm:McMullen} and prove our main theorem.

\subsection{Invariant standardly embedded train tracks} \label{subsec:invariantstandardlyemb}

In this subsection, we will explain how for every Markov train track partition $\mathcal{M}$, $\tau_\mathcal{M}$ is a $f$-invariant train track, which implies \Cref{thm:standardlyembeddedtt}. To do so, we have to discuss elementary moves on train tracks.

\begin{defn} \label{defn:subdivision}
Let $\tau$ be a train track. Suppose $e$ is an edge of $\tau$ with endpoints at switches $v_1, v_2$. We can define a new train track $\tau'$ by declaring an interior point $v$ of $e$ as a new vertex and replacing $e$ with two edges $e_1,e_2$ connecting $v_1$ and $v_2$ respectively with $v$. The ordering and smoothing at $v_i$ is unchanged, with the half edge determined by $e_i$ replacing that determined by $e$. For $v$, we take the unique choice of ordering and smoothing. 

Note that $\tau'$ is homeomorphic to $\tau$ as a topological space but non-isomorphic as a graph. Nevertheless, there is a natural train track map $\tau \to \tau'$. We refer to this map as \textit{the subdivision move on $e$}. See \Cref{fig:elementarymoves} top.
\end{defn}

\begin{defn} \label{defn:elementaryfolding}
Let $\tau$ be a train track. Suppose $e_1,e_2$ are two edges of $\tau$ which determine a cusp at switch $v$, and having their other endpoint on switches $v'_1$ and $v'_2$ respectively. Define a new train track $\tau'$ by combining $v'_1$ and $v'_2$ into one switch $v'$ and replacing $e_1$ and $e_2$ by a single edge $e$ connecting $v$ and $v'$. Without loss of generality suppose that $e_1$ lies to the left of $e_2$ at $v$ and the half edge of $e_i$ lies in $\mathcal{E}^2_{v_i}$ for both $i=1,2$. Then the ordering on $\mathcal{E}^1_{v'} = \mathcal{E}^1_{v'_1} \sqcup \mathcal{E}^1_{v'_2}$ is determined by placing all the half edges in $\mathcal{E}^1_{v'_1}$, in their original order, to the left of those in $\mathcal{E}^1_{v'_2}$, in their original order. The ordering on $\mathcal{E}^2_{v'} = (\mathcal{E}^2_{v'_1} \backslash \{e_1\}) \sqcup \{e\} \sqcup (\mathcal{E}^2_{v'_2} \backslash \{e_2\})$ is determined by placing all the half edges in $\mathcal{E}^2_{v_2} \backslash \{e_2\}$, in their original order, to the left of $e$, which is in turn placed to the left of all the half edges in $\mathcal{E}^1_{v_1} \backslash \{e_1\}$, in their original order.

There is a train track map $\tau \to \tau'$ defined by sending $v_1$ and $v_2$ to $v$, sending $e_1$ and $e_2$ to $e$, and sending the remaining vertices and edges to themselves. We refer to this map as \textit{the elementary folding move on $(e_1,e_2)$}. See \Cref{fig:elementarymoves} bottom.
\end{defn}

\begin{figure}
    \centering
    \fontsize{18pt}{18pt}\selectfont
    \resizebox{!}{5cm}{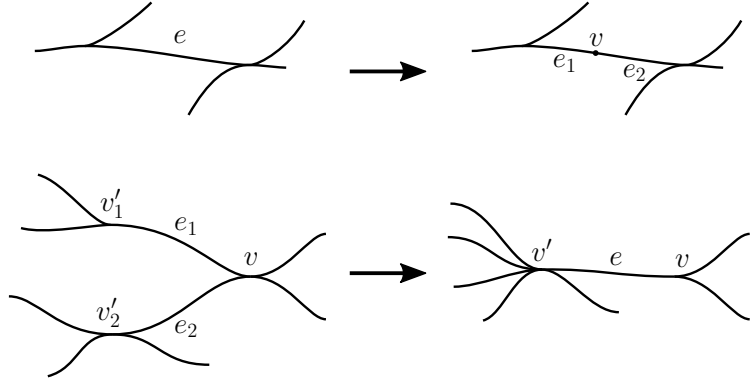}
    \caption{Elementary moves. Top: the subdivision move on $e$. Bottom: the elementary folding move on $(e_1,e_2)$.}
    \label{fig:elementarymoves}
\end{figure}

We refer to a subdivision move or an elementary folding move as an \textit{elementary move} in general.

\begin{defn} \label{defn:invarianttt}
Let $\tau$ be a train track and with tie neighborhood $N$, and let $\iota:N \hookrightarrow S$ be an embedding. Suppose there exists train track maps $f_1,...,f_n, \sigma$ where each $f_i$ is an elementary move and $\sigma$ is an isomorphism of train tracks, such that $f \iota$ and $\iota \sigma f_n \cdots f_1$ are homotopic embeddings of $N$ in $S$. Then $\iota(\tau)$ is said to be \textit{$f$-invariant}.

When $\iota$ is understood, we will just say that $\tau$ is $f$-invariant.
\end{defn}

\Cref{prop:widerfolding} below will do most of the heavy lifting in showing that $\tau_\mathcal{M}$ is $f$-invariant. We will also be applying \Cref{prop:widerfolding} and its corollary, \Cref{prop:innerouterfolding}, in \Cref{sec:thurstonform} and \Cref{sec:radical}.

\begin{defn} \label{defn:wider}
Let $\mathcal{M}=(\{R_i\},\{\sigma^s_x\},\{\sigma^u_x\})$ and $\mathcal{M}'=(\{R'_i\},\{\sigma'^s_x\},\{\sigma'^u_x\})$ be two train track partitions with respect to $(\mathcal{X}_I,\mathcal{X}_O)$ and $(\mathcal{X}'_I,\mathcal{X}'_O)$ respectively. We say that $\mathcal{M}$ is \textit{wider} than $\mathcal{M}'$ if $\bigcup_{x \in \mathcal{X}_I} \sigma^s_x \subset \bigcup_{x \in \mathcal{X}'_I} \sigma'^s_x$ and $\bigcup_{x \in \mathcal{X}_O} \sigma^u_x \supset \bigcup_{x \in \mathcal{X}'_O} \sigma'^u_x$. Notice this implies that $\mathcal{X}_I \subset \mathcal{X}'_I$ and $\mathcal{X}_O \supset \mathcal{X}'_O$. 
\end{defn}

\begin{prop} \label{prop:widerfolding}
If $\mathcal{M}$ is wider than $\mathcal{M}'$ then there exists elementary moves $f_1,...,f_n$ such that $f_n \cdots f_1$ maps $\tau_\mathcal{M}$ to $\tau_{\mathcal{M}'}$. 
\end{prop}
\begin{proof}
For notational convenience we set $\sigma^s_x = \varnothing$ for $x \in \mathcal{X}_O$ and $\sigma^u_x = \varnothing$ for $x \in \mathcal{X}_I$, and similarly $\sigma'^s_x = \varnothing$ for $x \in \mathcal{X}'_O$ and $\sigma'^u_x = \varnothing$ for $x \in \mathcal{X}'_I$. Then we will have $\sigma^s_x \subset \sigma'^s_x$ and $\sigma^u_x \supset \sigma'^u_x$ for every $x$.

We label the prongs of all the $\sigma^u_x$ as $p_1,...,p_N$, and label the prong of the $\sigma'^u_x$ that is contained in $p_i$ as $p'_i$. The idea of the proof is that in contracting each $p_i$ into $p'_i$, we combine some rectangles, and this determines corresponding elementary moves. See \Cref{fig:widerfolding1}. The precise description of the proof is rather technical. The reader may wish to skip it on the first reading.

Define subsets $p^j_i$ of $p_i$ by setting $p^0_i=p_i$ and inductively defining $p^{j+1}_i$ to be the subset maximal with respect to the property of containing $p'_i$ and being properly contained in $p^j_i$, and having an endpoint lying on $\bigcup_x \sigma'^s_x$. Suppose $p'_i=p^{n_i}_i$. Write $U_{\sum_{i=1}^{k-1} n_i + j}=\bigcup_{i=1}^{k-1} p'_i \cup p^j_k \cup \bigcup_{i=k}^N p_i$.

Notice that for each $j$, $(\bigcup_x \sigma'^s_x) \backslash U_j$ is a union of intervals. We call each such interval a \textit{prong} of $(\bigcup_x \sigma'^s_x) \backslash U_j$. Each interval has two sides in the $\ell^u$ direction. We consider two sides to such intervals to be equivalent if there is a rectangle in the complement of $\bigcup_x \sigma'^s_x \cup U_j$ with a stable side incident to the two sides. We call each equivalence class of sides a \textit{side} of $(\bigcup_x \sigma'^s_x) \backslash U_j$. Thus each prong lies on two sides, and for each $j$, $(\bigcup_x \sigma'^s_x) \backslash U_{j+1}$ has one or two less sides than $(\bigcup_x \sigma'^s_x) \backslash U_j$. Intuitively, as $j$ increases, $U_j$ shrinks and the sides of $(\bigcup_x \sigma'^s_x) \backslash U_j$ get combined, until we hit $j=\sum_{i=1}^N n_i$ and $(\bigcup_x \sigma'^s_x) \backslash U_{\sum_{i=1}^N n_i}$ becomes exactly the union of stable stars that are the $\sigma'^s_x$ and the definitions of prongs and sides agree with our previous usage.

We will define train tracks $\tau_j$ for each $j=0,...,\sum_{i=1}^N n_i$ such that 
\begin{itemize}
    \item the vertices of $\tau_j$ are in one-to-one correspondence with the sides of $(\bigcup_x \sigma'^s_x) \backslash U_j$, and
    \item the edges of $\tau_j$ are in one-to-one correspondence with the rectangles in the complement of $\bigcup_x \sigma'^s_x \cup U_j$ and the prongs of $(\bigcup_x \sigma'^s_x) \backslash U_j$,
\end{itemize}
and such that each $\tau_{j+1}$ is obtained from $\tau_j$ via elementary folding moves.

First, to define $\tau_0$, consider the components of $(\bigcup_x \sigma'^s_x) \backslash U_0 = (\bigcup_x \sigma'^s_x) \backslash (\bigcup_x \sigma^u_x)$ that do not lie in $\bigcup_x \sigma^s_x$. Each of these is an interval lying in a rectangle $R_i$. For each of these intervals, we subdivide twice the real edge of $\tau_\mathcal{M}$ corresponding to the rectangle the interval lies in. That is, we add two new vertices, each corresponding to a side of $(\bigcup_x \sigma'^s_x) \backslash U_0$.

Inductively, suppose $\tau_j$ is defined. When going from $U_j$ to $U_{j+1}$, two rectangles in the complement of $\bigcup_x \sigma'^s_x \cup U_j$ are combined into one in $\bigcup_x \sigma'^s_x \cup U_{j+1}$. We fold the two edges of $\tau_j$ corresponding to these two rectangles. If $(\bigcup_x \sigma'^s_x) \backslash U_{j+1}$ has one less component than $(\bigcup_x \sigma'^s_x) \backslash U_j$, then two prongs of $(\bigcup_x \sigma'^s_x) \backslash U_j$ are combined into one. In this case we also fold the two edges of $\tau_j$ corresponding to these two prongs. After these one or two elementary folding moves, we obtain $\tau_{j+1}$. See \Cref{fig:widerfolding1}.

\begin{figure}
    \centering
    \resizebox{!}{5cm}{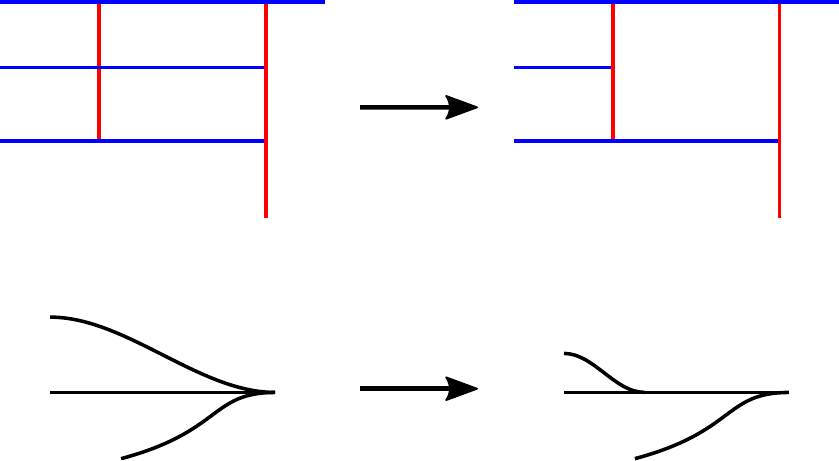}
    \caption{Performing elementary folding moves on $\tau_j$ to obtain $\tau_{j+1}$.}
    \label{fig:widerfolding1}
\end{figure}

Continuing inductively, by the time we reach $\tau_{\sum_{i=1}^N n_i}$, this will be a train track with vertices in one-to-one correspondence with the sides of $\bigcup_x \sigma'^s_x$ and edges in one-to-one correspondence with the rectangles $R'_i$ and the prongs of $\bigcup_x \sigma'^s_x$, hence $\tau_{\sum_{i=1}^N n_i}=\tau_{\mathcal{M}'}$.
\end{proof}

\begin{prop} \label{prop:innerouterfolding}
If $\mathcal{X}_I \subset \mathcal{X}'_I$, then there are elementary folding moves $f_1,\dots,f_n$ such that $f_n \cdots f_1$ maps $\tau_{(\mathcal{X}_I,\mathcal{X}_O)}$ to $\tau_{(\mathcal{X}'_I,\mathcal{X}'_O)}$
\end{prop}
\begin{proof}
If $\mathcal{X}_I \subset \mathcal{X}'_I$, then $\mathcal{X}_O \supset \mathcal{X}'_O$, so $\bigcup \widehat{\sigma^u_x} \supset \bigcup \widehat{\sigma'^u_x}$ in \Cref{constr:ttpartition}. This implies that $\bigcup \sigma^s_x \subset \bigcup \sigma'^s_x$, which in turn implies that $\bigcup \sigma^u_x \supset \bigcup \sigma'^u_x$. That is, $\mathcal{M}_{(\mathcal{X}_I,\mathcal{X}_O)}$ is wider than $\mathcal{M}_{(\mathcal{X}'_I,\mathcal{X}'_O)}$. So this proposition follows from \Cref{prop:widerfolding}.
\end{proof}

\begin{lemma} \label{lemma:ttmapfactorization}
For any Markov train track partition $\mathcal{M}$, the train track map $f_\mathcal{M}:\tau_\mathcal{M} \to \tau_\mathcal{M}$ can be written as a composition of train track maps $\sigma f_n \cdots f_1$ where each $f_i$ is an elementary move and $\sigma$ is an isomorphism of train tracks.
\end{lemma}
\begin{proof}
Note that $\mathcal{M}$ being Markov implies that $\mathcal{M}$ is wider than $f^{-1}(\mathcal{M})$. So we can apply \Cref{prop:widerfolding} to get elementary moves $f_1,\dots,f_n$ such that $f_n \cdots f_1$ maps $\tau_\mathcal{M}$ to $\tau_{f^{-1}(\mathcal{M})}$. Meanwhile, since $f$ sends $f^{-1}(\mathcal{M})$ to $\mathcal{M}$, there is an induced isomorphism of train tracks $\sigma:\tau_{f^{-1}(\mathcal{M})} \to \tau_{\mathcal{M}}$. 

It remains to show that $f_\mathcal{M} = \sigma f_n \cdots f_1$. This will follow from checking that these maps send each edge $e$ to the same edge path. 

Suppose $e$ is infinitesimal and corresponds to a prong $p$. Then $f_n \cdots f_1$ sends $e$ to the infinitesimal edge in $\tau_{f^{-1}(\mathcal{M})}$ corresponding to the prong containing $p$, which $\sigma$ sends to the infinitesimal edge in $\tau_{\mathcal{M}}$ corresponding to the prong containing $f(p)$. This is equal to the image of $e$ under $f_\mathcal{M}$. 

Suppose $e$ is real and corresponds to a rectangle $R$. Then $f_n \cdots f_1$ sends $e$ to the edge path in $\tau_{f^{-1}(\mathcal{M})}$ corresponding to the sequence of rectangles and prongs passed by an unstable side of $R$, which $\sigma$ sends to the edge path in $\tau_{\mathcal{M}}$ corresponding to the sequence of rectangles and prongs passed by an unstable side of $f(R)$. This is equal to the image of $e$ under $f_\mathcal{M}$. 
\end{proof}

Combining \Cref{prop:markovttpartition}, \Cref{prop:tienbd=surface}, and \Cref{lemma:ttmapfactorization}, we have the following.

\begin{prop} \label{prop:invariantstandardlyembedded}
Let $f:S \to S$ be a  fully-punctured pseudo-Anosov map with at least two puncture orbits. Let $\mathcal{X}=\mathcal{X}_I \sqcup \mathcal{X}_O$ be some partition of the set of punctures into two nonempty $f$-invariant subsets. Then there exists a $f$-invariant train track $\tau$ that is standardly embedded with respect to $(\partial_I \tau, \partial_O \tau)$, where the boundary components in $\partial_{I/O} \tau$ are homotopic into the punctures in $\mathcal{X}_{I/O}$ respectively.
\end{prop}

\Cref{thm:standardlyembeddedtt} follows from \Cref{prop:invariantstandardlyembedded}.

\section{The Thurston symplectic form} \label{sec:thurstonform}

In this section, we make the first steps towards showing that $f_{(\mathcal{X}_I,\mathcal{X}_O)*}^\real$ is reciprocal. The idea is to consider the space of weights on the train track, on which the Thurston symplectic form can be defined. There is a slight misnomer here: the Thurston symplectic form is in general degenerate hence not a symplectic form. However, we will only set up the theory in this section, leaving the task of addressing this to \Cref{sec:radical}.

\subsection{Weight space} \label{subsec:weightspace}

\begin{defn} \label{defn:weightspace}
Let $\tau$ be a train track. A system of \textit{weights} on $\tau$ is an assignment of a real number $w(e)$ to each edge $e$ of $\tau$ such that at each switch $v$,
$$\sum_{e \in \mathcal{E}^1_v} w(e)=\sum_{e \in \mathcal{E}^2_v} w(e).$$
It is convenient to think of the weight of a branch as the width of the branch. From this point of view, the equation above simply states that the total width on the two sides of a switch match up.

The \textit{weight space} of $\tau$ is the linear space of all systems of weights on $\tau$, which we denote by $\mathcal{W}(\tau)$.
\end{defn}

Consider the space $\mathbb{R}^{\mathcal{E}(\tau)}$ of all assignments of real numbers to the edges of $\tau$. Define the linear map $T_v:\mathbb{R}^{\mathcal{E}(\tau)} \to \mathbb{R}$ by $T_v(w) = \sum_{e \in \mathcal{E}^1_v} w(e) - \sum_{e \in \mathcal{E}^2_v} w(e)$ for each switch $v$, and define $T_\mathcal{V}:\mathbb{R}^{\mathcal{E}(\tau)} \to \mathbb{R}^{\mathcal{V}(\tau)}$ by $T_\mathcal{V}(w)=(T_v(w))_{v \in \mathcal{V}(\tau)}$. Then the weight space $\mathcal{W}(\tau)$ is the kernel of $T_V$. In other words, we have the exact sequence
\begin{center}
\begin{tikzcd}
0 \arrow[r] & \mathcal{W}(\tau) \arrow[r] & \mathbb{R}^{\mathcal{E}(\tau)} \arrow[r, "T_\mathcal{V}"] & \mathbb{R}^{\mathcal{V}(\tau)} \\
\end{tikzcd}
\end{center}

We remark that there is no canonical way to label $\mathcal{E}^1_v$ and $\mathcal{E}^2_v$ for each $v$, so each $T_v$ can only be canonically defined up to a sign. This is not very significant for our purposes, but for concreteness let us simply fix some labelling of $\mathcal{E}^1_v$ and $\mathcal{E}^2_v$ for each $v$.

\begin{prop} \label{prop:weightspaceses}
Let $f:\tau \to \tau'$ be a train track map. There exists a signed permutation matrix $P \in M_{\mathcal{V}(\tau') \times \mathcal{V}(\tau)}(\mathbb{R})$ which fits into the commutative diagram
\begin{center}
\begin{tikzcd}
0 \arrow[r] & \mathcal{W}(\tau) \arrow[r] \arrow[d, "\mathcal{W}(f)"] & \mathbb{R}^{\mathcal{E}(\tau)} \arrow[r, "T_\mathcal{V}"] \arrow[d, "f_*"] & \mathbb{R}^{\mathcal{V}(\tau)} \arrow[d, "P"] \\
0 \arrow[r] & \mathcal{W}(\tau') \arrow[r] & \mathbb{R}^{\mathcal{E}(\tau')} \arrow[r, "T_\mathcal{V}"] & \mathbb{R}^{\mathcal{V}(\tau')} 
\end{tikzcd}
\end{center}
Here $f_*$ is the transition matrix of $f$ and $\mathcal{W}(f)$ is the restriction of $f_*$ to $\mathcal{W}(\tau)$.
\end{prop}
\begin{proof}
Recall that $f$ sends switches to switches. Define 
$$P_{v',v}=
\begin{cases}
1 \text{ , if $v'=f(v)$ and $D_v f$ sends $\mathcal{E}^1_v$ into $\mathcal{E}^1_{v'}$} \\
-1 \text{ , if $v'=f(v)$ and $D_v f$ sends $\mathcal{E}^1_v$ into $\mathcal{E}^2_{v'}$} \\
0 \text{ , otherwise}
\end{cases}$$

We have to check that $P T_\mathcal{V} = T_\mathcal{V} f_*$. For convenience, let us denote the unit vector in $\mathbb{R}^{\mathcal{E}(\tau)}$ corresponding to $e \in \mathcal{E}(\tau)$ by $e$ as well. Then $P T_\mathcal{V}(e)$ is the vector having two $\pm 1$ entries at the images of the endpoints of $e$ under $f$, whereas $T_\mathcal{V} f_*(e)$ is the vector having two $\pm 1$ entries at the endpoints of the image of $e$ under $f$. The position of the only two nonzero entries hence coincide. It can also be checked that their signs coincide respectively.
\end{proof}

Note that by \Cref{prop:weightspaceses} and the discussion above \Cref{prop:realedgesPF} (along with \Cref{prop:reciprocalfacts}), in order to show that $f_{(\mathcal{X}_I,\mathcal{X}_O)*}^\real$ is reciprocal, it suffices to show that $\mathcal{W}(f_{(\mathcal{X}_I,\mathcal{X}_O)})$ is reciprocal.

\subsection{The Thurston symplectic form} \label{subsec:thurstonform}

In this subsection we define the Thurston symplectic form on the weight space of a train track. This is done for trivalent train tracks in \cite[Section 3.2]{PH92}; we simply generalize the discussion to general train tracks. 

\begin{defn} \label{defn:thurstonform}
Let $w_1, w_2 \in \mathcal{W}(\tau)$ be two systems of weights on $\tau$. We define
$$\omega(w_1,w_2) = \sum_{v \in \mathcal{V}(\tau)} \sum_{e_1 \text{ left of } e_2} (w_1(e_1) w_2(e_2) - w_1(e_2) w_2(e_1))$$
where the second summation is taken over all pairs $e_1,e_2$ in $\mathcal{E}_v$ for which $e_1$ is on the left of $e_2$. 

Then $\omega$ is clearly a skew-symmetric bilinear form on $\mathbb{R}^{\mathcal{E}(\tau)}$ hence on $\mathcal{W}(\tau)$. We call $\omega$ the \textit{Thurston symplectic form}.
\end{defn}

The definition of the Thurston symplectic form $\omega$ is motivated from the algebraic intersection number. We refer to \cite[Lemma 3.2.2]{PH92} for an explanation of this. The property of $\omega$ that matters to us is that it is preserved by the type of train track maps that we study.

\begin{lemma} \label{lemma:subdivisionsymplectic}
If $f:\tau \to \tau'$ is a subdivision move, then $\mathcal{W}(f)$ is an isomorphism that preserves the Thurston symplectic form $\omega$.
\end{lemma}
\begin{proof}
We use the notation as in \Cref{defn:subdivision}, and write $w'=(\mathcal{W}(f))(w)$. 

The map $\mathcal{W}(f)$ is an isomorphism since $w(e)$ can be recovered as $w'(e_1) = w'(e_2)$.

When computing $\omega$ in $\tau'$, $v$ does not make a contribution, since it only meets two half-edges. For the other vertices, since $w'(e_1) = w'(e_2) = w(e)$, their contributions remain the same. 
\end{proof}

\begin{lemma} \label{lemma:elementaryfoldingsymplectic}
If $f:\tau \to \tau'$ is an elementary folding move, then $\mathcal{W}(f)$ is an isomorphism that preserves the Thurston symplectic form $\omega$.
\end{lemma}
\begin{proof}
We use the notation as in \Cref{defn:elementaryfolding}, and write $w'_i=(\mathcal{W}(f))(w_i)$. 

The map $\mathcal{W}(f)$ is an isomorphism since $w_i(e_j)$ can be recovered as 
$$\sum_{e \in \mathcal{E}^1_{v_j}} w'_i(e) - \sum_{e \in \mathcal{E}^2_{v_j} \backslash \{e_j\}} w'_i(e).$$

When computing $\omega$ in $\tau'$, the contributions from $v$ and $v'$ are:

(we omit the summands when they are $w_1(e') w_2(e'') - w_1(e'') w_2(e')$ and omit writing `$e'$ to the left of $e''$' under the summations)

\begin{align*}
    & \sum_{e', e'' \in \mathcal{E}^2_v} + \sum_{e', e'' \in \mathcal{E}^1_v} + \sum_{e', e'' \in \mathcal{E}^2_{v'}} + \sum_{e', e'' \in \mathcal{E}^1_{v'}} \\
    =& \sum_{e', e'' \in \mathcal{E}^2_v} + ~(\sum_{e', e'' \in \mathcal{E}^1_v} - ~(w_1(e_1) w_2(e_2) - w_1(e_2) w_2(e_1))) \\
    &+ (\sum_{e', e'' \in \mathcal{E}^2_{v'_1}} + \sum_{e', e'' \in \mathcal{E}^2_{v'_2}} + \sum_{e' \in \mathcal{E}^2_{v'_2}, e'' \in \mathcal{E}^2_{v'_1}} - ~(w_1(e_2) w_2(e_1) - w_1(e_1) w_2(e_2))) \\
    &+ (\sum_{e', e'' \in \mathcal{E}^1_{v'_1}} + \sum_{e', e'' \in \mathcal{E}^1_{v'_2}} + \sum_{e' \in \mathcal{E}^1_{v'_1}, e'' \in \mathcal{E}^1_{v'_2}}) \\
    =& \sum_{e', e'' \in \mathcal{E}^2_v} + \sum_{e', e'' \in \mathcal{E}^1_v} + \sum_{e', e'' \in \mathcal{E}^2_{v'_1}} + \sum_{e', e'' \in \mathcal{E}^2_{v'_2}} \\
    &+ (\sum_{e' \in \mathcal{E}^2_{v'_2}} w_1(e')) (\sum_{e'' \in \mathcal{E}^2_{v'_1}} w_2(e'')) - (\sum_{e'' \in \mathcal{E}^2_{v'_1}} w_1(e'')) (\sum_{e' \in \mathcal{E}^2_{v'_2}} w_2(e')) \\
    &+ \sum_{e', e'' \in \mathcal{E}^1_{v'_1}} + \sum_{e', e'' \in \mathcal{E}^1_{v'_2}} + ~(\sum_{e' \in \mathcal{E}^1_{v'_1}} w_1(e')) (\sum_{e'' \in \mathcal{E}^1_{v'_2}} w_2(e'')) - (\sum_{e'' \in \mathcal{E}^1_{v'_2}} w_1(e'')) (\sum_{e' \in \mathcal{E}^1_{v'_1}} w_2(e')) \\
    =& \sum_{e', e'' \in \mathcal{E}^2_v} + \sum_{e', e'' \in \mathcal{E}^1_v} + \sum_{e', e'' \in \mathcal{E}^2_{v'_1}} + \sum_{e', e'' \in \mathcal{E}^2_{v'_2}} + \sum_{e', e'' \in \mathcal{E}^1_{v'_1}} + \sum_{e', e'' \in \mathcal{E}^1_{v'_2}}
\end{align*}
which is the contribution from $v,v'_1,v'_2$ when computing $\omega$ in $\tau$.

The contributions from the rest of the vertices stay the same, so $\mathcal{W}(f)$ preserves $\omega$.
\end{proof}

\begin{prop} \label{prop:ttmapsymplectic}
For any Markov train track partition $\mathcal{M}$, $\mathcal{W}(f_{\mathcal{M}}):\mathcal{W}(\tau_\mathcal{M}) \to \mathcal{W}(\tau_\mathcal{M})$ is an isomorphism that preserves $\omega$.
\end{prop}
\begin{proof}
This follows from \Cref{lemma:ttmapfactorization}, \Cref{lemma:subdivisionsymplectic}, \Cref{lemma:elementaryfoldingsymplectic}, and the fact that an isomorphism of train tracks induces an isomorphism which preserves the Thurston symplectic form on the weight space.
\end{proof}

In general, \Cref{lemma:subdivisionsymplectic} and \Cref{lemma:elementaryfoldingsymplectic} essentially say that we can modify a train track as much as we like using elementary moves when studying its weight space. Hence we make the following definition.

\begin{defn} \label{defn:equivalenttt}
Consider the equivalence relation on the set of all train tracks that is generated by there being an elementary move between two train tracks. We say that two train tracks in the same equivalence class are \textit{equivalent}.
\end{defn}

In particular we have the following lemma, which will be very useful when making computations in \Cref{sec:radical}.

\begin{lemma} \label{lemma:manystandardlyemb}
Let $\mathcal{X}=\mathcal{X}_I \sqcup \mathcal{X}_O$ and $\mathcal{X}=\mathcal{X}'_I \sqcup \mathcal{X}'_O$ be two partitions of the set of punctures of $S$. Then $\tau_{(\mathcal{X}_I,\mathcal{X}_O)}$ and $\tau_{(\mathcal{X}'_I,\mathcal{X}'_O)}$ are equivalent.
\end{lemma}
\begin{proof}
This follows immediately from \Cref{prop:innerouterfolding}.
\end{proof}

\section{Computing the radical} \label{sec:radical}

As mentioned in the last section, the Thurston symplectic form $\omega$ on the weight space $\mathcal{W}(\tau)$ is not actually a symplectic form in general. Its failure of being one is measured by its \textit{radical} $\rad(\omega) = \{w_0 \in \mathcal{W}(\tau): \omega(w,w_0)=0 \text{ for every $w \in \mathcal{W}(\tau)$} \}$. 

Now, if $\omega$ were symplectic, we would be able to show that $\mathcal{W}(f_{(\mathcal{X}_I,\mathcal{X}_O)})$ is reciprocal by just applying \Cref{prop:reciprocalfacts}(1). In general, to apply this approach, we need to understand what $\rad(\omega)$ is. In this section, we make the computations to determine this.

\subsection{Radical elements} \label{subsec:radicalelements}

\begin{defn} \label{defn:radicalelement}
Suppose $c$ is an even-pronged boundary component of $\tau$. Label the intervals in the complement of the cusps by $I_1,...,I_n$ in a cyclic order. We define the \textit{radical element} of $c$, denoted by $r_c$, to be the element of $\mathcal{W}(\tau)$ where we assign to each edge on $I_k$ a weight of $(-1)^k$. See \Cref{fig:radicalelement} left. Here we assign weights with multiplicity, meaning if an edge appears multiple times on possibly multiple $I_k$, then the weight we assign to it is the sum of the weights that is assigned to it each time it appears in some $I_k$.

In the degenerate case when $c$ is $0$-pronged, $r_c$ is the element of $\mathcal{W}(\tau)$ that assigns each edge on $c$ a weight of $1$.
\end{defn}

\begin{figure}
    \centering
    \fontsize{4pt}{4pt}\selectfont
    \resizebox{!}{4cm}{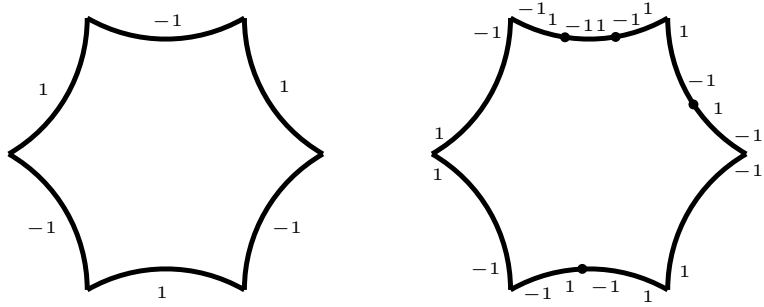}
    \caption{Left: A radical element $r_c$. Right: Showing that the terms cancel each other out when computing $\omega(w,r_c)$.}
    \label{fig:radicalelement}
\end{figure}

We remark that there is no canonical way to label the $I_k$, so $r_c$ can only be canonically defined up to a sign. This is not very significant for our purposes, we can simply fix a labelling where appropriate.

\begin{prop} \label{prop:radicalelementinradical}
Each $r_c$ lies in the radical of $\omega$.
\end{prop}
\begin{proof}
Let $w$ be an element of $\mathcal{W}(\tau)$, we have to check that $\omega(w,r_c)=0$. Let us first assume that $c$ is embedded in $\tau$ for simplicity. Let $v$ be a vertex on $c$ that lies at a cusp. Suppose half-edges $e_1$ and $e_2$ determine the cusp, where $e_1$ is to the left of $e_2$, and suppose that $r_c$ assigns the weight $(-1)^i$ to $e_i$. Then the contribution of $v$ in $\omega(w,r_c)$ is 
\begin{align*}
    & \sum_{e \text{ left of } e_1} (w(e) \times (-1) - w(e_1) \times 0) + (w(e) \times 1 - w(e_2) \times 0) \\
    &+ \sum_{e_2 \text{ left of } e} (w(e_2) \times 0 - w(e) \times 1) + (w(e) \times 1 - w(e_2) \times 1) + w(e_1) \times 1 - w(e_2) \times (-1) \\
    &= w(e_1) + w(e_2).
\end{align*}
If instead $r_c$ assigns the weight $(-1)^{i+1}$ to $e_i$, then the contribution of $v$ is $-w(e_1)-w(e_2)$.

Now let $v$ be a vertex on $c$ that does not lie at a cusp. Let $e_\beta \in \mathcal{E}^\beta_v$, $\beta=1,2$, be the two edges that lie on $c$, such that $e_1$ is the rightmost half-edge in $\mathcal{E}^1_v$ and $e_2$ is the leftmost half-edge in $\mathcal{E}^2_v$. Suppose that $r_c$ assigns the weight $1$ to $e_1$ and $e_2$. Then the contribution of $v$ in $\omega(w,r_c)$ is 
\begin{align*}
    & \sum_{e \text{ left of } e_1} (w(e) \times 1 - w(e_1) \times 0) + \sum_{e_2 \text{ left of } e} (w(e_2) \times 0 - w(e) \times 1) \\
    &= (\sum_{e \in \mathcal{E}^1_v} w(e) - w(e_1)) - (\sum_{e \in \mathcal{E}^2_v} w(e) - w(e_2)) \\
    &= w(e_2) - w(e_1).
\end{align*}
If instead $r_c$ assigns the weight $-1$ to $e_1$ and $e_2$, then the contribution of $v$ is $w(e_1)-w(e_2)$.

When adding together the contributions from all vertices, the terms cancel out in pairs, giving us $0$. We schematically illustrate how the cancelling occurs in \Cref{fig:radicalelement} right.

If $c$ is not embedded, that is, some switch of $\tau$ meets $c$ more than once, then the above computation still holds with some more careful book-keeping. 
For example, if a vertex $v$ lies on two cusps of $c$, say determined by pairs of half-edges $(e^1_1, e^1_2)$ and $(e^2_1, e^2_2)$ respectively, then one can check that the contribution of $v$ is $\pm (w(e^1_1)+w(e^1_2)) \pm (w(e^2_1)+w(e^2_2))$, depending on whether $r_c$ assigns $\pm (-1)^i$ to $e^1_i$ and $e^2_i$ respectively. 
We let the reader fill in the details. Alternatively, one can pass to a finite cover of $\tau$ where $c$ is embedded and perform the computation there.
\end{proof}

The radical elements behave nicely with respect to elementary moves. To state this precisely, first note that if $\tau'$ is obtained from $\tau$ by an elementary move, then the boundary components of $\tau'$ are in natural one-to-one correspondence with those of $\tau$. 

\begin{prop} \label{prop:elementarymovesradicalelement}
Suppose $c'$ is a boundary component of $\tau'$ that corresponds to a boundary component $c$ of $\tau$. Then $r_c$ maps to $r_{c'}$ under the isomorphism $\mathcal{W}(\tau) \to \mathcal{W}(\tau')$ induced by the elementary move.
\end{prop}
\begin{proof}
The proposition is clear for subdivision moves, and clear for elementary folding moves if, in the notation of \Cref{defn:elementaryfolding}, the cusp determined by $(e_1,e_2)$ is not in $c$. If the cusp is in $c$, then, up to a sign, $r_c$ assigns the weight $(-1)^i$ to $e_i$, hence upon folding these add up to $0$, which is the weight assigned by $r_{c'}$.
\end{proof}

\begin{prop} \label{prop:radicalreciprocal}
Let $\mathcal{X}=\mathcal{X}_I \sqcup \mathcal{X}_O$ be a partition of the set of punctures into two nonempty $f$-invariant sets. Then the train track map $f_{(\mathcal{X}_I,\mathcal{X}_O)}:\tau_{(\mathcal{X}_I,\mathcal{X}_O)} \to \tau_{(\mathcal{X}_I,\mathcal{X}_O)}$ preserves $\spn \{r_c\}$ and acts on it via a reciprocal map.
\end{prop}
\begin{proof}
\Cref{lemma:ttmapfactorization} and \Cref{prop:elementarymovesradicalelement} imply that $\mathcal{W}(f_{(\mathcal{X}_I,\mathcal{X}_O)})$ preserves $\spn \{r_c\}$.

For the second part of the statement, we claim that $\{r_c\}$ are linearly independent unless all punctures are even-pronged, in which case there is only at most one relation of the form $\sum_c \pm r_c = 0$.

Notice that by \Cref{prop:elementarymovesradicalelement}, the validity of the claim is invariant under equivalence of train tracks. Hence by \Cref{lemma:manystandardlyemb}, we can assume that $\mathcal{X}_O$ consists of a single element corresponding to $c_O \in \partial \tau$.

If there is an odd-pronged puncture, we can further assume that $c_O$ is odd-pronged. In this case, all even-pronged boundary components of the train track are disjoint, hence $\{r_c\}$ are linearly independent.

If all punctures are even-pronged, then all even-pronged boundary components except for $c_O$ are disjoint. For an edge $e$ on an infinitesimal polygon $c$, $r_c$ and $r_{c_O}$, and only $r_c$ and $r_{c_O}$, are nonzero on it. In fact, they assign weights $\pm 1$ to $e$, so the ratio of their coefficients is uniquely determined to be one of $\pm 1$. This proves the claim.

Returning to the proof of the second part of the proposition, if $\{r_c\}$ are linearly independent, then by \Cref{lemma:ttmapfactorization} and \Cref{prop:elementarymovesradicalelement}, $f_{(\mathcal{X}_I,\mathcal{X}_O)}$ acts on $\spn \{r_c\}$ by a signed permutation matrix, so this follows from \Cref{prop:reciprocalfacts}(2).

If $\{r_c\}$ are not linearly independent, then consider the commutative diagram
\begin{center}
\begin{tikzcd}
0 \arrow[r] & \langle \sum \pm r_c \rangle \arrow[r] \arrow[d, "\pm 1"] & \mathbb{R}^\mathcal{X} \arrow[r] \arrow[d, "P"] & \spn \{r_c\} \arrow[r] \arrow[d, "\mathcal{W}(f_{(\mathcal{X}_I,\mathcal{X}_O)})"] & 0 \\
0 \arrow[r] & \langle \sum \pm r_c \rangle \arrow[r] & \mathbb{R}^\mathcal{X} \arrow[r] & \spn \{r_c\} \arrow[r] & 0
\end{tikzcd}
\end{center}
where $P$ is a signed permutation matrix, hence is reciprocal, as in the last case. By \Cref{prop:reciprocalfacts}(3), $\mathcal{W}(f_{(\mathcal{X}_I,\mathcal{X}_O)})$ acts on $\spn \{r_c\}$ via a reciprocal map.
\end{proof}

The significance of \Cref{prop:radicalreciprocal} comes from the following result, whose proof will occupy the rest of this section.

\begin{prop} \label{prop:radical}
Let $\mathcal{X}=\mathcal{X}_I \sqcup \mathcal{X}_O$ be a partition of the set of punctures into two nonempty sets. For the train track $\tau_{(\mathcal{X}_I,\mathcal{X}_O)}$, we have 
\begin{equation} \label{eq:radical}
    \rad(\omega)=\spn \{r_c\}
\end{equation}
where $c$ ranges over all even-pronged boundary components of $\tau$.
\end{prop}

The strategy to proving \Cref{prop:radical} is to modify the train track into a convenient form before making concrete computations. This strategy was already used in the proof of \Cref{prop:radicalreciprocal}, where we modified the train track up to equivalence. This applies equally well in the setting of \Cref{prop:radical}, as we have the following observation.

\begin{lemma} \label{lemma:elementarymoveradical}
Suppose $\tau$ and $\tau'$ are equivalent, then (\ref{eq:radical}) holds for $\tau$ if and only if it holds for $\tau'$.
\end{lemma}
\begin{proof}
This follows from \Cref{lemma:elementaryfoldingsymplectic} and \Cref{prop:elementarymovesradicalelement}.
\end{proof}

Let us call a standardly embedded train track \textit{floral} if it has only one infinitesimal polygon. Visually, the single infinitesimal polygon forms the pistil while the real edges forms the petals of a flower. See \Cref{fig:oddnice} for an example. Hence using \Cref{lemma:manystandardlyemb} and \Cref{lemma:elementarymoveradical}, we can assume that $\mathcal{X}_I$ consists of a single element, that is, $\tau_{(\mathcal{X}_I,\mathcal{X}_O)}$ is floral.

Here the proof divides into two cases. Case 1 is if there is an odd-pronged puncture. In this case we can take the single infinitesimal polygon of $\tau_{(\mathcal{X}_I,\mathcal{X}_O)}$ to be odd-pronged. We show that \textit{admissible deletion} of real edges preserves (\ref{eq:radical}) (\Cref{lemma:combineprong} and \Cref{lemma:splitprong}). This allows us to modify our train track into a simple form where we can explicitly verify (\ref{eq:radical}). This case is tackled in \Cref{subsec:oddprong}.

Case 2 is if all punctures are even-pronged. In this case, we would like to repeat the reasoning in case 1, but here we must first pass to the orientable cover (\Cref{lemma:covering}) before the arguments work. This case is tackled in \Cref{subsec:evenprong}.

\subsection{Case 1: There is an odd-pronged puncture}  \label{subsec:oddprong}

As explained above, in this case we can assume that $\tau_{(\mathcal{X}_I,\mathcal{X}_O)}$ is floral with an odd-pronged infinitesimal polygon, or \textit{odd floral} for short. See \Cref{fig:oddnice} for an example.

\begin{figure}
    \centering
    \resizebox{!}{4cm}{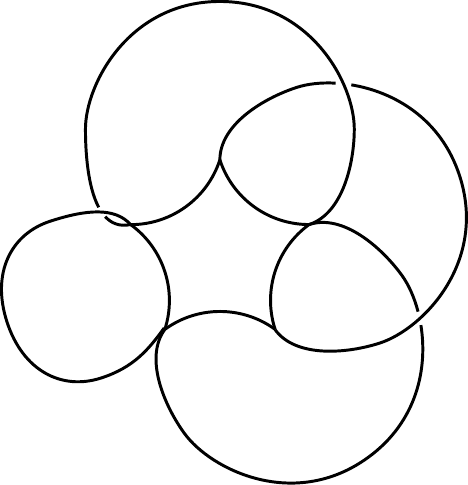}
    \caption{An odd floral train track.}
    \label{fig:oddnice}
\end{figure}

For odd floral train tracks, we can construct some convenient elements in the weight space for which we can test elements of $\rad(\omega)$ against. 

\begin{constr} \label{constr:realedgeelement}
Let $\tau$ be an odd floral train track. Let $e$ be a real edge in $\tau$. Let $c_I$ be the infinitesimal polygon of $\tau$. Label the vertices of $c_I$ by $v_1,..,v_n$ in a cyclic order such that the endpoints of $e$ lie on $v_1$ and $v_k$. Up to flipping the ordering, we can assume that $k$ is even. Also label the edges of $c_I$ by $e_1,...,e_n$ such that $e_i$ connects $v_i$ to $v_{i+1}$. 

We define $w_e \in \mathcal{W}(\tau)$ to be the element that assigns $(-1)^{i+1}$ to $e_i$ for $i=1,...,k-1$, assigns $1$ to $e$, and assigns $0$ to all other edges.
\end{constr}

We now introduce the operation of admissible deletions. This is a general way of modifying train tracks, but is particularly useful for proving \Cref{prop:radical} when applied within the realm of floral train tracks, as we will see.

\begin{constr} \label{constr:delete}
Let $\tau$ be a train track and $e$ be an edge of $\tau$. If $e$ is not the only half-edge at both of its end points, then we can delete $e$ from $\tau$ to get a new train track $\tau'$. See \Cref{fig:delete}. We call this operation an \textit{admissible deletion} (of the edge $e$).

\begin{figure}
    \centering
    \fontsize{16pt}{16pt}\selectfont
    \resizebox{!}{1.5cm}{
\begingroup%
  \makeatletter%
  \providecommand\color[2][]{%
    \errmessage{(Inkscape) Color is used for the text in Inkscape, but the package 'color.sty' is not loaded}%
    \renewcommand\color[2][]{}%
  }%
  \providecommand\transparent[1]{%
    \errmessage{(Inkscape) Transparency is used (non-zero) for the text in Inkscape, but the package 'transparent.sty' is not loaded}%
    \renewcommand\transparent[1]{}%
  }%
  \providecommand\rotatebox[2]{#2}%
  \newcommand*\fsize{\dimexpr\f@size pt\relax}%
  \newcommand*\lineheight[1]{\fontsize{\fsize}{#1\fsize}\selectfont}%
  \ifx\svgwidth\undefined%
    \setlength{\unitlength}{521.5600885bp}%
    \ifx\svgscale\undefined%
      \relax%
    \else%
      \setlength{\unitlength}{\unitlength * \real{\svgscale}}%
    \fi%
  \else%
    \setlength{\unitlength}{\svgwidth}%
  \fi%
  \global\let\svgwidth\undefined%
  \global\let\svgscale\undefined%
  \makeatother%
  \begin{picture}(1,0.13142113)%
    \lineheight{1}%
    \setlength\tabcolsep{0pt}%
    \put(0,0){\includegraphics[width=\unitlength,page=1]{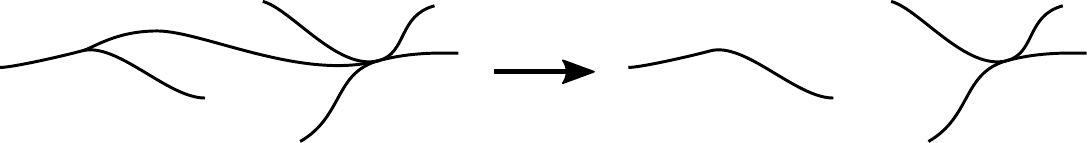}}%
    \put(0.19360118,0.10190867){\color[rgb]{0,0,0}\makebox(0,0)[lt]{\lineheight{1.25}\smash{\begin{tabular}[t]{l}$e$\end{tabular}}}}%
  \end{picture}%
\endgroup%
}
    \caption{Admissible deletion of an edge from a train track.}
    \label{fig:delete}
\end{figure}

Suppose $e$ meets boundary components $c_1$ and $c_2$ of $\tau$. If $c_1 \neq c_2$, then deleting $e$ combines $c_1$ and $c_2$ into one boundary component $c'$ of $\tau'$. If $c_i$ is $n_i$-pronged, then $c'$ is $(n_1+n_2-2)$-pronged. On the other hand, if $c_1=c_2=:c$, then deleting $e$ splits $c$ into two boundary components $c'_1$ and $c'_2$. In this case, if $c'_i$ is $n'_i$-pronged, then $c$ is $(n'_1+n'_2+2)$-pronged 

If $w'$ is a system of weights on $\tau'$, then we can define a system of weights $w$ on $\tau$ by $w(e')=w'(e')$ for $e' \neq e$ and $w(e)=0$. Conversely, if $w$ is a system of weights on $\tau$ with $w(e)=0$, then by restricting $w$ to the remaining edges, we get a system of weights $w'$ on $\tau'$. This allows us to identify $\mathcal{W}(\tau')$ as a subspace of $\mathcal{W}(\tau)$. Moreover, this inclusion preserves the Thurston symplectic form $\omega$.
\end{constr}

Our next task is to show that admissible deletion of a real edge from an odd floral train track preserves (\ref{eq:radical}). We split into two cases according to whether the admissible deletion combines two boundary components into one (\Cref{lemma:combineprong}) or splits one boundary component into two (\Cref{lemma:splitprong}).

\begin{lemma} \label{lemma:combineprong}
Suppose $\tau$ and $\tau'$ are odd floral train tracks where $\tau'$ is obtained from admissible deletion of a real edge $e$ from $\tau$. Suppose deleting $e$ combines two boundary components into one. Then (\ref{eq:radical}) holds for $\tau$ if and only if it holds for $\tau'$. 
\end{lemma}
\begin{proof}
There are two cases here. Case 1 is if at least one of $c_i$, say $c_1$, is even-pronged. In this case, we claim that $\mathcal{W}(\tau)=\mathcal{W}(\tau') \oplus \langle r_{c_1} \rangle$. Indeed, since $c_1 \neq c_2$, $r_{c_1}(e) \neq 0$, so $\mathcal{W}(\tau') \cap \langle r_{c_1} \rangle = 0$. Meanwhile, $\dim(\mathbb{R}^{\mathcal{E}(\tau)}/\mathbb{R}^{\mathcal{E}(\tau')})=1$ and $\mathcal{W}(\tau)/\mathcal{W}(\tau')$ can be identified with a subspace of $\mathbb{R}^{\mathcal{E}(\tau)}/\mathbb{R}^{\mathcal{E}(\tau')}$, so $\dim(\mathcal{W}(\tau)/\mathcal{W}(\tau')) \leq 1$, which proves the claim.

To prove the lemma in this case, first suppose that (\ref{eq:radical}) holds for $\tau$. Then it follows from the claim for any $w'_0 \in \rad(\omega)$ in $\mathcal{W}(\tau')$, $w'_0 \in \rad(\omega)$ in $\mathcal{W}(\tau)$ as well. Hence $w'_0 = \sum_c a_c r_c$. But $w'_0(e)=0$ so we have $a_{c_1}=0$ if $c_2$ is odd-pronged, and $a_{c_1}=a_{c_2}$ if $c_2$ is even-pronged (under an appropriate choice of signs for $r_{c_i}$). Together with the fact that $\pm r_{c'}=r_{c_1}+r_{c_2}$ when $c_2$ is even-pronged, this shows that $w'_0 = \sum_{c'} a_{c'} r_{c'}$ in $\mathcal{W}(\tau')$. 

Conversely, suppose that (\ref{eq:radical}) holds for $\tau'$. Then for any $w_0 \in \rad(\omega)$ in $\mathcal{W}(\tau)$, we can consider $w_0-w_0(e)r_{c_1}$. This lies in $\mathcal{W}(\tau')$ hence lies in $\rad(\omega)$ in $\mathcal{W}(\tau')$. By hypothesis, we then have $w_0-w_0(e)r_{c_1} \in \spn\{r_{c'}\}$, which together with the fact again that $\pm r_{c'}=r_{c_1}+r_{c_2}$ when $c_2$ is even-pronged, we have $w_0 \in \spn \{r_c\}$ in $\mathcal{W}(\tau)$. 

Case 2 is if both $c_i$ are odd-pronged. In this case $c'$ is even-pronged. It can be shown by the same reasoning as in the last case that $\mathcal{W}(\tau)=\mathcal{W}(\tau') \oplus \langle w_e \rangle$, where $w_e$ is defined in \Cref{constr:realedgeelement}.

We compute $\omega(r_{c'},w)$ for $w \in \mathcal{W}(\tau)$. Suppose first for simplicity that $e$ meets $c'$ in two of its cusps $p$ and $p'$. For $i=1,2$, let $e_i$ be the half-edge at $p$ which is adjacent to $e$ on $c_i$, and let $e'_i$ be the half-edge at $p'$ which is adjacent to $e$ on $c_i$. Without loss of generality suppose that $e_1$ lies to the left of $e$ and $r_{c'}$ assigns $(-1)^i$ to $e_i$, then since $c_i$ are odd-pronged, $r_{c'}$ assigns $(-1)^{i+1}$ to $e'_i$. See \Cref{fig:combineprong} left.

\begin{figure}
    \centering
    \fontsize{6pt}{6pt}\selectfont
    \resizebox{!}{4.5cm}{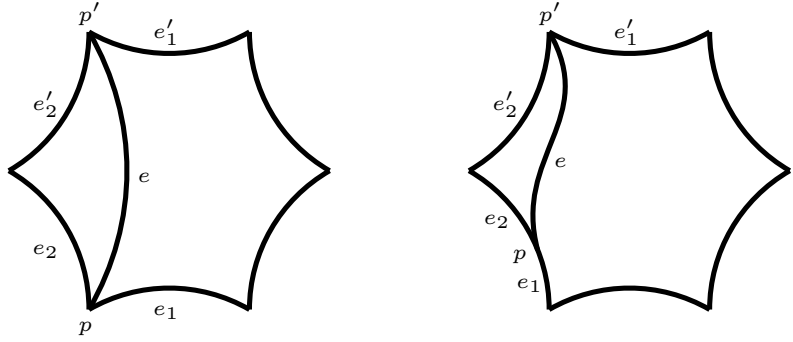}
    \caption{The situation in case 2 of \Cref{lemma:combineprong}. Left: if $e$ meets $c'$ in two of its cusps. Right: if $e$ meets $c'$ away from its cusps.}
    \label{fig:combineprong}
\end{figure}

By the computation made in \Cref{prop:radicalelementinradical}, the total contribution to $\omega(r_{c'},w)$ from the vertices on $c'$ aside from $p$ and $p'$ is $w(e_1)+w(e_2)+w(e'_1)+w(e'_2)$. The contribution from $p$ is $-w(e)-w(e_2)-w(e)-w(e_1)$, and the contribution from $p'$ is $-w(e)-w(e'_2)-w(e)-w(e'_1)$. Adding these together, we see that $\omega(r_{c'},w)=-4w(e)$.

If one of the endpoints of $e$ lies in the complement of the cusps on $c'$ instead, we define $e_i$ as above, but with $p$ or $p'$ being the endpoint of $e$. Suppose $p$ is such an endpoint, say $e_1$ does not determine a cusp with $e$ on $c_1$, and suppose that $r_{c'}$ assigns $1$ to $e_1$ and $e_2$. See \Cref{fig:combineprong} right. Then the the total contribution to $\omega(r_{c'},w)$ from the vertices on $c'$ aside from $p$ and $p'$ is $-w(e_1)+w(e_2)+w(e'_1)+w(e'_2)$. The contribution from $p$ is $(-\sum_{e \in \mathcal{E}^\beta_p} w(e) + w(e_1)) + (\sum_{e \in \mathcal{E}^\beta_p} w(e) - w(e_2) - 2w(e))$, and the contribution from $p'$ is $-w(e)-w(e'_2)-w(e)-w(e'_1)$. So we still have $\omega(r_{c'},w)=-4w(e)$.

We are now ready to prove the lemma in this case. Suppose (\ref{eq:radical}) holds for $\tau$. Let $w'_0 \in \rad(\omega)$ in $\mathcal{W}(\tau')$. Let $a=\omega(w'_0,w_e)$, then $w'_0 + \frac{a}{4} r_{c'} \in \rad(\omega)$ in $\mathcal{W}(\tau')$ and $\omega(w'_0 + \frac{a}{4} r_{c'},w_e) = a + \frac{a}{4} (-4w_e(e)) = 0$, so $w'_0 + \frac{a}{4} r_{c'} \in \rad(\omega)$ in $\mathcal{W}(\tau)$, implying that $w'_0 \in \spn \{r_c\} - \frac{a}{4} r_{c'} \subset \spn \{r_{c'}\}$ in $\mathcal{W}(\tau')$.

Conversely, suppose (\ref{eq:radical}) holds for $\tau'$. Let $w_0 \in \rad(\omega)$ in $\mathcal{W}(\tau)$. Then $0 = \omega(r_{c'},w_0) = -4w_0(e)$. Hence we can treat $w_0$ as an element of $\mathcal{W}(\tau')$, hence an element of $\rad(\omega)$ in $\mathcal{W}(\tau')$. Thus $w_0 \in \spn \{r_{c'}\}$ in $\mathcal{W}(\tau')$. Say $w_0 = \sum_{c'} a_{c'} r_{c'}$. To establish that $w_0 \in \spn \{r_c\}$ in $\mathcal{W}(\tau)$, we need to show that $a_{c'}=0$. This follows since $0 = \omega(w_0,w_e) = \omega(a_{c'}r_{c'},w_e) = -4a_{c'}$.
\end{proof}

\begin{lemma} \label{lemma:splitprong}
Suppose $\tau$ and $\tau'$ are odd floral train tracks where $\tau'$ is obtained from an admissible deletion of a real edge $e$ from $\tau$. Suppose deleting $e$ splits a boundary component into two. Then (\ref{eq:radical}) holds for $\tau$ if and only if it holds for $\tau'$.
\end{lemma}
\begin{proof}
There are three cases here. The proof of each case is similar to one of the cases in \Cref{lemma:combineprong}.

Case 1 is if $c$ is odd-pronged. In this case one of $c'_i$, say $c'_1$, is even-pronged while the other is odd-pronged.

We follow the strategy of case 2 in \Cref{lemma:combineprong}. It can be shown as before that $\mathcal{W}(\tau)=\mathcal{W}(\tau') \oplus \langle w_e \rangle$. Also, one can compute that $\omega(r_{c'_1},w)=-2w(e)$ for $w \in \mathcal{W}(\tau)$ (for an appropriate choice of sign for $r_{c'_1}$). 

Suppose (\ref{eq:radical}) holds for $\tau$. Let $w'_0 \in \rad(\omega)$ in $\mathcal{W}(\tau')$. Let $a=\omega(w'_0,w_e)$, then $w'_0 + \frac{a}{2} r_{c'_1} \in \rad(\omega)$ in $\mathcal{W}(\tau')$ and $\omega(w'_0 + \frac{a}{2} r_{c'_1},w_e) = a + \frac{a}{2} (-2w_e(e)) = 0$, so $w'_0 + \frac{a}{2} r_{c'_1} \in \rad(\omega)$ in $\mathcal{W}(\tau)$, implying that $w'_0 \in \spn \{r_c\} - \frac{a}{2} r_{c'_1} \subset \spn \{r_{c'}\}$ in $\mathcal{W}(\tau')$.

Conversely, suppose (\ref{eq:radical}) holds for $\tau'$. Let $w_0 \in \rad(\omega)$ in $\mathcal{W}(\tau)$. Then $0 = \omega(r_{c'_1},w_0) = -2w_0(e)$. Hence we can treat $w_0$ as an element of $\mathcal{W}(\tau')$, hence an element of $\rad(\omega)$ in $\mathcal{W}(\tau')$. Thus $w_0 \in \spn \{r_{c'_1}\}$ in $\mathcal{W}(\tau')$. Say $w_0 = \sum_{c'} a_{c'} r_{c'}$. To establish that $w_0 \in \spn \{r_c\}$ in $\mathcal{W}(\tau)$, we need to show that $a_{c'_1}=0$. This follows since $0 = \omega(w_0,w_e) = \omega(a_{c'_1}r_{c'_1},w_e) = -2a_{c'_1}$.

Case 2 is if $c$ is even-pronged while $c'_1$ and $c'_2$ are odd-pronged.

We follow the strategy of case 1 in \Cref{lemma:combineprong}. It can be shown as before that $\mathcal{W}(\tau)=\mathcal{W}(\tau') \oplus \langle r_c \rangle$, since $c'_i$ being odd-pronged implies that $r_c(e)=2$ (under an appropriate choice of sign for $r_c$).

Suppose that (\ref{eq:radical}) holds for $\tau$. For any $w'_0 \in \rad(\omega)$ in $\mathcal{W}(\tau')$, $w'_0 \in \rad(\omega)$ in $\mathcal{W}(\tau)$ as well. Hence $w'_0 = \sum_c a_c r_c$. But $w'_0(e)=0$ so we have $a_c=0$. This shows that $w'_0 = \sum_{c'} a_{c'} r_{c'}$ in $\mathcal{W}(\tau')$. 

Conversely, suppose that (\ref{eq:radical}) holds for $\tau'$. Then for any $w_0 \in \rad(\omega)$ in $\mathcal{W}(\tau)$, we can consider $w_0-\frac{w_0(e)}{2}r_c$. This lies in $\mathcal{W}(\tau')$ hence lies in $\rad(\omega)$ in $\mathcal{W}(\tau')$. By hypothesis, we then have $w_0-\frac{w_0(e)}{2}r_c \in \spn \{r_{c'}\}$, which gives $w_0 \in \spn \{r_c\}$ in $\mathcal{W}(\tau)$. 

Finally, case 3 is if $c$ is even-pronged while $c'_1$ and $c'_2$ are even-pronged.

We follow the strategy of case 2 in \Cref{lemma:combineprong}. It can be shown as before that $\mathcal{W}(\tau)=\mathcal{W}(\tau') \oplus \langle w_e \rangle$. Also, one can compute that $\omega(r_{c'_1},w)=-2w(e)$ for $w \in \mathcal{W}(\tau)$. Using the fact that $r_c=r_{c'_1}+r_{c'_2}$, we have $\omega(r_{c'_2},w)=2w(e)$ (for appropriate signs for $r_c,r_{c'_1},r_{c'_2}$).

Suppose (\ref{eq:radical}) holds for $\tau$. Let $w'_0 \in \rad(\omega)$ in $\mathcal{W}(\tau')$. Let $a=\omega(w'_0,w_e)$, then $w'_0 + \frac{a}{2} r_{c'_1} \in \rad(\omega)$ in $\mathcal{W}(\tau')$ and $\omega(w'_0 + \frac{a}{2} r_{c'_1},w_e) = a + \frac{a}{2} (-2w_e(e)) = 0$, so $w'_0 + \frac{a}{2} r_{c'_1} \in \rad(\omega)$ in $\mathcal{W}(\tau)$, implying that $w'_0 \in \spn \{r_c\} - \frac{a}{2} r_{c'_1} \subset \spn \{r_{c'}\}$ in $\mathcal{W}(\tau')$.

Conversely, suppose (\ref{eq:radical}) holds for $\tau'$. Let $w_0 \in \rad(\omega)$ in $\mathcal{W}(\tau)$. Then $0 = \omega(r_{c'_1},w_0) = -2w_0(e)$. Hence we can treat $w_0$ as an element of $\mathcal{W}(\tau')$, hence an element of $\rad(\omega)$ in $\mathcal{W}(\tau')$. Thus $w_0 \in \spn \{r_{c'_1}\}$ in $\mathcal{W}(\tau')$. Say $w_0 = \sum_{c'} a_{c'} r_{c'}$. To establish that $w_0 \in \spn \{r_c\}$ in $\mathcal{W}(\tau)$, we need to show that $a_{c'_1}=a_{c'_2}$. This follows since $0 = \omega(w_0,w_e) = \omega(a_{c'_1}r_{c'_1}+a_{c'_2}r_{c'_2},w_e) = -2a_{c'_1}+2a_{c'_2}$.
\end{proof}

Now notice that if $\tau_1$ and $\tau_2$ are two floral train tracks whose unique infinitesimal polygons have the same number of prongs, then they can be related by a sequence of admissible deletion of real edges. This is because upon fixing an identification of their infinitesimal polygons, one can first add all the real edges of $\tau_2$ to $\tau_1$ (which is the reverse of deleting those edges), with the cyclic ordering of half-edges determined by, say, placing all the half-edges of $\tau_2$ to the left of those of $\tau_1$, then deleting the real edges of $\tau_1$. See \Cref{fig:modify} for an example of this procedure. 

\begin{figure}
    \centering
    \resizebox{!}{4cm}{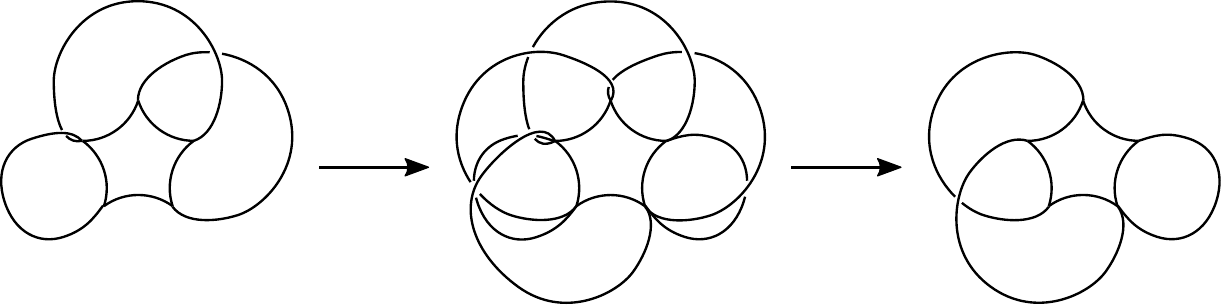}
    \caption{Any two floral train tracks whose unique infinitesimal polygons have the same number of prongs are related by a sequence of removals of real edges.}
    \label{fig:modify}
\end{figure}

Hence \Cref{lemma:combineprong} and \Cref{lemma:splitprong} implies that in order to prove \Cref{prop:radical} in this case, we can simply establish (\ref{eq:radical}) for one floral train track with an $n$-pronged infinitesimal polygon for every odd $n$.

We choose the floral train track $\tau_n$ illustrated in \Cref{fig:eztt} left to this end. That is, $\tau_n$ has an $n$-pronged infinitesimal polygon, say with vertices labelled $v_1,...,v_n$ in a cyclic way, and $n$ real edges $e_1,...,e_n$, each $e_i$ having endpoints on $v_i$ and $v_{i+1}$, with $e_{i-1}$ lying to the left of $e_i$ at $v_i$ for every $i$. 
Then $\tau_n$ has $n+2$ boundary components, two of them $n$-pronged and $n$ of them $0$-pronged. It is straightforward to check that $\mathcal{W}(\tau_n)$ is generated by $\{r_c\}$ for all the $0$-pronged boundary components $c$. Hence $\rad(\omega)=\spn \{r_c\}$ must hold. This concludes the proof of \Cref{prop:radical} in this case.

\begin{figure}
    \centering
    \fontsize{20pt}{20pt}\selectfont
    \resizebox{!}{4cm}{
\begingroup%
  \makeatletter%
  \providecommand\color[2][]{%
    \errmessage{(Inkscape) Color is used for the text in Inkscape, but the package 'color.sty' is not loaded}%
    \renewcommand\color[2][]{}%
  }%
  \providecommand\transparent[1]{%
    \errmessage{(Inkscape) Transparency is used (non-zero) for the text in Inkscape, but the package 'transparent.sty' is not loaded}%
    \renewcommand\transparent[1]{}%
  }%
  \providecommand\rotatebox[2]{#2}%
  \newcommand*\fsize{\dimexpr\f@size pt\relax}%
  \newcommand*\lineheight[1]{\fontsize{\fsize}{#1\fsize}\selectfont}%
  \ifx\svgwidth\undefined%
    \setlength{\unitlength}{439.46407388bp}%
    \ifx\svgscale\undefined%
      \relax%
    \else%
      \setlength{\unitlength}{\unitlength * \real{\svgscale}}%
    \fi%
  \else%
    \setlength{\unitlength}{\svgwidth}%
  \fi%
  \global\let\svgwidth\undefined%
  \global\let\svgscale\undefined%
  \makeatother%
  \begin{picture}(1,0.55087001)%
    \lineheight{1}%
    \setlength\tabcolsep{0pt}%
    \put(0,0){\includegraphics[width=\unitlength,page=1]{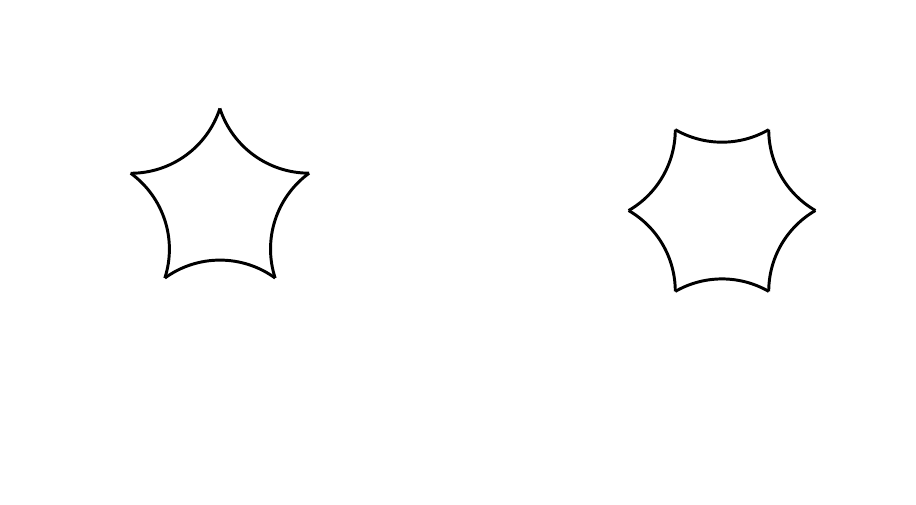}}%
    \put(0.74262874,0.00721982){\color[rgb]{0,0,0}\makebox(0,0)[lt]{\lineheight{1.25}\smash{\begin{tabular}[t]{l}$n=6$\end{tabular}}}}%
    \put(0.18314795,0.00721982){\color[rgb]{0,0,0}\makebox(0,0)[lt]{\lineheight{1.25}\smash{\begin{tabular}[t]{l}$n=5$\end{tabular}}}}%
    \put(0,0){\includegraphics[width=\unitlength,page=2]{eztt.pdf}}%
  \end{picture}%
\endgroup%
}
    \caption{The train tracks we use to demonstrate (\ref{eq:radical}) at the end of our modifications.}
    \label{fig:eztt}
\end{figure}

\subsection{Case 2: All punctures are even-pronged} \label{subsec:evenprong}

As explained under \Cref{prop:radical}, the proof in this case is very similar to the last case. The only difference is the preparatory step of taking a 2-fold cover.

\begin{defn} \label{defn:orientable}
A train track is said to be \textit{orientable} if its edges can be oriented in a way such that at each switch $v$, all the edges in one $\mathcal{E}^\beta_v$ are oriented into $v$ while all the edges in $\mathcal{E}^{\beta+1}_v$ are oriented out of $v$.
\end{defn}

For example, the train tracks $\tau_{\mathcal{X}_I,\mathcal{X}_O}$ we have been considering are orientable if and only if the unstable foliation $\ell^u$ is orientable in the usual sense.

Now, up to passing to a 2-fold cover, $\ell^u$ can always be made orientable. More specifically, one can define a 1-cocycle $\alpha \in H^1(S, \mathbb{Z}/2)$ by $\alpha(\gamma)=0$ if and only if $\ell^u$ is orientable in a neighborhood of a curve $\gamma$. Then the \textit{orientable 2-fold cover} $\widetilde{S}$ is determined by $\alpha$. By embedding $\tau$ in $S$ as usual, $\alpha$ also determines the \textit{orientable 2-fold cover} $\widetilde{\tau} \to \tau$.

The special property when there are no odd-pronged punctures is that $\alpha(c)=0$ for every boundary component $c$ of $\tau$. So if $\widetilde{\tau}$ is the orientable 2-fold cover, then every boundary component of $\tau$ lifts homemorphically to $\widetilde{\tau}$. Correspondingly, we have the 2-to-1 map $\widetilde{\mathcal{X}} \to \mathcal{X}$. In particular, we can lift a partition $\mathcal{X}=\mathcal{X}_I \sqcup \mathcal{X}_O$ to $\widetilde{\mathcal{X}}=\widetilde{\mathcal{X}_I} \sqcup \widetilde{\mathcal{X}_O}$.

By naturality of \Cref{constr:ttpartition}, $\tau_{(\widetilde{\mathcal{X}_I}, \widetilde{\mathcal{X}_O})}$ is the orientable 2-fold cover of $\tau_{(\mathcal{X}_I, \mathcal{X}_O)}$. Hence by the following lemma, we can assume that the $\tau_{(\mathcal{X}_I, \mathcal{X}_O)}$ we are dealing with is orientable.

\begin{lemma} \label{lemma:covering}
Let $\pi:\widetilde{\tau} \to \tau$ be a finite normal covering of train tracks such that each boundary component of $\tau$ lifts homeomorphically to $\widetilde{\tau}$. Then (\ref{eq:radical}) holds for $\tau$ if (\ref{eq:radical}) holds for $\widetilde{\tau}$.
\end{lemma}
\begin{proof}
Let $d$ be the degree of the covering and let $G$ be the group of deck transformations. We can define operators $\pi_*:\mathcal{W}(\widetilde{\tau}) \to \mathcal{W}(\tau)$, $\pi^*:\mathcal{W}(\tau) \to \mathcal{W}(\widetilde{\tau})$, and $s:\mathcal{W}(\widetilde{\tau}) \to \mathcal{W}(\widetilde{\tau})$ by
\begin{align*}
    (\pi_*(\widetilde{w}))(e) &= \sum_{\pi(\widetilde{e})=e} \widetilde{w}(e) \\
    (\pi^*(w))(\widetilde{e}) &= w(\pi(\widetilde{e})) \\
    (s(\widetilde{w}))(\widetilde{e}) &= \sum_{g \in G} \widetilde{w}(g\widetilde{e})
\end{align*}
We have the following properties.
\begin{itemize}
    \item $\pi^*(r_c)=\sum_{\pi(\widetilde{c})=c} r_{\widetilde{c}}$
    \item $\pi_* \pi^* (w)= dw$
    \item $\omega(w_1,\pi_*(\widetilde{w_2}))=\omega(\pi^*(w_1),\widetilde{w_2})$
    \item $s\pi^*(w)=d \pi^*(w)$
\end{itemize}
which imply that $\pi^*$ is injective and $\pi^*(\rad(\omega)) \subset \rad(\omega)$.

Now suppose that (\ref{eq:radical}) holds for $\widetilde{\tau}$. Let $w_0 \in \rad(\omega)$ in $\mathcal{W}(\tau)$. $\pi^*(w_0) \in \rad(\omega)$ so $\pi^*(w_0)=\sum_{\widetilde{c}} a_{\widetilde{c}} r_{\widetilde{c}}$. Hence
$$d \pi^*(w_0) = s \pi^*(w_0) = \sum_{\widetilde{c}} (\sum_{g \in G} a_{g\widetilde{c}}) r_{\widetilde{c}} = \pi^*(\sum_c (\sum_{\pi(\widetilde{c})=c} a_{\widetilde{c}}) r_c)$$
which implies that $w_0 = \frac{1}{d} \sum_c (\sum_{\pi(\widetilde{c})=c} a_{\widetilde{c}}) r_c \in \spn \{r_c\}$ by injectivity of $\pi^*$.
\end{proof}

Together with \Cref{lemma:manystandardlyemb}, we can assume that $\tau_{(\mathcal{X}_I,\mathcal{X}_O)}$ is orientable and floral (which implies that the infinitesimal polygon $c_I$ is even-pronged), or \textit{orientable floral} for short. See \Cref{fig:evennice} for an example of such a train track. In this case, if we label the vertices of $c_I$ by $v_1,..,v_n$ in a counterclockwise order, then by orientability, the endpoints on each real edge must lie on $v_i$ and $v_j$ for $i,j$ of different parity.

\begin{figure}
    \centering
    \resizebox{!}{4cm}{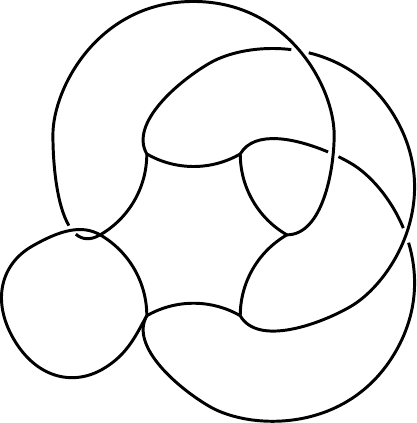}
    \caption{An orientable floral train track.}
    \label{fig:evennice}
\end{figure}

\Cref{constr:realedgeelement} can be repeated for orientable floral train tracks word-by-word. Notice here that we require $\tau$ to be orientable in order for $k$ to be even, in the notation of \Cref{constr:realedgeelement}. Also notice that in this case both cyclic orderings give an even $k$. The $w_e$ defined under the two choices will differ by $r_{c_I}$. This is not very significant for our purposes, but for concreteness we can simply fix some choice for each real edge $e$.

The proofs of \Cref{lemma:combineprong} and \Cref{lemma:splitprong} then carry over word for word to show the following lemma.

\begin{lemma} \label{lemma:deleteevenprong}
Suppose $\tau$ and $\tau'$ are orientable floral train tracks where $\tau'$ is obtained from an admissible deletion of a real edge $e$ from $\tau$. Then (\ref{eq:radical}) holds for $\tau$ if and only if it holds for $\tau'$.
\end{lemma}

Hence by the reasoning at the end of the last subsection, we can simply establish (\ref{eq:radical}) for one orientable floral train track with an $n$-pronged infinitesimal polygon for every even $n$.

We choose the floral train track $\tau_n$ illustrated in \Cref{fig:eztt} right to this end. That is, $\tau_n$ has an $n$-pronged infinitesimal polygon, say with vertices labelled $v_1,...,v_n$ in a cyclic way, and $n$ real edges $e_1,...,e_n$, each $e_i$ having endpoints on $v_i$ and $v_{i+1}$, with $e_{i-1}$ lying to the left of $e_i$ at $v_i$ for every $i$. 
Then $\tau_n$ has $n+2$ boundary components, two of them $n$-pronged and $n$ of them $0$-pronged. It is straightforward to check that $W(\tau_n)$ is generated by $\{r_c\}$ for all these boundary components $c$. Hence $\rad(\omega)=\spn \{r_c\}$ must hold. This concludes the proof of \Cref{prop:radical}.

\section{Proof of the main theorem} \label{sec:mainthm}

We gather all the ingredients to prove our main theorem, which we restate below for the reader's convenience.

\begin{thm} \label{thm:mainthmtext}
Let $f:S \to S$ be a fully-punctured pseudo-Anosov map with at least two puncture orbits. Then the normalized expansion factor $L(S,f) = \lambda(f)^{|\chi(S)|}$ satisfies the inequality
$$L(S,f) \geq \mu^4.$$

More precisely, for $|\chi(S)| \geq 3$, we have
$$\lambda(f) \geq |LT_{1,\frac{K}{2}}|$$
if $|\chi(S)|=K$ is even, and
$$\lambda(f)^K \geq 8$$
if $|\chi(S)|=K$ is odd.
\end{thm}
\begin{proof}
Take some partition $\mathcal{X}=\mathcal{X}_I \sqcup \mathcal{X}_O$ of the set of punctures of $S$ into two nonempty $f$-invariant sets. Consider the standardly embedded train track $\tau=\tau_{(\mathcal{X}_I,\mathcal{X}_O)}$, the train track map $f=f_{(\mathcal{X}_I,\mathcal{X}_O)}$, and the matrix $f_*^\real=f_{(\mathcal{X}_I,\mathcal{X}_O)*}^\real$. \Cref{prop:realedgesPF} shows that $f_*^\real$ is Perron-Frobenius.

Meanwhile, consider the weight space $\mathcal{W}(\tau)$. We have the commutative diagram
\begin{center}
\begin{tikzcd}
0 \arrow[r] & \rad(\omega) \arrow[r] \arrow[d, "\mathcal{W}(f)"] & \mathcal{W}(\tau) \arrow[r] \arrow[d, "\mathcal{W}(f)"] & \mathcal{W}(\tau)/\rad(\omega) \arrow[r] \arrow[d, "\mathcal{W}(f)"] & 0 \\
0 \arrow[r] & \rad(\omega) \arrow[r] & \mathcal{W}(\tau) \arrow[r] & \mathcal{W}(\tau)/\rad(\omega) \arrow[r] & 0 
\end{tikzcd}
\end{center}
\Cref{prop:radicalreciprocal} and \Cref{prop:radical} imply that the restriction of $\mathcal{W}(f)$ to $\rad(\omega)$ is reciprocal. On the other hand, $\mathcal{W}(\tau)/\rad(\omega)$ inherits the form $\omega$, which is now symplectic. The induced map of $\mathcal{W}(f)$ on $\mathcal{W}(\tau)/\rad(\omega)$ preserves $\omega$ hence is symplectic, thus reciprocal by \Cref{prop:reciprocalfacts}(1). By \Cref{prop:reciprocalfacts}(3), $\mathcal{W}(f):\mathcal{W}(\tau) \to \mathcal{W}(\tau)$ is reciprocal. 

We also have the following commutative diagram from \Cref{prop:weightspaceses}.
\begin{center}
\begin{tikzcd}
0 \arrow[r] & \mathcal{W}(\tau) \arrow[r] \arrow[d, "\mathcal{W}(f)"] & \mathbb{R}^{\mathcal{E}} \arrow[r, "T_\mathcal{V}"] \arrow[d, "f_*"] & T_\mathcal{V}(\mathbb{R}^{\mathcal{E}}) \arrow[d, "P|_{T_\mathcal{V}(\mathbb{R}^{\mathcal{E}})}"] \arrow[r] & 0 \\
0 \arrow[r] & \mathcal{W}(\tau) \arrow[r] & \mathbb{R}^{\mathcal{E}} \arrow[r, "T_\mathcal{V}"] & T_\mathcal{V}(\mathbb{R}^{\mathcal{E}}) \arrow[r] & 0
\end{tikzcd}
\end{center}
We have deduced that $\mathcal{W}(f):\mathcal{W}(\tau) \to \mathcal{W}(\tau)$ is reciprocal above. By \Cref{prop:reciprocalfacts}(2), $P|_{T_\mathcal{V}(\mathbb{R}^{\mathcal{E}})}$ is reciprocal. So by \Cref{prop:reciprocalfacts}(3), $f_*$ is reciprocal.

Finally, by the discussion above \Cref{prop:realedgesPF}, $f_*$ being reciprocal implies that $f_*^\real$ is reciprocal. Hence $f_*^\real$ is a $|\chi(S)|$-by-$|\chi(S)|$ reciprocal Perron-Frobenius matrix.

By \Cref{prop:realedgesdilatation}, $\lambda(f)$ is the spectral radius of $f_*^\real$. Hence the theorem follows from \Cref{thm:McMullen}.
\end{proof}

\section{Sharpness of the main theorem} \label{sec:sharpness}

In this section, we discuss the sharpness of the main theorem. We give two families of pseudo-Anosov maps realizing the lower bounds in \Cref{thm:sharpthm} for even $\chi(S)$. We do this in two ways: In \Cref{subsec:evensharpness} we describe folding sequences of train tracks which determine the maps. In \Cref{subsec:fiberedface} we describe classes in certain fibered faces which determine the same maps. In \Cref{subsec:braidsharpness} we give examples showing that \Cref{thm:braids} and \Cref{thm:sharpthm} for odd $\chi(S)$ are not sharp in general. In \Cref{subsec:oneboundary} we give examples showing that the assumption of $f$ having at least two puncture orbits is necessary in \Cref{thm:mainthm}.

\subsection{Examples for even $\chi(S)$: train tracks} \label{subsec:evensharpness}

In this subsection, we will show that \Cref{thm:sharpthm} is sharp in the cases when $\chi(S)$ is even, by demonstrating \textit{folding sequences} of standardly embedded train tracks, i.e. sequences of the form $\tau_0 \overset{f_1}{\to} \cdots \overset{f_n}{\to} \tau_n \overset{\sigma}{\to} \tau_0$ where each $f_i$ is the composition of a subdivision move and an elementary folding move involving one of the subdivided edges, and $\sigma$ is an isomorphism of train tracks.

Given such a sequence, the induced map $\sigma f_n \cdots f_1$ on the tie neighborhood $N$ of $\tau_1$ will determine a mapping class on the punctured surface $S$ that is the interior of $N$. If the real transition matrix $f_*^\real$ of $\sigma f_n \cdots f_1$ is Perron-Frobenius, then this mapping class contains a (unique) fully-punctured pseudo-Anosov map $f:S \to S$, for which $\tau_0$ is an invariant train track. For a more detailed explanation of recovering pseudo-Anosov maps from folding sequences of train tracks, see for example \cite{BH95}. 

As described in \Cref{sec:standardlyemb}, one can then compute the expansion factor $\lambda(f)$ of $f$ as the spectral radius of $f_*^\real$. Here we will make this computation using the method described in \cite{McM15}. Namely, we write down the directed graph associated to $f_*^\real$, compute its curve complex $G$, then compute the clique polynomial $Q_G(t)$ of $G$. \cite[Theorem 1.2 and Theorem 1.4]{McM15} state that the smallest positive root of $Q_G(t)$ is equal to $\frac{1}{\lambda(f)}$. Here, as shown in the proof of \Cref{thm:mainthm}, $f_*^\real$ will be reciprocal, so we can more directly compute $\lambda(f)$ as the largest positive root of $Q_G(t)$.

Our first family of examples is shown in \Cref{fig:l6a2tt}. Here, each train track has a single infinitesimal polygon with $3k$ cusps, which we have used as the center of reference, and $2k$ real edges. We indicated each fold $f_i$ by highlighting the relevant edges in bright red (before) and dark red (after). The train track isomorphism $\sigma$ is induced by a rotation of the center infinitesimal polygon.

\begin{figure}
    \centering
    \resizebox{!}{10cm}{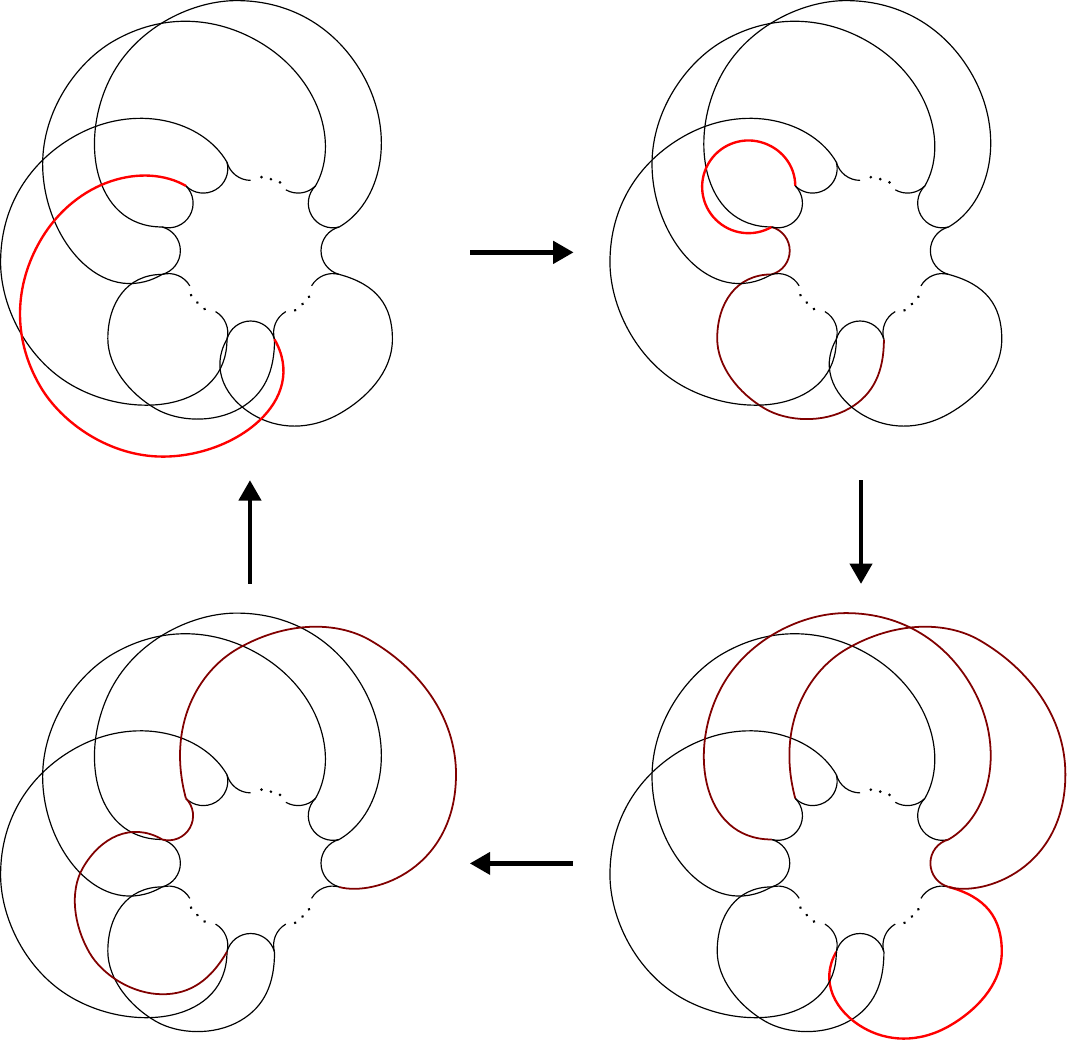}
    \caption{Folding sequence of train tracks in the first family of examples.}
    \label{fig:l6a2tt}
\end{figure}

When $k=2$, one computes the real transition matrix to be
$$f_*^\real = \left[ \begin{array}{cccc}
1 & 0 & 0 & 1 \\
0 & 0 & 1 & 1 \\
1 & 0 & 0 & 2 \\
0 & 1 & 0 & 0 \\
\end{array} \right]$$
and when $k \geq 3$, one computes it to be
$$f_*^\real = \left[ \begin{array}{cc|c|cccc}
0 & 0 & I_{2k-6} & 0 & 0 & 0 & 0 \\
\hline
0 & 0 & 0 & 1 & 0 & 0 & 1 \\
0 & 0 & 0 & 0 & 1 & 0 & 0 \\
\hline
1 & 0 & 0 & 0 & 0 & 0 & 0 \\
0 & 0 & 0 & 0 & 0 & 1 & 1 \\
1 & 0 & 0 & 0 & 0 & 0 & 1 \\
0 & 1 & 0 & 0 & 0 & 0 & 0 \\
\end{array} \right]$$
where $I_n$ denotes the $n$ by $n$ identity matrix.

The corresponding directed graphs and curve complexes for $k=2$ and $k \geq 3$ are shown in \Cref{fig:l6a2graph} top and bottom respectively. Here a number $n$ besides a directed edge $\to$ is shorthand for $n$ consecutive edges $\to \cdot \to \cdots \to \cdot \to$. Meanwhile, a number besides a vertex in the curve complex denotes its weight.

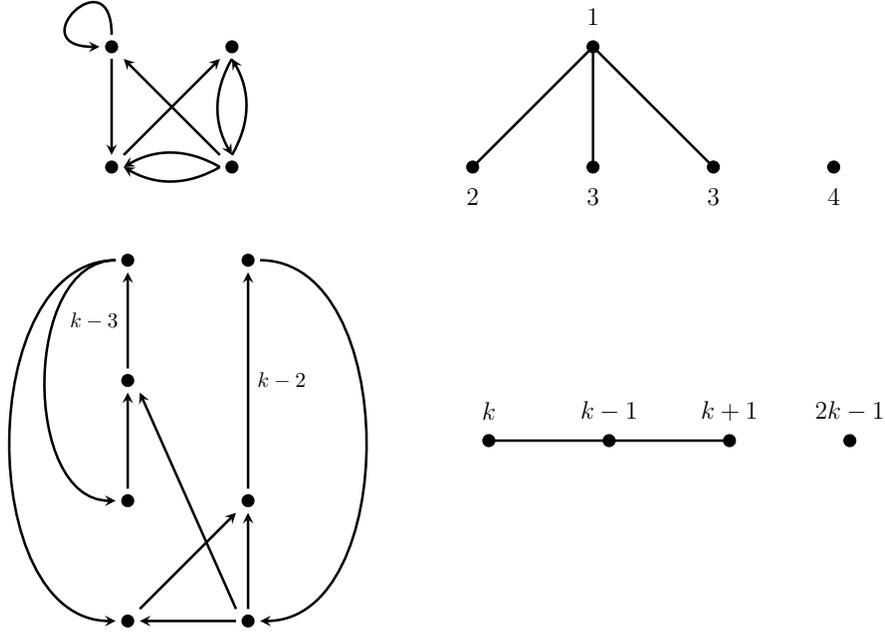
\begin{figure}
    \centering
    \scalebox{0.8}{\begin{tikzpicture}
        \draw[fill=black] (0,0) circle(0.1);
        \draw[fill=black] (0,2) circle(0.1);
        \draw[fill=black] (2,0) circle(0.1);
        \draw[fill=black] (2,2) circle(0.1);
        
        \draw[very thick,-stealth] (0,1.8) to (0,0.2);
        \draw[very thick,-stealth] (2,1.8) to[out=240,in=120] (2,0.2);
        \draw[very thick,-stealth] (2,0.2) to[out=60,in=300] (2,1.8);
        \draw[very thick,-stealth] (1.8,0) to[out=150,in=30] (0.2,0);
        \draw[very thick,-stealth] (1.8,0) to[out=210,in=-30] (0.2,0);
        \draw[very thick,-stealth] (0.2,0.2) to (1.8,1.8);
        \draw[very thick,-stealth] (1.8,0.2) to (0.2,1.8);
        \draw[very thick,-stealth] (0,2.2) to[out=90,in=180,looseness=12] (-0.2,2);
        
        \draw[fill=black] (6,0) circle(0.1);
        \node at (6,-0.5) {$2$};
        \draw[fill=black] (8,0) circle(0.1);
        \node at (8,-0.5) {$3$};
        \draw[fill=black] (10,0) circle(0.1);
        \node at (10,-0.5) {$3$};
        \draw[fill=black] (12,0) circle(0.1);
        \node at (12,-0.5) {$4$};
        \draw[fill=black] (8,2) circle(0.1);
        \node at (8,2.5) {$1$};
        
        \draw[very thick] (6,0) to (8,2);
        \draw[very thick] (8,0) to (8,2);
        \draw[very thick] (10,0) to (8,2);
    \end{tikzpicture}}
    
    \vspace{0.5cm}
    
    \scalebox{0.8}{\begin{tikzpicture}
        \draw[fill=black] (0,0) circle(0.1);
        \draw[fill=black] (0,2) circle(0.1);
        \draw[fill=black] (0,4) circle(0.1);
        \draw[fill=black] (0,6) circle(0.1);
        \draw[fill=black] (2,0) circle(0.1);
        \draw[fill=black] (2,2) circle(0.1);
        \draw[fill=black] (2,6) circle(0.1);
        
        \draw[very thick,-stealth] (0,2.2) to (0,3.8);
        \draw[very thick,-stealth] (0,4.2) to node[left]{\footnotesize $k-3$} (0,5.8);
        
        \draw[very thick,-stealth] (2,0.2) to (2,1.8);
        \draw[very thick,-stealth] (2,2.2) to node[right]{\footnotesize $k-2$} (2,5.8);
        
        \draw[very thick,-stealth] (1.8,0) to (0.2,0);
        \draw[very thick,-stealth] (0.2,0.2) to (1.8,1.8);
        \draw[very thick,-stealth] (1.8,0.2) to (0.2,3.8);
        \draw[very thick,-stealth] (-0.2,6) to[out=180,in=180] (-0.2,0);
        \draw[very thick,-stealth] (-0.2,6) to[out=180,in=180] (-0.2,2);
        \draw[very thick,-stealth] (2.2,6) to[out=0,in=0] (2.2,0);

        \draw[fill=black] (6,3) circle(0.1);
        \node at (6,3.5) {$k$};
        \draw[fill=black] (8,3) circle(0.1);
        \node at (8,3.5) {$k-1$};
        \draw[fill=black] (10,3) circle(0.1);
        \node at (10,3.5) {$k+1$};
        \draw[fill=black] (12,3) circle(0.1);
        \node at (12,3.5) {$2k-1$};
        
        \draw[very thick] (6,3) to (8,3);
        \draw[very thick] (8,3) to (10,3);
    \end{tikzpicture}}
    \caption{The associated directed graphs to the real transition matrices of \Cref{fig:l6a2tt} and their curve complexes.}
    \label{fig:l6a2graph}
\end{figure}

For each $k \geq 2$, the clique polynomial of the corresponding curve complex is $LT_{1,k}=t^{2k}-t^{k+1}-t^k-t^{k-1}+1$. Hence the induced fully-punctured pseudo-Anosov maps in this family of examples have expansion factors $\left| LT_{1,k} \right|$.

Notice that this first family of examples already shows the sharpness statement in \Cref{thm:sharpthm}. However, we are actually able to find a second family of examples that attain equality in \Cref{thm:sharpthm} for even $|\chi(S)|$. We will describe this second family next, following the same format. 

Unfortunately, for number theoretical reasons, it is difficult to present a single picture that illustrates all the members of this second family; we need to split into $5$ subcases depending on the value of $k$ (mod $5$), where $|\chi(S)|=2k$.

For $k \equiv 3$ (mod $5$), the folding sequence is shown in \Cref{fig:l13n5885ttk3}. 

For consistency, we have drawn \Cref{fig:l13n5885ttk3} in the style as \Cref{fig:l6a2tt}. Namely, we used the single infinitesimal polygon as the center of reference, we highlighted each fold in red, and the train track isomorphism is induced by a rotation of the infinitesimal polygon. 

However, we caution that there are a few differences: This time the infinitesimal polygon has $2k+1$ cusps (but there are still $2k$ real edges). Real edges that connect a shown cusp to a cusp in the $\cdots$ range are truncated. Also, the number of cusps in each $\cdots$ differs. The fact that the last map is a train track isomorphism will determine how these real edges should be connected and how many cusps each $\cdots$ should contain.

\begin{figure}
    \centering
    \fontsize{30pt}{30pt}\selectfont
    \resizebox{!}{10cm}{
\begingroup%
  \makeatletter%
  \providecommand\color[2][]{%
    \errmessage{(Inkscape) Color is used for the text in Inkscape, but the package 'color.sty' is not loaded}%
    \renewcommand\color[2][]{}%
  }%
  \providecommand\transparent[1]{%
    \errmessage{(Inkscape) Transparency is used (non-zero) for the text in Inkscape, but the package 'transparent.sty' is not loaded}%
    \renewcommand\transparent[1]{}%
  }%
  \providecommand\rotatebox[2]{#2}%
  \newcommand*\fsize{\dimexpr\f@size pt\relax}%
  \newcommand*\lineheight[1]{\fontsize{\fsize}{#1\fsize}\selectfont}%
  \ifx\svgwidth\undefined%
    \setlength{\unitlength}{523.50803863bp}%
    \ifx\svgscale\undefined%
      \relax%
    \else%
      \setlength{\unitlength}{\unitlength * \real{\svgscale}}%
    \fi%
  \else%
    \setlength{\unitlength}{\svgwidth}%
  \fi%
  \global\let\svgwidth\undefined%
  \global\let\svgscale\undefined%
  \makeatother%
  \begin{picture}(1,0.99179964)%
    \lineheight{1}%
    \setlength\tabcolsep{0pt}%
    \put(0,0){\includegraphics[width=\unitlength,page=1]{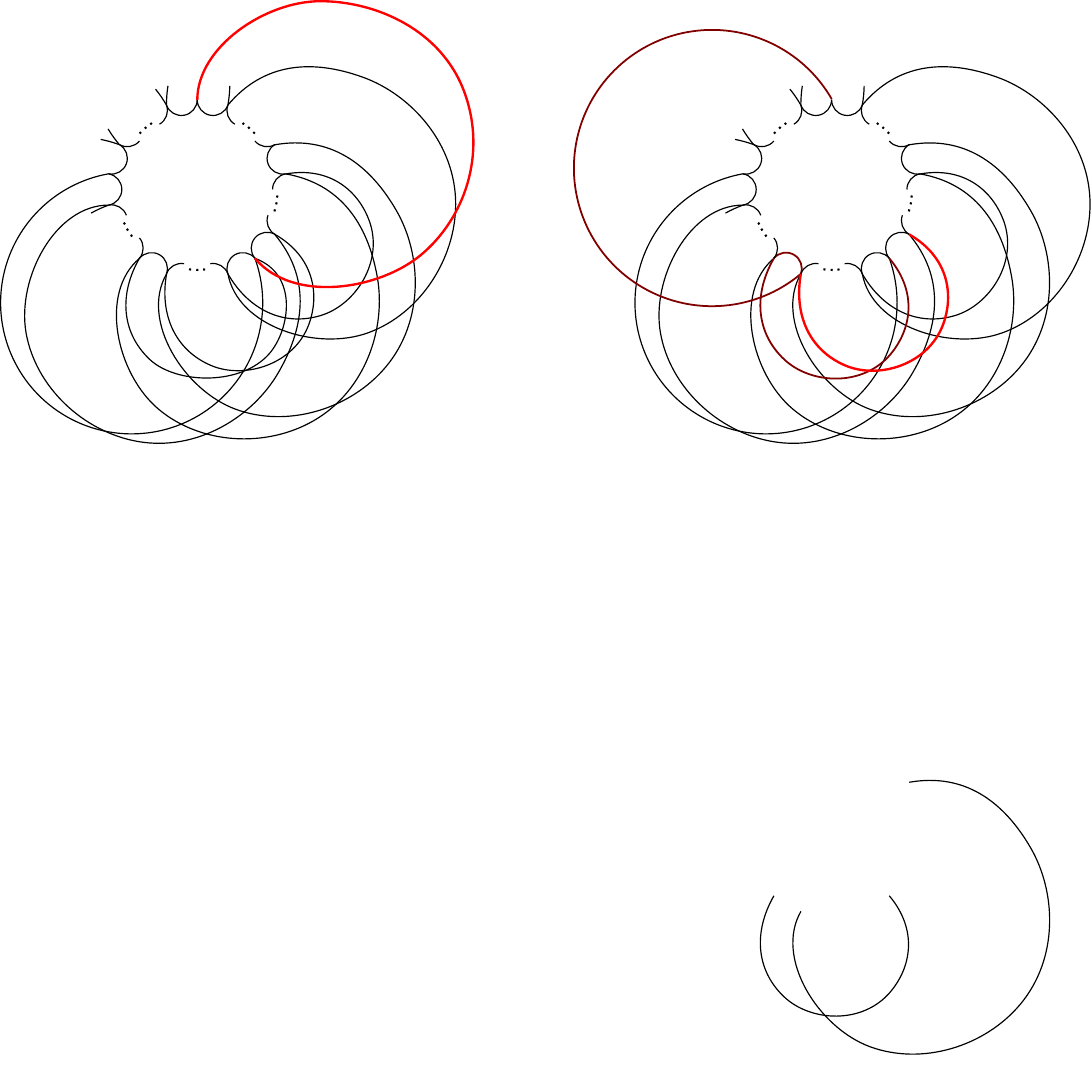}}%
    \put(0.12755346,0.1274319){\color[rgb]{0,0,0}\makebox(0,0)[lt]{\lineheight{1.25}\smash{\begin{tabular}[t]{l}$k \equiv 3$ (mod $5$)\end{tabular}}}}%
    \put(0,0){\includegraphics[width=\unitlength,page=2]{l13n5885tt1.pdf}}%
  \end{picture}%
\endgroup%
}
    \caption{Folding sequence of train tracks in the second family of examples.}
    \label{fig:l13n5885ttk3}
\end{figure}

For $k \equiv 4$, $0$, and $1$ (mod $5$), the folding sequences are shown in \Cref{fig:l13n5885ttk401}. These are drawn in the same style as \Cref{fig:l13n5885ttk3}.

\begin{figure}
    \centering
    \fontsize{30pt}{30pt}\selectfont
    \resizebox{!}{10cm}{
\begingroup%
  \makeatletter%
  \providecommand\color[2][]{%
    \errmessage{(Inkscape) Color is used for the text in Inkscape, but the package 'color.sty' is not loaded}%
    \renewcommand\color[2][]{}%
  }%
  \providecommand\transparent[1]{%
    \errmessage{(Inkscape) Transparency is used (non-zero) for the text in Inkscape, but the package 'transparent.sty' is not loaded}%
    \renewcommand\transparent[1]{}%
  }%
  \providecommand\rotatebox[2]{#2}%
  \newcommand*\fsize{\dimexpr\f@size pt\relax}%
  \newcommand*\lineheight[1]{\fontsize{\fsize}{#1\fsize}\selectfont}%
  \ifx\svgwidth\undefined%
    \setlength{\unitlength}{475.65587752bp}%
    \ifx\svgscale\undefined%
      \relax%
    \else%
      \setlength{\unitlength}{\unitlength * \real{\svgscale}}%
    \fi%
  \else%
    \setlength{\unitlength}{\svgwidth}%
  \fi%
  \global\let\svgwidth\undefined%
  \global\let\svgscale\undefined%
  \makeatother%
  \begin{picture}(1,1.1781919)%
    \lineheight{1}%
    \setlength\tabcolsep{0pt}%
    \put(0,0){\includegraphics[width=\unitlength,page=1]{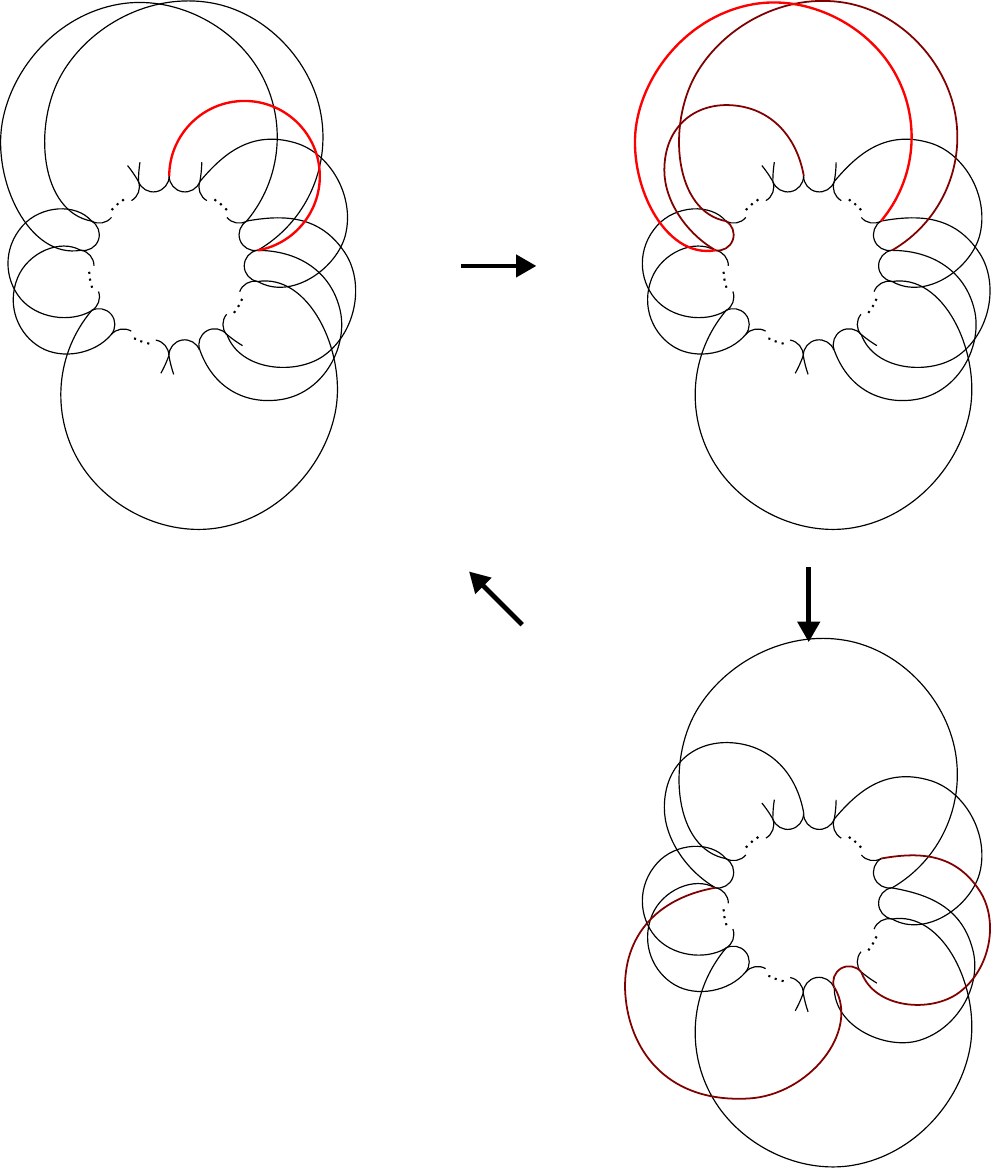}}%
    \put(0.11212752,0.14974927){\color[rgb]{0,0,0}\makebox(0,0)[lt]{\lineheight{1.25}\smash{\begin{tabular}[t]{l}$k \equiv 4$ (mod $5$)\end{tabular}}}}%
  \end{picture}%
\endgroup%
}
    
    \vspace{0.5cm}
    
    \resizebox{!}{10cm}{
\begingroup%
  \makeatletter%
  \providecommand\color[2][]{%
    \errmessage{(Inkscape) Color is used for the text in Inkscape, but the package 'color.sty' is not loaded}%
    \renewcommand\color[2][]{}%
  }%
  \providecommand\transparent[1]{%
    \errmessage{(Inkscape) Transparency is used (non-zero) for the text in Inkscape, but the package 'transparent.sty' is not loaded}%
    \renewcommand\transparent[1]{}%
  }%
  \providecommand\rotatebox[2]{#2}%
  \newcommand*\fsize{\dimexpr\f@size pt\relax}%
  \newcommand*\lineheight[1]{\fontsize{\fsize}{#1\fsize}\selectfont}%
  \ifx\svgwidth\undefined%
    \setlength{\unitlength}{524.87518742bp}%
    \ifx\svgscale\undefined%
      \relax%
    \else%
      \setlength{\unitlength}{\unitlength * \real{\svgscale}}%
    \fi%
  \else%
    \setlength{\unitlength}{\svgwidth}%
  \fi%
  \global\let\svgwidth\undefined%
  \global\let\svgscale\undefined%
  \makeatother%
  \begin{picture}(1,1.05852522)%
    \lineheight{1}%
    \setlength\tabcolsep{0pt}%
    \put(0,0){\includegraphics[width=\unitlength,page=1]{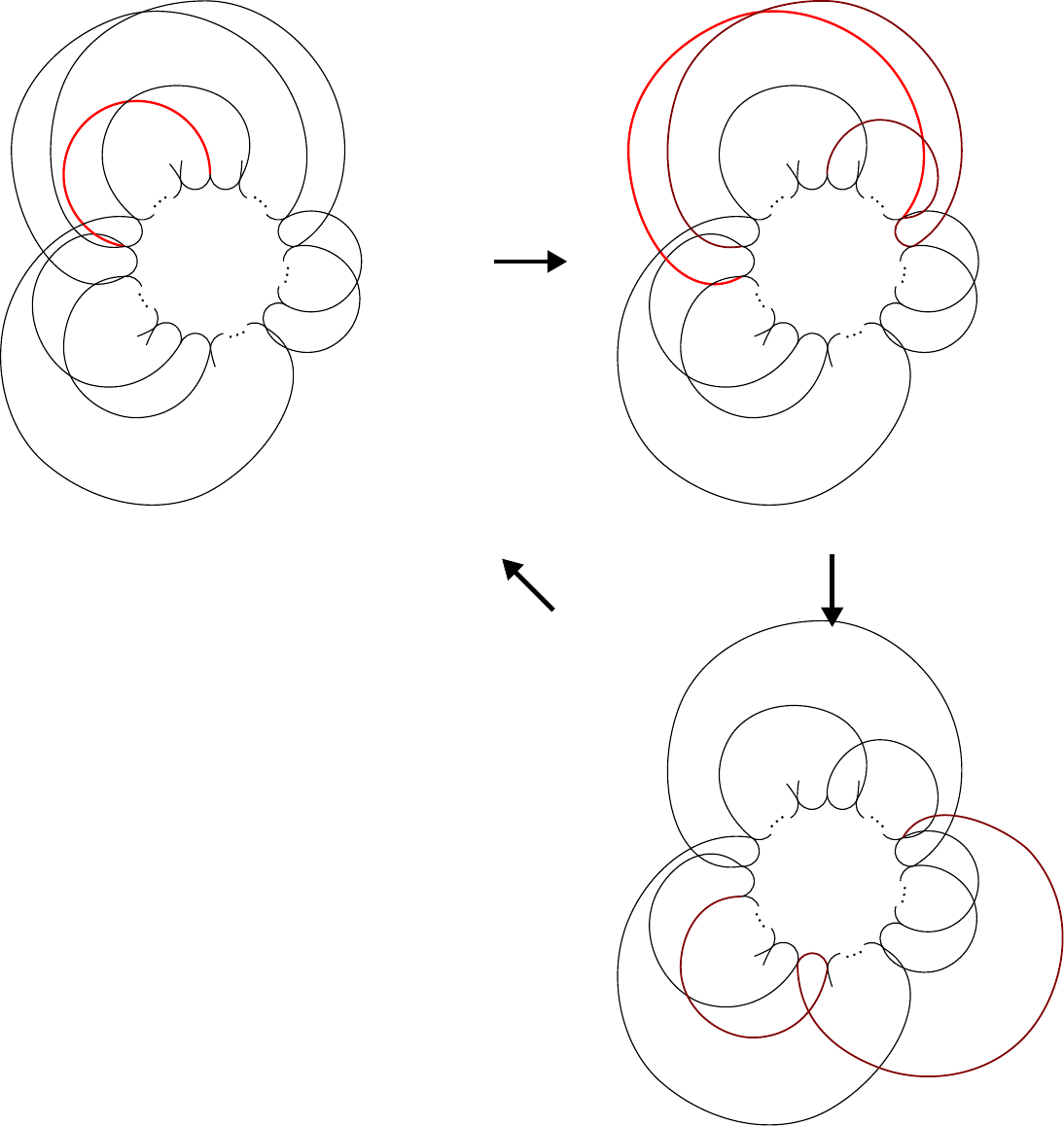}}%
    \put(0.14456122,0.12380803){\color[rgb]{0,0,0}\makebox(0,0)[lt]{\lineheight{1.25}\smash{\begin{tabular}[t]{l}$k \equiv 0$ (mod $5$)\end{tabular}}}}%
  \end{picture}%
\endgroup%
}
    \caption{Folding sequence of train tracks in the second family of examples.}
\end{figure}
    
\begin{figure}
    \ContinuedFloat
    \centering
    \fontsize{30pt}{30pt}\selectfont
    \resizebox{!}{10cm}{
\begingroup%
  \makeatletter%
  \providecommand\color[2][]{%
    \errmessage{(Inkscape) Color is used for the text in Inkscape, but the package 'color.sty' is not loaded}%
    \renewcommand\color[2][]{}%
  }%
  \providecommand\transparent[1]{%
    \errmessage{(Inkscape) Transparency is used (non-zero) for the text in Inkscape, but the package 'transparent.sty' is not loaded}%
    \renewcommand\transparent[1]{}%
  }%
  \providecommand\rotatebox[2]{#2}%
  \newcommand*\fsize{\dimexpr\f@size pt\relax}%
  \newcommand*\lineheight[1]{\fontsize{\fsize}{#1\fsize}\selectfont}%
  \ifx\svgwidth\undefined%
    \setlength{\unitlength}{561.27439919bp}%
    \ifx\svgscale\undefined%
      \relax%
    \else%
      \setlength{\unitlength}{\unitlength * \real{\svgscale}}%
    \fi%
  \else%
    \setlength{\unitlength}{\svgwidth}%
  \fi%
  \global\let\svgwidth\undefined%
  \global\let\svgscale\undefined%
  \makeatother%
  \begin{picture}(1,0.96407187)%
    \lineheight{1}%
    \setlength\tabcolsep{0pt}%
    \put(0,0){\includegraphics[width=\unitlength,page=1]{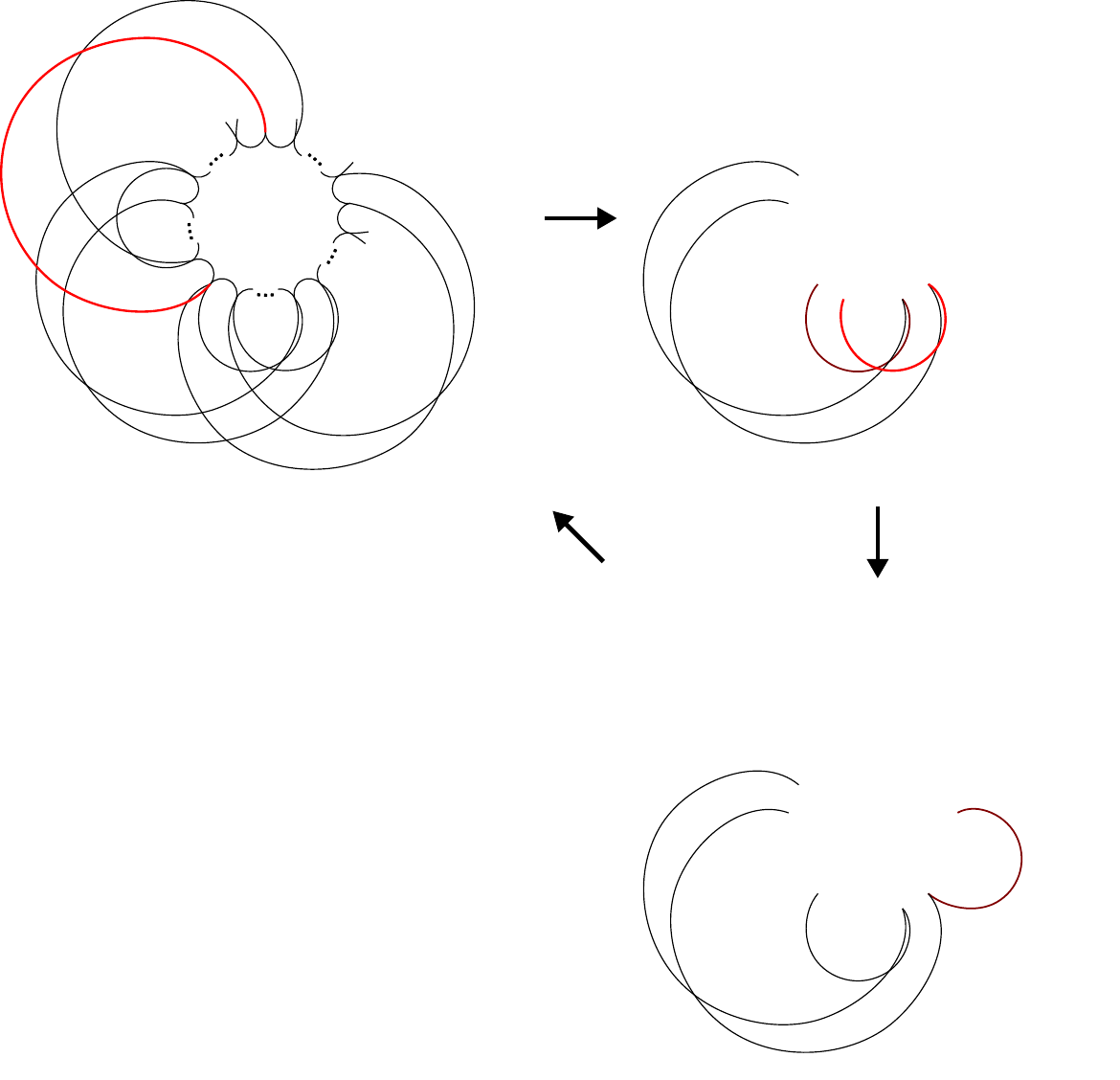}}%
    \put(0.1875413,0.12499424){\color[rgb]{0,0,0}\makebox(0,0)[lt]{\lineheight{1.25}\smash{\begin{tabular}[t]{l}$k \equiv 1$ (mod $5$)\end{tabular}}}}%
    \put(0,0){\includegraphics[width=\unitlength,page=2]{l13n5885tt4.pdf}}%
  \end{picture}%
\endgroup%
}
    \caption{Folding sequence of train tracks in the second family of examples.}
    \label{fig:l13n5885ttk401}
\end{figure}

Finally, for $k \equiv 2$ (mod $5$), the folding sequence is shown in \Cref{fig:l13n5885ttk2}.

\begin{figure}
    \centering
    \fontsize{30pt}{30pt}\selectfont
    \resizebox{!}{10cm}{
\begingroup%
  \makeatletter%
  \providecommand\color[2][]{%
    \errmessage{(Inkscape) Color is used for the text in Inkscape, but the package 'color.sty' is not loaded}%
    \renewcommand\color[2][]{}%
  }%
  \providecommand\transparent[1]{%
    \errmessage{(Inkscape) Transparency is used (non-zero) for the text in Inkscape, but the package 'transparent.sty' is not loaded}%
    \renewcommand\transparent[1]{}%
  }%
  \providecommand\rotatebox[2]{#2}%
  \newcommand*\fsize{\dimexpr\f@size pt\relax}%
  \newcommand*\lineheight[1]{\fontsize{\fsize}{#1\fsize}\selectfont}%
  \ifx\svgwidth\undefined%
    \setlength{\unitlength}{575.55390752bp}%
    \ifx\svgscale\undefined%
      \relax%
    \else%
      \setlength{\unitlength}{\unitlength * \real{\svgscale}}%
    \fi%
  \else%
    \setlength{\unitlength}{\svgwidth}%
  \fi%
  \global\let\svgwidth\undefined%
  \global\let\svgscale\undefined%
  \makeatother%
  \begin{picture}(1,0.96598002)%
    \lineheight{1}%
    \setlength\tabcolsep{0pt}%
    \put(0,0){\includegraphics[width=\unitlength,page=1]{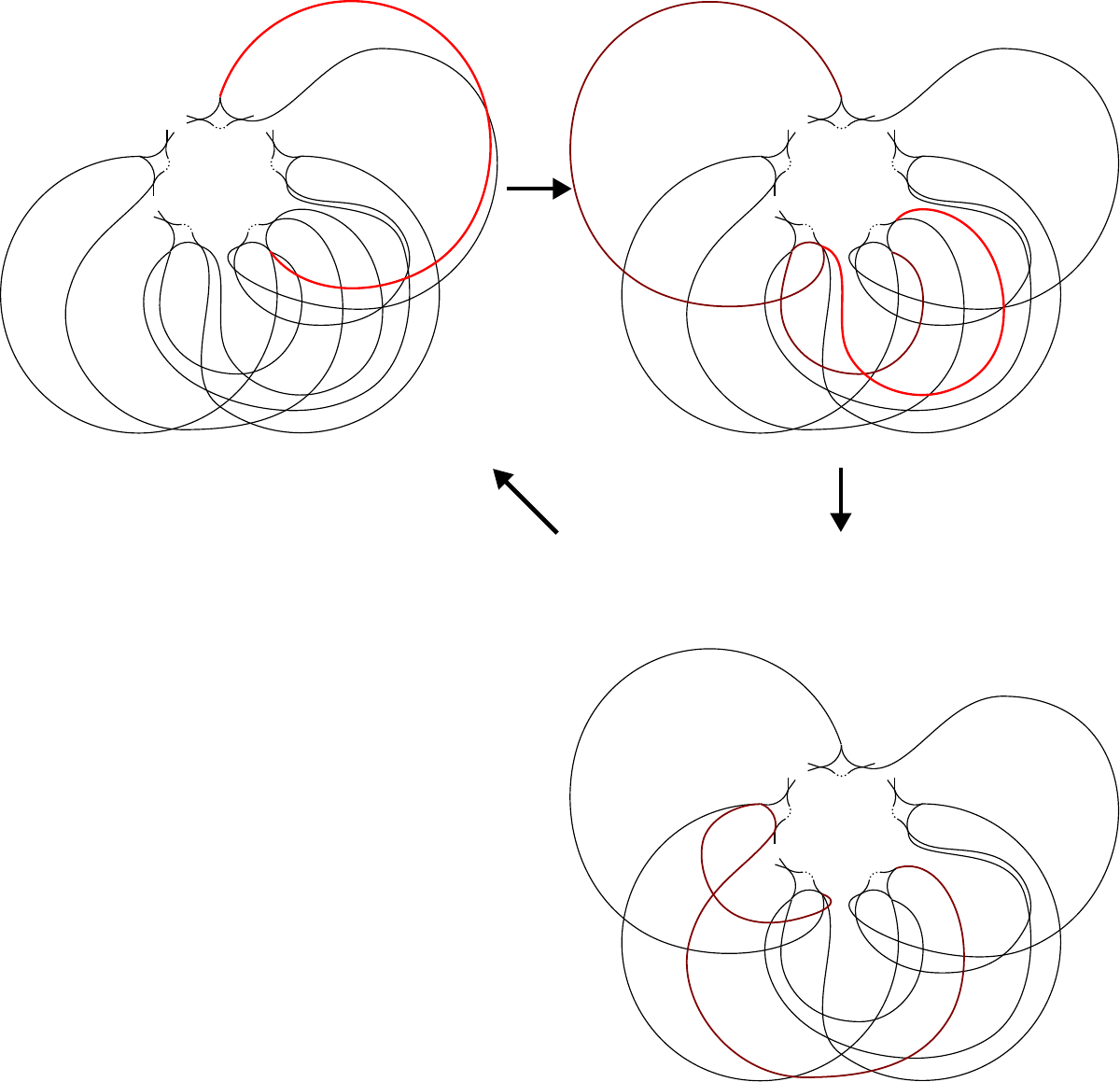}}%
    \put(0.14185382,0.17396916){\color[rgb]{0,0,0}\makebox(0,0)[lt]{\lineheight{1.25}\smash{\begin{tabular}[t]{l}$k \equiv 2$ (mod $5$)\end{tabular}}}}%
  \end{picture}%
\endgroup%
}
    \caption{Folding sequence of train tracks in the second family of examples.}
    \label{fig:l13n5885ttk2}
\end{figure}

\Cref{fig:l13n5885ttk2} is drawn in a slightly different style from the previous pictures by necessity. There are now $5$ infinitesimal polygons, each having $\frac{2k+1}{5}$ cusps. The train track isomorphism now permutes the $5$ infinitesimal polygons, we have arranged these so that this permutation is induced by a rotation. However, we caution that the isomorphism is not just a rotation of the picture; one also needs to rotate one of the infinitesimal polygons.

The real transition matrices of these subcases admit a more consistent description. When $k=2$, we have
$$f_*^\real = \left[ \begin{array}{cccc}
0 & 1 & 1 & 1 \\
0 & 0 & 1 & 0 \\
1 & 0 & 0 & 0 \\
0 & 0 & 1 & 1 \\
\end{array} \right]$$
and when $k \geq 3$, we have
$$f_*^\real = \left[ \begin{array}{ccc|c|ccc}
0 & 1 & 1 & 0 & 0 & 0 & 0 \\
\hline
0 & 0 & 0 & I_{2k-6} & 0 & 0 & 0 \\
\hline
0 & 0 & 0 & 0 & 0 & 1 & 1 \\
0 & 0 & 0 & 0 & 1 & 0 & 0 \\
0 & 0 & 0 & 0 & 0 & 1 & 0 \\
1 & 0 & 0 & 0 & 0 & 0 & 0 \\
0 & 1 & 0 & 0 & 0 & 0 & 0 \\
\end{array} \right]$$

The corresponding directed graphs and curve complexes for $k=2$ and $k \geq 3$ are shown in \Cref{fig:l13n5885graph} top and bottom respectively.

\begin{figure}
    \centering
    \scalebox{0.8}{\begin{tikzpicture}
        \draw[fill=black] (0,0) circle(0.1);
        \draw[fill=black] (0,2) circle(0.1);
        \draw[fill=black] (2,0) circle(0.1);
        \draw[fill=black] (2,2) circle(0.1);
        
        \draw[very thick,-stealth] (0,0.2) to (0,1.8);
        \draw[very thick,-stealth] (2,0.2) to (2,1.8);
        \draw[very thick,-stealth] (0.2,1.8) to[out=-15,in=105] (1.8,0.2);
        \draw[very thick,-stealth] (1.8,0.2) to[out=165,in=-75] (0.2,1.8);
        \draw[very thick,-stealth] (1.8,0) to (0.2,0);        
        \draw[very thick,-stealth] (1.8,2) to (0.2,2);
        \draw[very thick,-stealth] (2,2.2) to[out=90,in=0,looseness=12] (2.2,2);
        
        \draw[fill=black] (6,1) circle(0.1);
        \node at (6,1.5) {$2$};
        \draw[fill=black] (8,1) circle(0.1);
        \node at (8,1.5) {$1$};
        \draw[fill=black] (10,1) circle(0.1);
        \node at (10,1.5) {$3$};
        \draw[fill=black] (12,1) circle(0.1);
        \node at (12,1.5) {$3$};
        
        \draw[very thick] (6,1) to (8,1);
        \draw[very thick] (8,1) to (10,1);
    \end{tikzpicture}}
    
    \vspace{0.5cm}
    
    \scalebox{0.8}{\begin{tikzpicture}
        \draw[fill=black] (0,0) circle(0.1);
        \draw[fill=black] (0,2) circle(0.1);
        \draw[fill=black] (0,4) circle(0.1);
        \draw[fill=black] (0,6) circle(0.1);
        \draw[fill=black] (2,0) circle(0.1);
        \draw[fill=black] (2,2) circle(0.1);
        \draw[fill=black] (2,4) circle(0.1);
        \draw[fill=black] (4,2) circle(0.1);
        
        \draw[very thick,-stealth] (0,0.2) to (0,1.8);
        \draw[very thick,-stealth] (0,2.2) to node[left]{\footnotesize $k-3$} (0,3.8);
        \draw[very thick,-stealth] (0,4.2) to (0,5.8);
        
        \draw[very thick,-stealth] (2,0.2) to (2,1.8);
        \draw[very thick,-stealth] (2,2.2) to node[left]{\footnotesize $k-3$} (2,3.8);
        
        \draw[very thick,-stealth] (1.8,0) to (0.2,0);
        \draw[very thick,-stealth] (1.8,4.2) to (0.2,5.8);
        \draw[very thick,-stealth] (3.8,2) to (2.2,2);
        \draw[very thick,-stealth] (0.1,5.8) to[out=-77,in=115] (1.8,0.2);
        \draw[very thick,-stealth] (2.2,3.8) to (3.8,2.2);
        
        \draw[fill=black] (8,3) circle(0.1);
        \node at (8,3.5) {$k+1$};
        \draw[fill=black] (10,3) circle(0.1);
        \node at (10,3.5) {$k-1$};
        \draw[fill=black] (12,3) circle(0.1);
        \node at (12,3.5) {$k$};
        
        \draw[very thick] (8,3) to (10,3);
    \end{tikzpicture}}
    \caption{The associated directed graphs to the real transition matrices of \Cref{fig:l13n5885ttk3}, \Cref{fig:l13n5885ttk401}, \Cref{fig:l13n5885ttk2}, and their curve complexes.}
    \label{fig:l13n5885graph}
\end{figure}

For each $k \geq 2$, the clique polynomial of the corresponding curve complex is $LT_{1,k}=t^{2k}-t^{k+1}-t^k-t^{k-1}+1$. Hence the induced fully-punctured pseudo-Anosov maps in this family of examples have expansion factors $\left| LT_{1,k} \right|$.

We point out that similar train track maps were obtained in \cite{Kin15}. In particular there is some overlap between our first family and \cite[Example 4.2]{Kin15} and between our second family and \cite[Example 4.4]{Kin15}. See also the discussion in the next subsection.

\subsection{Examples for even $\chi(S)$: fibered face theory} \label{subsec:fiberedface}

In this subsection, we will show the sharpness statement in \Cref{thm:sharpthm} again, but this time using the tool of fibered face theory. This will also explain the source of our train track maps in \Cref{subsec:evensharpness}.

We first provide a brief review of fibered face theory, referring to \cite{McM00} for details. Let $M$ be a fibered hyperbolic 3-manifold, with $n=b_1(M) \geq 2$ and let $F \subset H^1(M;\mathbb{R})$ be a fibered face. To every primitive integral point $a \in \cone(F)$, there is a fibration of $M$ over $S^1$ whose monodromy is a pseudo-Anosov mapping class $(S_a, f_a)$.

Let $\Delta_M \in \mathbb{Z}H_1(M)$ be the Alexander polynomial of $M$. The \textit{Newton polytope} of $\Delta_M$ is the convex hull of points in $H_1(M)$ that have nonzero coefficient in $\Delta_M$. The \textit{Alexander norm} of an element $a \in H^1(M)$ is defined to be 
$$||a|| = \sup_{g,h} ~\langle a,g-h \rangle$$
where $g,h$ range over the Newton polytope of $\Delta_M$. 

It is a well-known fact that the Euler characteristic of $S_a$ can be calculated by 
$$|\chi(S_a)|=||a||.$$
In other words, the Thurston norm agrees with the Alexander norm on $\cone(F)$. See, for example, \cite[Theorem 7.1]{McM00}.

Meanwhile, there is another polynomial $\Theta_F \in \mathbb{Z}H_1(M)$, called the \textit{Teichm\"uller polynomial} associated to the fibered face $F$, with the property that the expansion factor of the pseudo-Anosov monodromy $(S_a, f_a)$ corresponding to $a=(a_1,...,a_n)$ can be calculated by 
$$\lambda(f_a)=|\Theta_F(t^{a_1},...,t^{a_n})|.$$
See, for example, \cite[Theorem 5.1]{McM00}.

Now consider the L6a2 link complement, which we denote by $M_1$. It can be computed (for example by using the veering triangulation \texttt{eLMkbcddddedde\_2100}, see \cite{GSS} and \cite{Par21} for the details of such a computation) that the Alexander polynomial of $M_1$ is 
$$\Delta_1(a,b)=b^2+b(a^2-a+1)+a^2$$
under some choice of basis $(a,b)$ for $H_1(M_1)$, and the Teichm\"uller polynomial associated to the fibered cone $C_1=\cone \{a^*+2b^*,-a^*\}$ is
$$\Theta_1(a,b)=b^2-b(a^2+a+1)+a^2$$
Hence the Thurston norm on $C_1$ is given by $-2a+2b$. 

Consider the class $x_1=b^*$. The corresponding pseudo-Anosov monodromy $f_{1,1}$ is defined on a surface $S_1$ with $|\chi(S_1)|=||x_1||=2$. The expansion factor of $f_{1,1}$ is the largest root of $t^2-3t+1$, which is $\mu^2$. Hence $f_{1,1}$ attains equality in \Cref{thm:mainthm}.

Now consider the class $x_k=a^*+(k+1)b^* \in C_1$, for $k \geq 2$. The corresponding pseudo-Anosov monodromy $f_{1,k}$ is defined on a surface $S_k$ with $|\chi(S_k)|=||x_k||=2k$ and its expansion factor is the largest root of $t^{2k+2}-t^{k+1}(t^2+t+1)+t^2=t^2 LT_{1,k}$, which is $|LT_{1,k}|$. Hence $f_{1,k}$ attains equality in \Cref{thm:sharpthm}.

Indeed, the first family of train track maps we described in \Cref{subsec:evensharpness} is computed from $f_{1,k}$. Here we will not demonstrate this computation since the details are rather tedious. It suffices to say that we essentially followed the methodology in \cite{Kin15}; see in particular \cite[Example 4.2]{Kin15}. We also remark that the maps $f_{1,k}$ are considered in \cite{Hir10} as well, even though invariant train tracks were not provided there.

Similarly, consider the L13n5885 link (= the Whitehead sister link = the $(-2,3,8)$-pretzel link) complement, which we denote by $M_2$. It can be computed (for example by using the veering triangulation \texttt{fLLQcbeddeehhbghh\_01110}) that the Alexander polynomial of $M_2$ is 
$$\Delta_2(a,b)=b^2+b(a^2+a+1)+a^2$$
under some choice of basis $(a,b)$ for $H_1(M_2)$, and the Teichm\"uller polynomial associated to the fibered cone $C_2=\cone \{a^*+2b^*,-a^*\}$ is
$$\Theta_2(a,b)=b^2-b(a^2+a+1)+a^2$$
Hence the Thurston norm on $C_2$ is given by $-2a+2b$. 

Consider the class $x_1=b^*$. The corresponding pseudo-Anosov monodromy $f_{2,1}$ is defined on a surface $S_1$ with $|\chi(S_1)|=||x_1||=2$. The expansion factor of $f_{2,1}$ is the largest root of $t^2-3t+1$, which is $\mu^2$. Hence $f_{2,1}$ attains equality in \Cref{thm:mainthm}.

Now consider the class $x_k=a^*+(k+1)b^* \in C_2$, for $k \geq 2$. The corresponding pseudo-Anosov monodromy $f_{2,k}$ is defined on a surface $S_k$ with $|\chi(S_k)|=||x_k||=2k$ and its expansion factor is the largest root of $t^{2k+2}-t^{k+1}(t^2+t+1)+t^2=t^2 LT_{1,k}$, which is $|LT_{1,k}|$.  Hence $f_{2,k}$ attains equality in \Cref{thm:sharpthm}.

The second family of train track maps we described in \Cref{subsec:evensharpness} is computed from $f_{2,k}$. We refer to \cite[Example 4.4]{Kin15} for the methodology of our computation. We also remark that some of the maps $f_{2,k}$ were considered in \cite{AD10} and \cite{KT13}, even though invariant train tracks were not provided in those works.

We summarize all the examples we have discussed in \Cref{tab:sharpness}. In the table, we list a puncture as having singularity type $p$ if it has $p$ prongs, and punctures with the same color belong to the same orbit.

\begin{table}[h]
    \centering
    \footnotesize
    \caption{Examples of maps that attain the lower bound in \Cref{thm:mainthmtext}.}
    \begin{tabular}{|c|c|c|c|c|}
    \hline
    $g$ & $s$ & Range of $k$ & Description of $f$ & Singularity type \\
    \hline
    $0$ & $4$ & - & \multirow{3}{*}{Fiberings of L6a2} & $(\textcolor{red}{1}, \textcolor{red}{1}, \textcolor{red}{1}, \textcolor{blue}{1})$ \\
    $k$ & $2$ & $k \geq 2$, $k \equiv 1,2$ (mod $3$) & & $(\textcolor{red}{3k}, \textcolor{blue}{k})$ \\
    $k-1$ & $4$ & $k \geq 3$, $k \equiv 0$ (mod $3$) & & $(\textcolor{red}{3k}, \textcolor{blue}{\frac{k}{3}}, \textcolor{blue}{\frac{k}{3}}, \textcolor{blue}{\frac{k}{3}})$ \\
    \hline
    $1$ & $2$ & - & \multirow{4}{*}{Fiberings of L13n5885} & $(\textcolor{red}{2}, \textcolor{blue}{2})$ \\
    $k$ & $2$ & $k \geq 4$, $k \equiv 0,1,4$ (mod $5$) & & $(\textcolor{red}{2k+1}, \textcolor{blue}{2k-1})$ \\
    $k-2$ & $6$ & $k \geq 2$, $k \equiv 2$ (mod $5$) & & $(\textcolor{red}{\frac{2k+1}{5}}, \textcolor{red}{\frac{2k+1}{5}}, \textcolor{red}{\frac{2k+1}{5}}, \textcolor{red}{\frac{2k+1}{5}}, \textcolor{red}{\frac{2k+1}{5}}, \textcolor{blue}{2k-1})$ \\
    $k-2$ & $6$ & $k \geq 3$, $k \equiv 3$ (mod $5$) & & $(\textcolor{red}{2k+1}, \textcolor{blue}{\frac{2k-1}{5}}, \textcolor{blue}{\frac{2k-1}{5}}, \textcolor{blue}{\frac{2k-1}{5}}, \textcolor{blue}{\frac{2k-1}{5}}, \textcolor{blue}{\frac{2k-1}{5}})$ \\
    \hline
    \end{tabular}
    \label{tab:sharpness}
\end{table}

Finally, we remark that both $M_1$ and $M_2$ can be obtained by Dehn filling a single fibered 3-manifold $M$, commonly known as the \textit{magic manifold}. Indeed, all the known examples of small expansion factor maps are realized as monodromies of $M$ and its Dehn fillings along a component \cite{KKT13}.

\subsection{Braids and odd $\chi(S)$} \label{subsec:braidsharpness}

Note that the only braid monodromies that appear in \Cref{tab:sharpness} are $f_{1,1}$ and $f_{2,2}$. It can be checked that these are the simplest hyperbolic 3-braid $\sigma_1\sigma_2^{-1}$ (see \cite{Hir10}) and the 5-braid of minimal expansion factor $\sigma_1 \sigma_2 \sigma_3 \sigma_4 \sigma_1 \sigma_2$ (see \cite{HS07}) respectively. These imply that \Cref{thm:braids} is sharp for $n=3,5$.

We do not currently know whether \Cref{thm:braids} is sharp for odd $n \geq 7$, see \Cref{quest:braidsharpness}. Also note that the results of \cite{LT11b} on braids of minimum expansion factor do not provide an answer here, since many of those braids are not fully-punctured.

So far we have only discussed sharpness in the cases when $|\chi(S)|$ is even. This is because the cases when $|\chi(S)|$ is odd are likely not sharp. For example, it is shown in \cite{LT11b} that the 6-braid of minimum expansion factor $\sigma_2 \sigma_1 \sigma_2 \sigma_1 (\sigma_1 \sigma_2 \sigma_3 \sigma_4 \sigma_5)^2$ is fully-punctured and has normalized expansion factor $\lambda_{0,7}^5 \approx 15.14>8$. 
See \Cref{sec:questions} for some discussion on how one might try to sharpen the bound for odd $|\chi(S)|$.

\subsection{Single orbit of punctures} \label{subsec:oneboundary}

In this subsection, we describe some examples that show that the assumption of $f$ having at least two punctures orbits in \Cref{thm:mainthm} is necessary. 

Most of the examples are defined on the once-punctured torus $S_{1,1}$. To describe them, write $S_{1,1}$ as $(\mathbb{R}^2 \backslash \mathbb{Z}^2) / \mathbb{Z}^2$ and notice that an element $A \in SL(2,\mathbb{Z})$ induces a map $f_A$ on $S_{1,1}$. It is a classical fact that if $|\mathrm{tr} A|>2$, then $f_A$ is pseudo-Anosov with expansion factor given by the spectral radius of $A$. Meanwhile, since $|\chi(S_{1,1})|=1$, the normalized expansion factor equals to the expansion factor, and since $S_{1,1}$ only has one puncture, any map defined on it must only have one puncture orbit.

With this understanding, notice that the maps induced by $\pm \begin{bmatrix} 2 & 1 \\ 1 & 1 \end{bmatrix}$, $\pm \begin{bmatrix} 3 & 2 \\ 1 & 1 \end{bmatrix}$, $\pm \begin{bmatrix} 4 & 3 \\ 1 & 1 \end{bmatrix}$, and $\pm \begin{bmatrix} 5 & 4 \\ 1 & 1 \end{bmatrix}$ have expansion factors $\frac{3+\sqrt{5}}{2} \approx 2.62$, $\frac{4+\sqrt{12}}{2} \approx 3.73$, $\frac{5+\sqrt{21}}{2} \approx 4.79$, and $\frac{6+\sqrt{32}}{2} \approx 5.83$ respectively, each of them being strictly less than $\mu^4 \approx 6.85$. This shows that the inequality in \Cref{thm:mainthm} fails if the assumption `with at least two punctures orbits' is removed.

There are (at least) two other examples that can be used to show this. Consider the K12n242 knot (= the (-2,3,7)-pretzel knot) complement. This 3-manifold has a unique fibering, with monodromy $f:S_{5,1} \to S_{5,1}$. The expansion factor of $f$ is given by the largest real root of $t^{10}+t^9-t^7-t^6-t^5-t^4-t^3+t+1$, also known as \textit{Lehmer's number}, which is $\approx 1.18$. Hence the normalized expansion factor of $f$ is $\approx 4.31$, which is strictly less than $\mu^4 \approx 6.85$.

Meanwhile consider the map $f_{1,2}$ defined in \Cref{subsec:fiberedface}. According to \Cref{tab:sharpness}, it is defined on $S_{2,2}$, where one of the punctures is $6$-pronged while the other is $2$-pronged. If we fill in the $2$-pronged puncture, we would still get a fully-punctured pseudo-Anosov map $\overline{f_{1,2}}$ with the same expansion factor but now defined on $S_{2,1}$. Hence the normalized expansion factor of $\overline{f_{1,2}}$ is $\left| LT_{1,2} \right|^3 \approx 5.10$, which is strictly less than $\mu^4 \approx 6.85$.

We summarize these examples in \Cref{tab:oneboundary}.

\begin{table}[h]
    \centering
    \footnotesize
    \caption{Examples of fully-punctured pseudo-Anosov maps $f$ with only one puncture orbit for which the bound in \Cref{thm:mainthm} fails.}
    \begin{tabular}{|c|c|c|c|c|}
    \hline
    $g$ & $s$ & Description of $f$ & Singularity type & $L(S,f)$ \\
    \hline
    $1$ & $1$ & Induced by $\pm \begin{bmatrix} 2 & 1 \\ 1 & 1 \end{bmatrix}$ & $(2)$ & $\frac{3+\sqrt{5}}{2} \approx 2.62$ \\
    $1$ & $1$ & Induced by $\pm \begin{bmatrix} 3 & 2 \\ 1 & 1 \end{bmatrix}$ & $(2)$ & $\frac{4+\sqrt{12}}{2} \approx 3.73$ \\
    $5$ & $1$ & Fibering of K12n242 & $(18)$ & $(\text{Lehmer's number})^9 \approx 4.31$ \\
    $1$ & $1$ & Induced by $\pm \begin{bmatrix} 4 & 3 \\ 1 & 1 \end{bmatrix}$ & $(2)$ & $\frac{5+\sqrt{21}}{2} \approx 4.79$ \\
    $2$ & $1$ & Fill in $2$-pronged puncture of $f_{1,2}$ & (6) & $\left| LT_{1,2} \right|^3 \approx 5.10$ \\ 
    $1$ & $1$ & Induced by $\pm \begin{bmatrix} 5 & 4 \\ 1 & 1 \end{bmatrix}$ & $(2)$ & $\frac{6+\sqrt{32}}{2} \approx 5.83$ \\
    \hline
    \end{tabular}
    \label{tab:oneboundary}
\end{table}

\section{Discussion and further questions} \label{sec:questions}

We first note the following generalization of \Cref{prop:radical}. (Recall \Cref{defn:radicalelement} for the definition of the radical elements $r_c$.)

\begin{prop} \label{prop:generalradical}
Suppose $\tau$ is a train track that fully carries the unstable lamination of a pseudo-Anosov map. Then the radical of the Thurston symplectic form $\omega$ on $\tau$ is given by 
$$\rad(\omega)=\spn\{r_c\}$$
where $c$ ranges over all even-pronged boundary components of $\tau$.

In particular, 
$$\dim \rad(\omega) = \text{\# even-pronged boundary components} - \epsilon$$
for $\epsilon = 0$ or $1$.
\end{prop}
\begin{proof}
Up to puncturing the pseudo-Anosov map $f$ at an orbit of a (non-singular) periodic point, we can arrange for $f$ to have at least two puncture orbits. Correspondingly, we can modify $\tau$ by slitting a small segment of an edge at every punctured point, see \Cref{fig:slit}. 

\begin{figure}
    \centering
    \resizebox{!}{0.6cm}{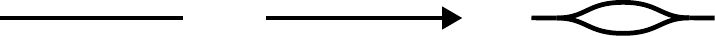}
    \caption{Puncturing at a non-singular point corresponds to slitting a small segment of an edge.}
    \label{fig:slit}
\end{figure}

It is straightforward to check that this slitting operation preserves $\rad(\omega)=\spn\{r_c\}$, hence we can assume that $f$ has at least two puncture orbits.

Then by \Cref{prop:radical}, $\rad(\omega)=\spn\{r_c\}$ is true for at least one train track $\tau'$ fully carrying the unstable lamination of $f$. Now by a theorem of Stallings \cite{Sta83}, $\tau$ and $\tau'$ are related by a sequence of elementary moves, hence by \Cref{lemma:elementarymoveradical}, $\rad(\omega)=\spn\{r_c\}$ is true for $\tau$ as well.

For the statement about the dimension of $\rad(\omega)$, recall that the proof of \Cref{prop:radicalreciprocal} shows that the radical elements $r_c$ have at most one relation.
\end{proof}

\begin{quest} \label{quest:traintrackrad}
Is the equation $\rad(\omega)=\spn\{r_c\}$, where $c$ ranges over all even-pronged boundary components of $\tau$, true for any train track $\tau$?
\end{quest}

Given the invariance of the equation under a wide array of operations on train tracks, it seems reasonable to expect a positive answer. However, if the answer is negative, then \Cref{prop:generalradical} would be an obstruction for train tracks to carry the unstable lamination of a pseudo-Anosov map.

Next, we propose some questions about small expansion factors which one might hope to tackle using the ideas in this paper. 

Firstly, recall that \Cref{tab:sharpness} records some maps that attain the lower bound in \Cref{thm:mainthmtext}. These being the only examples we are aware of, it is a natural question to ask if they are actually the only possible examples.

\begin{quest} \label{quest:evensharpness}
Let $f:S \to S$ be a fully-punctured pseudo-Anosov mapping class with at least two puncture orbits. Suppose $|\chi(S)|=2k \geq 4$ and $L(f)=\left| LT_{1,k}  \right|^{2k}$. Must $f$ be one of the maps listed in \Cref{tab:sharpness}?
\end{quest}

Recall that for the examples in \Cref{tab:sharpness}, we calculated that $L(S,f)=\left| LT_{1,k} \right|^{2k}$ using the techniques described in \cite{McM15}. More specifically, starting with the matrix $f_*^\real$ we get from the train track map, we consider its associated digraph, compute its curve complex $G$, compute the clique polynomial of $G$, and find its largest real root.

Now, McMullen has actually classified the weighted graphs $G$ that could occur in this computation, provided that $L(S,f)=\left| LT_{1,k}  \right|^{2k}$. Hence one can hope to work backwards from that information and classify the matrices $f_*^\real$ that could occur in this computation, then classify the train track maps themselves.

If the answer to \Cref{quest:evensharpness} is `yes', then this means that it is possible to sharpen \Cref{thm:sharpthm} for the other values of $(g,s)$. A natural candidate here would be to replace $|LT_{1,k}|$ by $|LT_{3,k}|$, since in McMullen's analysis in \cite{McM15}, this is the second smallest value for the largest real root of the clique polynomial that could occur in the computation (at least for large enough values of $k$). We remark that the first author has found maps that attain this expansion factor in \cite{Hir10}. 

If this can be done, then one could repeat \Cref{quest:evensharpness} for this second-best bound to ask for the exact range of $(g,s)$ it applies to. Then repeating this scheme as far as possible, one could hope to paint a nice picture to the minimum expansion factor problem when restricted to fully-punctured maps.

With our current understanding of Perron-Frobenius matrices (and pseudo-Anosov maps), one can hope to carry out this program for even $|\chi(S)|$. However, the same question for odd $|\chi(S)|$ would likely require new ideas.

\begin{quest} \label{quest:oddsharpness}
What is the sharpest lower bound for odd $|\chi(S)|$ in \Cref{thm:sharpthm}?
\end{quest}

Indeed, this is down to the fact that in \Cref{thm:McMullen}, the case when $n$ is odd is likely not sharp. This is in turn due to the approach taken in \cite{McM15}, where McMullen simply shows that $\rho(A)^n <8$ is impossible by showing that there would not be an appropriate curve complex $G$ in the computation, via analyzing some small graphs. To understand the sharpest lower bound possible when $n$ is odd, one would likely have to extend the analysis to some slightly larger graphs.

As another potential direction for the ideas in this paper, we turn our attention to the hypothesis that $f$ has at least two puncture orbits in \Cref{thm:mainthm}. Notice that all the examples in \Cref{tab:sharpness} have exactly two puncture orbits. In particular, if the answer to \Cref{quest:evensharpness} is `yes', then it must be possible to sharpen \Cref{thm:mainthm} if we replace the hypothesis by, say, `at least three puncture orbits'. In general, we ask:

\begin{quest} \label{quest:threeboundary}
What is the sharpest lower bound in \Cref{thm:mainthm} if one replaces `at least two puncture orbits' by `at least $q$ puncture orbits' for some number $q \geq 3$?
\end{quest}

As before, the approach to answering \Cref{quest:threeboundary} that we have in mind is to dive into the analysis of \cite{McM15}. The condition on the number of puncture orbits should impose some conditions on the Perron-Frobenius digraph, which might translate to some conditions on the corresponding curve complex $G$. By restricting to graphs satisfying this condition (potentially broadening the analysis if necessary), one can hope to get lower bounds in this scenario.

Another approach would be to generalize the definition of standardly embedded train tracks. For standardly embedded train tracks, the set of punctures is naturally divided into the sets of inner and outer punctures. If there is a type of train track that naturally treats, say, three classes of punctures distinctly, then the resulting dynamics could also give interesting lower bounds.

We remark that whether the sharpest such lower bound is attained is also an interesting question. In \cite{Sun15}, Sun showed that there exists 3-manifolds $M$ with fibered faces $F$ such that the minimum normalized expansion factor on $F$ is attained at an irrational point, i.e. not attained by a pseudo-Anosov map corresponding to a rational point on $F$. If this is the case for the 3-manifolds that produce the sharpest lower bound in \Cref{quest:threeboundary} for some particular value of $q$, then such a lower bound will not be attained. 

\bibliographystyle{alpha}

\bibliography{bib.bib}

\end{document}